\documentclass[11pt]{amsart}
\pdfoutput=1

\usepackage[ansinew]{inputenc}
\usepackage[all]{xy}
\usepackage{pinlabel}
\SelectTips{cm}{}
\usepackage[pagebackref,
  ,pdfborder={0 0 0}
  ,urlcolor=black,hypertexnames=false]{hyperref}
\hypersetup{pdfauthor={Caterina Campagnolo, Francesco Fournier-Facio, Yash Lodha and Marco Moraschini}
  ,pdftitle=Commuting cyclic conjugates}

\usepackage{comment}
\usepackage{amsfonts}
\usepackage{amssymb, amsmath,amsthm, mathtools}
\usepackage{tabularx, hyperref}
\usepackage{enumerate}
\usepackage{mathrsfs}
\usepackage{todonotes}
\usepackage{bbm}

\usepackage{tikz-cd}
\usepackage{calc}

\newtheorem{lemma}{Lemma}[section]
\newtheorem{thm}[lemma]{Theorem}
\newtheorem*{thm*}{Theorem}
\newtheorem{prop}[lemma]{Proposition}
\newtheorem*{prop*}{Proposition}
\newtheorem{cor}[lemma]{Corollary}

\newtheorem{conj}[lemma]{Conjecture}
\theoremstyle{definition}
\newtheorem{defi}[lemma]{Definition}

\newtheorem*{quest*}{Question}

\newtheorem{example}[lemma]{Example}

\newtheorem{rem}[lemma]{Remark}
\newtheorem{question}[lemma]{Question}

\newtheorem{setup}[lemma]{Setup}
\theoremstyle{definition}

\makeatletter
\newcommand\norm{\bBigg@{0.8}}

\makeatother
 
 \newcommand{\indnorm}[2][flex]{\csname #1l\endcsname\|#2%
                                 \csname #1r\endcsname\|\mathclose{}}
                                  \newcommand{\indnorml}[4][flex]{\csname #1l\endcsname\|#2%
                                 \csname #1r\endcsname\|_{#3}^{#4}\mathclose{}}
\newcommand{\sv}[2][flex]{\indnorm[#1]{#2}}

\newcommand{\ifsv}[2][norm]{\!\csname #1l\endcsname\bracevert\!#2\!%
                            \csname #1r\endcsname\bracevert\!}

\DeclareMathOperator{\mcg}{\textup{MCG}}
\DeclareMathOperator{\mcgc}{\textup{MCG}_c}

\DeclareMathOperator{\symp}{\textup{Symp}_c}
\DeclareMathOperator{\sympo}{\textup{Symp}^0_c}
\DeclareMathOperator{\ham}{\textup{Ham}_c}
\DeclareMathOperator{\vol}{\textup{vol}}
\DeclareMathOperator{\homeo}{\textup{Homeo}_c}
\DeclareMathOperator{\diff}{\textup{Diff}_c}
\DeclareMathOperator{\diffo}{\textup{Diff}^0_c}
\DeclareMathOperator{\diffvol}{\textup{Diff}_{c, \vol}}
\DeclareMathOperator{\diffovol}{\textup{Diff}^0_{c, \vol}}
\DeclareMathOperator{\homeoleb}{\textup{Homeo}_{c, \textup{Leb}}}
\DeclareMathOperator{\supp}{supp}
\DeclareMathOperator{\cont}{\textup{Cont}_c}
\DeclareMathOperator{\conto}{\textup{Cont}^0_c}

\DeclareMathOperator{\iet}{\textup{IET}}

\DeclareMathOperator{\comp}{comp}

\DeclareMathOperator{\im}{im}
\DeclareMathOperator{\Aut}{Aut}
\DeclareMathOperator{\id}{id}

\DeclareMathOperator{\X}{\mathfrak{X}}
\DeclareMathOperator{\Xdual}{\mathfrak{X}^{\textup{dual}}}
\DeclareMathOperator{\Xsep}{\mathfrak{X}^{\textup{sep}}}
\DeclareMathOperator{\Xr}{\mathfrak{X}^{\{\R\}}}

\DeclareMathOperator{\heza}{\textup{HEZA}}

\newcommand{\N}{\ensuremath {\mathbb{N}}}
\newcommand{\R} {\ensuremath {\mathbb{R}}}

\newcommand{\Z} {\ensuremath {\mathbb{Z}}}

\renewcommand{\rho}{\varrho}
\def\phi{\varphi}

\def\actson{\curvearrowright}

\usepackage{color}
\usepackage{pdfcolmk}

\long\def\forget#1{}

\def\longrightarrow{\rightarrow}
\def\longmapsto{\mapsto}

\def\emptyset{\varnothing}

\makeatletter
\@namedef{subjclassname@2020}{%
  \textup{2020} Mathematics Subject Classification}
\makeatother

\begin{document}

\title[]{An algebraic criterion for the vanishing of bounded cohomology}

\author[]{Caterina Campagnolo}
\address{Departamento de Matem\'aticas, Universidad Aut\'onoma de Madrid, Spain}
\email{caterina.campagnolo@uam.es}

\author[]{Francesco Fournier-Facio}
\address{Department of Pure Mathematics and Mathematical Statistics, University of Cambridge, UK}
\email{ff373@cam.ac.uk}

\author[]{Yash Lodha}
\address{Department of Mathematics, Purdue University, USA}
\email{yashlodha763@gmail.com}

\author[]{Marco Moraschini}
\address{Dipartimento di Matematica, Universit{\`a} di Bologna, Italy}
\email{marco.moraschini2@unibo.it}

\thanks{}





\date{\today.\ \copyright{ C.~Campagnolo, F.~Fournier-Facio, Y.~Lodha and M.~Moraschini}.
}

\begin{abstract}
We prove the vanishing of bounded cohomology with separable dual coefficients for many groups of interest in geometry, dynamics, and algebra. These include compactly supported structure-preserving diffeomorphism groups of certain manifolds; the group of interval exchange transformations of the half line; piecewise linear and piecewise projective groups of the line, giving strong answers to questions of Calegari and Navas; direct limit linear groups of relevance in algebraic $K$-theory, thereby answering a question by Kastenholz and Sroka and a question of two of the authors and L{\"o}h; and certain subgroups of big mapping class groups, such as the stable braid group and the stable mapping class group, proving a conjecture of Bowden. Moreover, we prove that in the recently introduced framework of enumerated groups, the generic group has vanishing bounded cohomology with separable dual coefficients. At the heart of our approach is an elementary algebraic criterion called the \emph{commuting cyclic conjugates condition} that is readily verifiable for the aforementioned large classes of groups.
\end{abstract}
\maketitle

\setcounter{tocdepth}{1}

\section{Introduction}

On the flip side of its wide applicability, bounded cohomology of groups is infamously hard to compute. One of the few strong and general results is that all amenable groups are \emph{boundedly acyclic}, namely their bounded cohomology $H^n_b(-; \mathbb{R})$ vanishes for all $n \geq 1$ \cite{Johnson}. This is not a characterization of amenability: Matsumoto and Morita proved that the group of compactly supported homeomorphisms of $\mathbb{R}^n$ is boundedly acyclic \cite{Matsu-Mor}.
Indeed, Johnson \cite{Johnson} provided the following criterion for amenability by appealing to a larger class of coefficients: $\Gamma$ is an amenable group if and only if $H^n_b(\Gamma; E) = 0$ for every \emph{dual} Banach $\Gamma$-module $E$ and all $n\geq 1$.
Beyond amenable groups, the strongest vanishing result known to date was recently obtained by Monod, who showed that for every group $\Gamma$ the wreath product $\Gamma \wr \Z$ has vanishing bounded cohomology with all \emph{separable} dual coefficients~\cite{monod:thompson}. 

The subject has a history of the deployment of certain topologically flavoured \emph{displacement techniques}, such as those emerging in the work of Matsumoto and Morita ~\cite{fisher1960group, mather1971vanishing, Matsu-Mor}. Such techniques also led to the proof of bounded acyclicity for certain homeomorphism groups~\cite{fflm2}, and diffeomorphism groups~\cite{monodnariman}. They also played a crucial role in the study of the second bounded cohomology and stable commutator length \cite{kotschick}, \cite[Example~3.66]{calegari_scl} (which are closely related via Bavard duality \cite{bavard}).
More recently the vanishing of the second bounded cohomology has been studied via the notion of commuting conjugates~\cite{fflodha}.
However, as remarked by Monod~\cite[p.~664]{monod:thompson}, there is no reason why all groups studied with these techniques~\cite{kotschick, fflodha} should exhibit vanishing beyond degree $2$ or beyond trivial coefficients.

In this paper, we take an algebraic approach, as captured by the following definition. 

\begin{defi}[commuting cyclic conjugates]\label{intro:def:comm:cycl:conj}
A group $\Gamma$ has \emph{commuting cyclic conjugates} if for every finitely generated subgroup $H \leq \Gamma$ there exist $t \in \Gamma$ and $n \in \N_{\geq 2} \cup \{\infty\}$ such that: 
\begin{enumerate}
\item $[H, t^pHt^{-p}] = 1$ for all $1\leq p< n$;
\item $[H, t^n] = 1$. 
\end{enumerate}
Here we read that $t^\infty = 1$.
\end{defi}

We state the property algebraically; however as we mentioned before the motivation comes from displacement properties satisfied by many groups of dynamical origin (Lemma \ref{lem:ccc:dynamical}). Observe that Definition \ref{intro:def:comm:cycl:conj} only applies to groups that cannot be finitely generated (unless they are abelian); however we will see that many finitely generated groups (even finitely presented and type $F_\infty$) contain subgroups with commuting cyclic conjugates, which provides a starting point for further computations.

Our primary goal is to show that this notion has powerful consequences for the vanishing of bounded cohomology for a large family of coefficients, in particular the following:

\begin{defi}[$\Xsep$-bounded acyclicity]
\label{intro:def:sepbac}
A group $\Gamma$ is said to be \emph{$\Xsep$-boundedly acyclic} if $H^n_b(\Gamma; E) = 0$ for all $n \geq 1$ and every separable dual Banach $\Gamma$-module $E$.
\end{defi}

$\Xsep$-bounded acyclicity is a very strong property with powerful consequences in rigidity theory: See Section \ref{s:rigidity} for a discussion. We can now state our main result:

\begin{thm}[Main Theorem]
\label{intro:thm:ccc}
Let $\Gamma$ be a group with commuting cyclic conjugates. Then $\Gamma$ is $\Xsep$-boundedly acyclic.
\end{thm}

Theorem \ref{intro:thm:ccc} provides a very general tool that
applies to many groups of interest in geometry, dynamics, and algebra, as shown in the following examples; they include all the groups whose stable commutator length and second bounded cohomology were previously studied with displacement techniques~\cite{kotschick, burago2008conjugation, fflodha}.
Our criterion applies immediately to various families of non-finitely generated groups. We also apply it to several important classes of finitely generated and finitely presented groups, via coamenability: Indeed $\Xsep$-bounded acyclicity passes to co-amenable supergroups (see Theorem \ref{thm:co-amenable}).

In a companion paper~\cite{companion}, we survey all the displacement techniques available in the literature for the study of bounded cohomology and stable commutator length and we discuss the implications between them. In particular, to our knowledge having commuting cyclic conjugates turns out to be the weakest condition that ensures $\Xsep$-bounded acyclicity.

\subsection{Subgroups of big mapping class groups}
The first source of applications are \emph{stable} mapping class groups.

\subsubsection{The stable mapping class group $\Gamma_\infty$}\label{subsect:stable:mapping:class:group}
Since we can embed a genus $g$ surface with $k$ boundary components $\Sigma_{g, k}$ into $\Sigma_{g+1, k}$ by attaching a two-holed torus along a single boundary component, every diffeomorphism of $\Sigma_{g, k}$ fixing each boundary component extends to the larger surface.\ This leads to a group inclusion $\Gamma_g^{k} \leq \Gamma_{g+1}^k$, where $\Gamma_g^k$ denotes the group of isotopy classes of diffeomorphisms of $\Sigma_{g, k}$ with compact support in the interior (i.e.\ those that fix the boundary pointwise). The direct union of these groups is denoted by $\Gamma_\infty^{k-1}$; it is known as the \emph{stable mapping class group} for $k = 1$, and denoted by either $\Gamma_\infty^0$ or $\Gamma_\infty$. The rational cohomology ring of $\Gamma_\infty$ was conjectured by Mumford and then computed by Madsen and Weiss (also in the integral case)~\cite{madsen2007stable}.

Moreover, Kotschick first proved that the stable commutator length of $\Gamma_\infty$ vanishes~\cite[Theorem 3.1]{kotschick} and then Bowden proved that $H^2_b(\Gamma_\infty; \R)$ is zero~\cite[Prop. 4.7]{bowden}; this together with the fact that stable classes in stable groups tend to unbounded classes led Bowden to conjecture that the bounded cohomology of the stable mapping class group vanishes in all positive degrees:

\begin{conj}[{Bowden~\cite{bowden}}]
The stable mapping class group $\Gamma_\infty$ is boundedly acyclic.
\end{conj}

We confirm here the conjecture in a strong way:

\begin{cor}[Corollary \ref{cor:stablemcg}]
\label{corintro:stablemcg}
The stable mapping class group $\Gamma_\infty$ is $\Xsep$-boundedly acyclic.
\end{cor}

This contrasts with the case $k \geq 2$ since Baykur, Korkmaz and Monden showed that the stable commutator length of $\Gamma_\infty^{k-1}$ does not vanish~\cite[Section 3.2]{baykur2013sections}. This is also in contrast with other compactly supported subgroups of big mapping class groups (see Section \ref{subsec:big:mcg}).

\subsubsection{Stable braid group and the bounded cohomology of the braided Ptolemy--Thompson group}

The \emph{stable braid group} is the direct union of the braid groups under the inclusion $B_n \to B_{n+1}$ given by adding a strand. Kotschick showed that the stable commutator length of such a group vanishes~\cite[Theorem 3.5]{kotschick}. Here we prove the following:

\begin{cor}[Corollary \ref{cor:braid}]
\label{corintro:braid}
    The stable braid group $B_\infty$ is $\Xsep$-boundedly acyclic. The same is true for its commutator subgroup $[B_\infty, B_\infty]$.
\end{cor}

Recall that Thompson's group $T$ can be identified with the asymptotically rigid mapping classes of the two-dimensional closed disc $D$~\cite{kimT}. As an immediate consequence of Corollary~\ref{corintro:braid}, we can fully compute the bounded cohomology of the \emph{braided Ptolemy--Thompson group} $T^\ast$~\cite{funar2008braided} with trivial real coefficients. Recall that the group $T^\ast$ consists in all asymptotically rigid mapping classes of $D^\ast$, where $D^\ast$ denotes $D$ with a countable infinite locally finite collection of punctures~\cite{kimT}.

\begin{cor}[Corollary \ref{cor:ptolemy}]
\label{corintro:ptolemy}
    The bounded cohomology ring (with the cup product structure) of $T^\ast$ with trivial real coefficients is the following:
    \[
    H_b^\bullet(T^\ast; \R) \cong \R[x],
    \]
    where the degree of $x$ is $2$. Under this isomorphism, $x$ corresponds to the Euler class of the defining action on the circle of the quotient $T$.
\end{cor}

This approach can produce analogous results for several other groups defined this way, such as the braided Higman--Thompson groups~\cite{funarmore} and other related groups~\cite[Section 2]{funarlike}.

\subsubsection{Braided Thompson groups}

The braided Thompson group $bV$ was introduced independently by Brin \cite{bV:brin} and Dehornoy \cite{bV:dehornoy}, and can be realized as a dense subgroup of the mapping class group of a disc minus a Cantor set \cite{bV:dense}. Analogously to the braided Ptolemy--Thompson group, there exists an epimorphism $bV \to V$ whose kernel $bP$ is a direct union of pure braid groups. However this directed union is not the standard one, and corresponds to splitting strands, rather than adding unbraided strands. In fact, both $bV$ and $bP$ were shown to admit an infinite-dimensional space of homogeneous quasimorphisms \cite{bV:qm}.

On the other hand, $bV$ contains a subgroup $\widehat{bV}$, also defined by Brin \cite{bV:brin}, which can be seen as a braided version of a point stabilizer in $V$. The commutator subgroup of $\widehat{bV}$ was shown to have commuting conjugates \cite{bV:qm}, and it turns out that it also has commuting cyclic conjugates. So we obtain:

\begin{cor}[Corollary \ref{cor:bV}]
\label{corintro:bV}
The Brin group $\widehat{bV}$ is $\Xsep$-boundedly acyclic.
\end{cor}

\subsection{Direct limit groups}

Many groups involved in constructions of algebraic $K$-theory are $\Xsep$-boundedly acyclic. Such groups also arise naturally as direct limits of linear groups, and other groups of this form can be treated with the same approach.

\subsubsection{Linear groups and $K$-theory}
Recall that the higher $K$-theory groups of a ring $R$ are defined as the higher homotopy groups of $\textup{BGL}(R)^+$, i.e. Quillen's plus construction applied to the classifying space of the direct limit general linear group $\textup{GL}(R)$ \cite{kbook}. By construction, $K_1(R)$ is the abelianization of $\textup{GL}(R)$, explicitly the quotient $\textup{GL}(R)/\textup{E}(R)$, where $\textup{E}(R)$ is the commutator subgroup of $\textup{GL}(R)$, generated by elementary matrices. Moreover, Milnor's definition of $K_2$ is given in terms of the kernel of the homomorphism $\textup{St}(R) \to \textup{E}(R)$ from the Steinberg group $\textup{St}(R)$ onto $\textup{E}(R)$ \cite{milnor1971introduction}. Finally, Gersten proved that $K_3(R) = H_3(\textup{St}(R), \Z)$ \cite{gersten}.

We show that all the groups appearing in this context are $\Xsep$-boundedly acyclic. In particular, the case of $\textup{GL}(R)$ (in the special case in which $R$ is the suspension of a ring) provides a strong positive answer to the question about its bounded acyclicity~\cite[Question 4.8]{fflm2}:

\begin{cor}[Corollary \ref{cor:linear}]
\label{corintro:linear}

Let $R$ be a ring with identity. Then the following groups are $\Xsep$-boundedly acyclic:
\begin{enumerate}
    \item The direct limit general linear group $\textup{GL}(R)$;
    \item The direct limit special linear group $\textup{SL}(R)$;
    \item The direct limit symplectic group $\textup{Sp}(R)$;
    \item The direct limit elementary group $\textup{E}(R)$;
    \item The Steinberg group $\textup{St}(R)$.
\end{enumerate}
\end{cor}

Note that the previous result has consequences also for the (acyclic) \emph{Volodin space} $X(R)$. Indeed, since the fundamental group of $X(R)$ is $\textup{St}(R)$, Gromov's mapping theorem with coefficients~\cite{vbc, clara:book, moraschini_raptis_2023} implies that the Volodin space is also boundedly acyclic for all coefficients in $\Xsep$. This situation witnesses the fact that Quillen's plus construction is \emph{intrinsically unbounded}, since each of the spaces appearing in the homotopy fibration $X(R) \to \textup{BGL}(R) \to \textup{BGL}(R)^+$ is boundedly acyclic~\cite{kbook, Raptis-acyclic, moraschini_raptis_2023}.

\subsubsection{Automorphisms of products} Similar to the previous cases, we can also stabilize the free group on $n$ generators as follows: We consider the inclusion $F_n \to F_{n+1} = F_n \ast \Z$ and then we consider the direct limit
\[
\textup{Aut}_{\infty}(F_\infty) = \bigcup_{n \in \N} \textup{Aut}(F_n) \leq \textup{Aut}(F_\infty).
\]
This is an instance of a more general construction for free products of arbitrary groups, which leads to the definition of the groups $\textup{Aut}_{\infty}(\Gamma^{* \infty})$, and can also be adapted for direct products and direct sums, yielding the groups $\textup{Aut}_{\infty}(\Gamma^{\times \infty})$ and $\textup{Aut}_{\infty}(\Gamma^{\oplus \infty})$ \cite{stability:auto}.\ We show that all groups arising in this way are $\Xsep$-boundedly acyclic. This extends Kotschick's result on the vanishing of the stable commutator length of $\textup{Aut}_{\infty}(F)$~\cite{kotschick}.

\begin{cor}[Corollary \ref{cor:auto}]
\label{corintro:auto}
    Let $\Gamma$ be any group.\ Then the groups $\textup{Aut}_{\infty}(\Gamma^{* \infty}), \textup{Aut}_{\infty}(\Gamma^{\times \infty})$ and $\textup{Aut}_{\infty}(\Gamma^{\oplus \infty})$ are $\Xsep$-boundedly acyclic. In particular, $\textup{Aut}_{\infty}(F_\infty)$ is $\Xsep$-boundedly acyclic.
\end{cor}

\subsubsection{Cremona groups}
Let $K$ be a field, and consider the Cremona groups $\mathrm{Cr}_n(K)=\mathrm{Aut}(K(x_1, \ldots, x_n)|K)$, which are the automorphism groups of the field of rational functions~$K(x_1, \ldots, x_n)$ over $K$. These groups can also be viewed as the birational transformation groups of the projective spaces $\mathbb{P}^n_K$ over $K$. They have been studied for centuries in classical algebraic geometry, and more recently they were approached via geometric group theory \cite{cremona:hyperbolic, cremona:survey}. Standout results include acylindrical hyperbolicity in the rank $2$ case \cite{cremona:ah1, cremona:ah2, DGO} and actions on CAT(0) cube complexes for all ranks \cite{cremona:cat}. As above, we can stabilize these groups along the natural inclusions $\mathrm{Cr}_n(K)\hookrightarrow \mathrm{Cr}_{n+1}(K)$, fixing the coordinate $x_{n+1}$.\ We consider their direct limit 
\[\mathrm{Cr}_{\infty}(K)=\bigcup_{n\in\mathbb{N}}\mathrm{Cr}_n(K).\]
We show that this group is $\Xsep$-boundedly acyclic:
\begin{cor}[Corollary \ref{cor:cremona}]
\label{corintro:cremona}
   The group $\mathrm{Cr}_{\infty}(K)$ is $\Xsep$-boundedly acyclic.
\end{cor}

\subsection{Stability}
The previous results have interesting implications for the theory of \emph{homological stability} (we refer the reader to work of Randal-Williams and Wahl~\cite{stability:auto} and the references therein). 
A family of groups $\Gamma_1 \hookrightarrow \Gamma_2 \hookrightarrow \cdots$ satisfies homological stability if, for all $i \geq 1$, the maps $H_i(\Gamma_n) \to H_i(\Gamma_{n+1})$ are eventually isomorphisms (this is usually meant with respect to trivial integral coefficients). Homological stability has been an important tool for studying the homology of certain families of groups. Typically the approach splits into three steps: Prove homological stability, then compute the homology of the colimit group $\textup{colim}_{n \to \infty} \Gamma_n$ (which in this setting is a direct limit), and use it to compute the homology of the groups $\Gamma_n$, at least for $n$ large enough.

The following families of groups encountered so far satisfy homological stability \cite{stability:auto, cremona}: The mapping class groups with one boundary component $\{ \Gamma^1_g \}_{g \geq 1}$; the braid groups $\{ B_n \}_{n \geq 1}$; the linear groups $\{ \textup{GL}_n(R) \}_{n \geq 1}$ and some subgroups thereof; the automorphism groups $\{ \Aut(\Gamma^{*n})\}_{n \geq 1}$ and $\{ \Aut(\Gamma^{\times n})\}_{n \geq 1}$; the Cremona groups $\{\mathrm{Cr}_n(K)\}_{n\geq 1}$.

These results therefore show that the colimit groups coming from homologically stable families tend to be $\Xsep$-boundedly acyclic (in fact, it was suggested to us by Thorben Kastenholz that Theorem \ref{intro:thm:ccc} might apply every time that the homologically stable system has an underlying $E_\infty$ structure, see \cite{stability:auto}). This provides an obstacle for adapting homological stability arguments to the setting of bounded cohomology, and answers a question of Kastenholz--Sroka~\cite[Question 1.4]{kastenholz-sroka} about the scope of this approach, for the homologically stable families above. Let us however remark that bounded cohomological stability has been proved and applied with success in various contexts, where the bounded cohomology of the colimit groups plays no role \cite{monod04, monod:semiseparable, delacruzH1, delacruzH2, cdcm, kastenholz, kastenholz-sroka}.

\subsection{Transformation groups}
Monod and Nariman recently showed that the compactly supported diffeomorphism and homeomorphism groups of certain \emph{portable} manifolds, in the sense of Burago--Ivanov--Polterovich~\cite{burago2008conjugation} are $\Xsep$-boundedly acyclic~\cite{monodnariman}. More precisely, they considered manifolds that are products of a closed manifold $M$ with $\R^n$ for some $n \geq 1$. Their argument is based on the observation that every bounded set can be displaced infinitely many times within a bounded set; thus there is no hope to extend it to the case of compactly supported symplectic or volume-preserving diffeomorphism groups. We show here how our algebraic criterion can still apply in this rigid setting.

Let $(M, \omega)$ be a symplectic manifold and let $\symp(M, \omega)$ be  the group of compactly supported diffeomorphisms that preserve the $2$-form $\omega$. We denote by $\sympo(M, \omega)$ the path-connected component of the identity.\ An important subgroup of $\sympo(M, \omega)$ is the group of compactly supported Hamiltonian diffeomorphisms that is denoted by $\ham(M, \omega)$.\ Kotschick proved that the stable commutator length of $\sympo(\R^{2m}, \omega)$ and $\ham(\R^{2m}, \omega)$ vanishes, as well as the one of the group of compactly supported volume-preserving diffeomorphisms $\diffvol(\R^n)$ for $n \geq 2$~\cite[Theorem 4.2]{kotschick}.

These groups are part of the \emph{classical diffeomorphism groups} of $\R^n$ \cite{banyaga}. To complete the picture, we also consider the group of compactly supported contactomorphisms $\cont(\R^{2k+1}, \xi)$. We are able to show that all classical diffeomorphism groups of $\R^n$ are $\Xsep$-boundedly acyclic, as are their identity components, in all regularities.

\begin{cor}[Corollaries \ref{cor:tg:symp}, \ref{cor:tg:volume} and \ref{cor:contact}]
\label{corintro:transformation}
    Let $n \geq 2$, and let $\R^n$ be endowed with its standard volume form, with its standard symplectic form $\omega$ if $n = 2m$ is even, and with its standard contact structure $\xi$ if $n = 2k+1$ is odd. Then the following groups are $\Xsep$-boundedly acyclic:
    \begin{enumerate}
        \item $\diffvol(\R^n)$;
        \item $\symp(\R^{2m}, \omega)$ and $\ham(\R^{2m}, \omega)$;
        \item $\cont(\R^{2k+1}, \xi)$.
    \end{enumerate}
    Moreover, the same holds for the identity components of these groups, for all regularities, and for commutator subgroups.
\end{cor}

In fact, we prove that the conclusion of Corollary \ref{corintro:transformation} holds for all compactly supported groups of homeomorphisms that contain one of the groups in the statement: This is the advantage of proving $\Xsep$-bounded acyclicity using the commuting cyclic conjugates condition. So for example it also holds for the group $\homeoleb(\R^n)$ of Lebesgue-measure-preserving homeomorphisms of $\R^n$, or for the group of compactly supported Hamiltonian homeomorphisms of $\R^{2m}$ \cite{hamhomeo1, hamhomeo2}. These statements can be generalized from $\R^n$ to manifolds of the form $\R^n \times M$, where $M$ is closed (Corollary \ref{cor:tg:products}). Notice that the statement for commutator subgroups is a strict generalization, since these groups are not perfect \cite[Chapters 4 and 5]{banyaga} (for contactomorphisms, this seems to be an open question \cite[Remark after Theorem 6.3.7]{banyaga}).

The ordinary homology of diffeomorphism groups of $\R^n$ has often been studied in the past. Its study is typically split into two parts: On the one hand, one can investigate the homology of the \emph{topological group} $\diffvol(\R^n)^\tau$, with its usual $C^\infty$-topology.\ This is the same as the homology of the topological group $\textup{Diff}_{\vol}(\mathbb{D}^n, \partial \mathbb{D}^n)^\tau$ \cite{mcduff:local2}, which in turn is the same as the homology of the topological group $\textup{Diff}(\mathbb{D}^n, \partial \mathbb{D}^n)^\tau$, as follows by an application of Banyaga's version of Moser's trick \cite{moser, banyaga:moser}.\ The latter group has a very rich and mysterious homology theory (when $n \geq 4$) \cite{rw:survey}. On the other hand, there is the study of the \emph{local homology} of $\diffvol(\R^n)$, i.e.\ the homology of the homotopy fiber of the map $B\diffvol(\R^n) \to B\diffvol(\R^n)^\tau$.\ This has been extensively studied, especially by McDuff \cite{mcduff:local1, mcduff:local2, mcduff:local3, mcduff:canonical, hurder}.

The bounded cohomology of transformation groups as in Corollary \ref{corintro:transformation} is well-studied in the cases of compact (or more generally finite-volume) manifolds \cite{gambaudoghys, polterovich:qm, py, polterovich:qm2, brandemarci2, nitsche, brandemarci1, qm:sphere, brandemarci3}. In those cases it is generally non vanishing, even in high degree, in stark contrast with our result. Note that the inclusion $\textup{Diff}_{\vol}(\mathbb{D}^n, \partial\mathbb{D}^n) \to \diffvol(\R^n)$ induces an isomorphism in ordinary cohomology \cite{mcduff:local2}. However, this is not the case in bounded cohomology, at least for $n = 2$ and degrees $2$ and $3$ \cite{bargeghys, kimura_2023} (as far as we know, nothing is known about the bounded cohomology of $\textup{Diff}_{\vol}(\mathbb{D}^n, \partial\mathbb{D}^n)$ for $n \geq 3$ \cite[Open Problem, p. 230]{burago2008conjugation}).

\medskip

If we consider the full group of compactly supported diffeomorphisms, we can in fact prove $\Xsep$-bounded acyclicity for all portable manifolds:

\begin{cor}[Corollary \ref{cor:portable}]
\label{corintro:portable}
    Let $M$ be a portable manifold. Then, $\diffo(M)$ and $\diff(M)$ are $\Xsep$-boundedly acyclic.
\end{cor}

Examples of portable manifolds include (open) manifolds that are the product of $\R^n$ and a closed manifold $M$.\ Moreover, an (open) manifold $M$ is portable if it admits an exhausting Morse function $f$ such that $f$ has only finitely many critical points and they all have index strictly less than $\frac{1}{2}\dim(M)$.\ For instance, the latter property is satisfied by every $3$-dimensional (open) handlebody \cite{burago2008conjugation}. In fact, our result on contactomorphisms of $\R^{2k+1}$ is part of a more general result on contact portable manifolds \cite{contact:portable}; however the only other examples of contact portable manifolds seem to be contact handlebodies.

\subsection{One-dimensional actions}

\subsubsection{Piecewise linear groups of the interval, and piecewise projective groups of the line}
These groups have a rich subgroup structure: They both contain the famous Thompson's group $F$ \cite{cfp_96}, and the latter contains Monod's groups $H(A)$ where $A$ is a dense subring of $\R$ \cite{monod:pp1, monod:pp2}, as well as the Lodha--Moore group $G_0$ \cite{LodhaMoore}.
The groups $H(A)$ where $A$ is a dense subring of $\R$, and $G_0$, are examples of non-amenable groups without non-abelian free subgroups (and the Lodha--Moore group $G_0$ is a finitely presented and type $F_{\infty}$ example).
These groups have the striking property that they do not contain non-abelian free subgroups, yet do not satisfy a law~\cite{BrinSquier}.\ Their homology has been computed in various cases, and it is typically non-zero in infinitely many degrees \cite{cohoF, cohoT, greenberg, stein, brown:cohoF}.

As a consequence of our main theorem, we obtain the following corollary:

\begin{cor}[Corollaries \ref{cor:bsupp}, \ref{cor:coherent} and \ref{cor:PL}]
\label{corintro:PL}
Let $\Gamma$ be a group that lies in one of the following families: 
\begin{enumerate}
\item A group of piecewise linear homeomorphisms of the interval, or a group of piecewise projective homeomorphisms of the line;
\item A group of compactly supported homeomorphisms of the line acting without a global fixpoint;
\item A chain group of homeomorphisms of the line~\cite{chain};
\item A group that admits a coherent action on the line~\cite{coherent}.
\end{enumerate}

Then $\Gamma$ is $\Xsep$-boundedly acyclic.
\end{cor}

This corollary recovers not only the $\Xsep$-bounded acyclicity of Thompson's group $F$~\cite{monod:thompson}, but also generalizes the result to all subgroups thereof.\ The subgroup structure of Thompson's group $F$ is wild \cite{subgroups1, subgroups2, subgroups3}, and so very few general statements are known to be valid for all subgroups. Among these are the Brin--Squier Theorem stating that they have no free subgroups~\cite{BrinSquier}; the result of Calegari stating that they have vanishing stable commutator length~\cite{scl_pl}; and that they have vanishing second bounded cohomology~\cite{fflodha}. Each result is a strengthening of the previous one, and Corollary~\ref{corintro:PL} extends all the aforementioned results.

The same is true in the case of piecewise projective groups. In particular, we can prove that the group $G_0$~\cite{LodhaMoore} has the following surprising property:

\begin{cor}[Corollary \ref{cor:alternative}]
\label{corintro:alternative}
There exists a group of type $F_\infty$ that is non-amenable, but all of whose subgroups are $\Xsep$-boundedly acyclic.
\end{cor}

Calegari proposed a homological version of the von Neumann Day problem~\cite{scl_pl}, stating that if $\Gamma$ is a (finitely presented, torsion-free) group such that every subgroup has vanishing stable commutator length, then $\Gamma$ is amenable. This has been recently proved to fail, even for vanishing of second bounded cohomology \cite{fflodha}. Corollary \ref{cor:alternative} shows that this fails also when including higher degrees, and non-trivial coefficients.

\subsubsection{Highly transitive actions}
\label{intro:highlytrans}

Bounded acyclicity results can also be used as a basis for more complicated computations. Indeed, if a group admits a highly transitive action on a set with (uniformly) boundedly acyclic stabilizers, this can lead to full computations of bounded cohomology rings. This approach was carried out with success in various contexts \cite{fflm2, monodnariman, fflodha, konstantin}, and here we can push these methods further. First, we generalize the computation of the bounded cohomology of Thompson's group $T$ to many more groups of homeomorphisms of the circle:

\begin{cor}[Corollary \ref{cor:circleaction}]
\label{corintro:circleaction}

Let $\Gamma$ be a group of orientation-preserving piecewise linear homeomorphisms of the circle, or piecewise projective homeomorphisms of the projective line. Suppose that there exists an infinite orbit $S$ of $\Gamma$ such that, for all $n \geq 1$, $\Gamma$ acts transitively on positively oriented $n$-tuples in $S$. Then there exists an isomorphism
\[H^\bullet_b(\Gamma; \R) \cong \R[x],\]
where the degree of $x$ is $2$, and $x$ corresponds to the Euler class of the defining action on the circle of $\Gamma$.
\end{cor}

This applies to Thompson's group $T$, but also to Stein--Thompson groups of the form $T_{2, n_2, \ldots, n_k}$ \cite{stein}, to the golden ratio Thompson's group $T_\tau$ \cite{cleary, ttau}, and to the piecewise projective group $S$~\cite{lodhaS}. See Example \ref{ex:circleaction} for more details.

In this work we push this method further to encompass also group actions on the line, and obtain the following:

\begin{cor}[Corollary \ref{cor:lineaction}]
\label{corintro:lineaction}

Let $\Gamma$ be a group of orientation-preserving homeomorphisms of the line, and let $S$ be an infinite orbit of $\Gamma$. Suppose that for all $n \geq 1$, the action of $\Gamma$ on linearly ordered $n$-tuples in $S$ is transitive. Suppose moreover that, for $s \in S$, the stabilizer $\Gamma_s$ admits a co-amenable subgroup $\Gamma_s^0$, such that every finitely generated subgroup of $\Gamma_s^0$ is isomorphic to a group of piecewise linear homeomorphisms of the interval. Then $\Gamma$ is boundedly acyclic.
\end{cor}

This corollary applies to several constructions of finitely generated simple left orderable groups:\ the groups $G_\rho$ of the third author and Hyde \cite{grho}, the groups $T(\varphi)$ of Matte Bon and Triestino \cite{tphi}, and the groups $T(\varphi, \sigma)$ of Le Boudec and Matte Bon \cite{tphisigma}.

In his 2018 ICM address, Navas asked whether there exist finitely generated left orderable groups which do not surject onto $\Z$, and with vanishing second bounded cohomology with trivial coefficients \cite[Question 8]{navas}. The groups $G_\rho$ were proved to have this property \cite{fflodha}, and now it is possible to strengthen this to include vanishing in higher degrees, thus giving a strong positive answer to Navas's question.

\medskip

Since the first version of this paper, this approach has been used successfully in more contexts, always using Theorem \ref{intro:thm:ccc} as a starting point: see \cite{WWZZ, ffmn, zhou}.

\subsubsection{Interval exchange transformations}

Another group of interest arising from one-dimensional dynamics is the group of interval exchange transformations $\iet([0, 1))$.\ This is defined as the group of right-continuous bijections of $[0, 1)$ given by cutting the interval into finitely many pieces and rearranging by translations.\ The group, as well as many generalizations, has been widely studied \cite{iet:free, extensive, iet:solvable, iet:up, RET}.\ The two main open problems about it are whether it is amenable (known as de Cornulier's conjecture \cite{iet:cornulier}), or contains non-abelian free subgroups (known as Katok's conjecture \cite{iet:free}).\ Therefore it is natural to wonder whether $\iet([0, 1))$ has vanishing bounded cohomology (Question \ref{q:iet}):\ This can be seen as an analogue to Grigorchuk's question about bounded acyclicity of Thompson's group $F$ \cite{grigorchuk}, which was answered by Monod \cite{monod:thompson}.

We consider instead the group of interval exchange transformations of the half line interval $[0, \infty)$, which is defined the same way:\ In particular we allow only finitely many intervals to be exchanged.\ Note that this group is amenable, respectively contains non-abelian free subgroups, if and only if $\iet([0, 1))$ has the same property (Lemma \ref{lem:iet:dirun}).

\begin{cor}[Corollary \ref{cor:iet}]
    \label{corintro:iet}

    The group $\iet([0, \infty))$ is $\Xsep$-boundedly acyclic, and so are all of its derived subgroups.
\end{cor}

\subsection{Algebraically closed and generic groups}

Our next two applications of Theorem~\ref{intro:thm:ccc} come from combinatorial group theory. A group $\Gamma$ is said to be \emph{algebraically closed} if every finite system of equations with constants in $\Gamma$, which has a solution in some overgroup of $\Gamma$, has a solution in $\Gamma$. This class of groups was introduced by Scott \cite{algclosed:scott}, and it turned out to coincide with the class of \emph{existentially closed groups}, where one includes also inequations in the definition. Baumslag--Dyer--Heller proved that algebraically closed groups are mitotic \cite{BDH}, from which it follows that they are acyclic \cite{BDH} and boundedly acyclic \cite{Loeh:dim}. We extend this to include non-trivial coefficients:

\begin{cor}[Corollary \ref{cor:algclosed}]
\label{corintro:algclosed}
    Algebraically closed groups are $\Xsep$-boundedly acyclic.
\end{cor}

Perhaps the most surprising application of Theorem~\ref{intro:thm:ccc} is the following. The \emph{space of enumerated groups} has been recently introduced~\cite{GKL}; it is a Polish space that parametrizes all countable groups with a fixed enumeration. The structure of Polish space allows to use the language of descriptive set theory to talk about properties of groups: In this sense, a property is called \emph{generic} if the subspace of groups that enjoy it is comeager in the sense of Baire category. We show that the generic enumerated group has commuting cyclic conjugates deducing the following:

\begin{cor}[Corollary \ref{cor:generic}]
\label{corintro:generic}
The generic enumerated group is $\Xsep$-boundedly acyclic.
\end{cor}

This theorem is in stark contrast with Gromov's model of random presentations of groups \cite[Section 9.B]{gromov:asymptotic}, where depending on the density the random group is either of order at most $2$ or hyperbolic \cite{random:hyperbolic1, random:hyperbolic2}, and in the latter case it has large second \cite{epstein:fujiwara} and third bounded cohomology in a large class of coefficients \cite{FPS}.\ We may interpret this discrepancy as meaning that, while bounded cohomology vanishing is rare among finitely presented groups, it becomes the norm among countable groups.

\subsection{Methods}

The technical tool underlying these results is the \emph{vanishing modulus}, a numerical invariant, introduced independently by different teams~\cite{monodnariman, fflm2}, that attaches a quantity to a vanishing bounded cohomology group. In particular, the bounded cohomology vanishes in degree $n$ if and only if its $n$-th vanishing modulus is finite~\cite{fflm2}. The importance of the vanishing modulus is encoded in the following: A direct union of boundedly acyclic groups stays boundedly acyclic if the vanishing moduli of all the groups are uniformly bounded~\cite[Proposition 4.13]{fflm2}. Here we generalize this result to $\Xsep$-bounded acyclicity. Therefore the technical core of this paper consists in a further study of the vanishing modulus, and in particular in a proof that the vanishing modulus of the class of permutational wreath products $\Gamma \wr_X A$ is uniformly bounded. 
To this end, we first give a complete proof of the $\Xsep$-bounded acyclicity of permutational wreath products, with the additional assumption that the action $A \actson X$ has infinite orbits~\cite[Section 4.2]{monod:thompson}.
Then, we have to show that those groups are not only $\Xsep$-boundedly acyclic (whence with finite vanishing modulus), but that their vanishing modulus is uniformly bounded independently of $\Gamma, X, A$ or even the coefficient module $E$.
The last step, which is purely algebraic, consists in expressing groups with commuting cyclic conjugates as a directed union of quotients of groups of the form $\Gamma \wr_X A$ by amenable normal subgroups.
This requires a careful inductive construction, since $A$ needs to be infinite for the previous results to hold, but the obvious wreath products arising from the definition of commuting cyclic conjugates may well have finite acting group.

\subsection*{Plan of the paper} In Section \ref{sec:preliminaries} we go through the basics of bounded cohomology, the framework of highly ergodic Zimmer-amenable actions, and discuss the main technical tool of this paper: The vanishing modulus and its application to directed unions. In Section \ref{sec:wp} we prove that certain permutational wreath products are $\Xsep$-boundedly acyclic, and estimate their vanishing modulus. In Section \ref{sec:ccc:bac} we prove Theorem \ref{intro:thm:ccc}. In Section \ref{sec:ex} we go through many examples, proving all corollaries from this introduction. This motivates some questions about possible generalizations of our $\Xsep$-bounded acyclicity results: See Questions \ref{q:braid}, \ref{q:cylinder}, \ref{q:largegroups} and \ref{q:iet}. In Section \ref{s:rigidity} we briefly discuss some implications of $\Xsep$-bounded acyclicity.

\subsection*{Acknowledgements} The authors are indebted to Javier Aramayona, Jon Berrick, Benjamin Bode, Jonathan Bowden, Luigi Caputi, Carlo Collari, Pierre de la Harpe, Rodrigo De Pool, Dominik Francoeur, Camille Horbez, Thorben Kastenholz, Nicol{\'a}s Matte Bon, Nicolas Monod, Sam Nariman, Pierre Py, Oscar Randal-Williams, George Raptis, Sven Sandfeldt, Arpan Saha, Lauro Silini, Marco Varisco, Xiaolei Wu and Matt Zaremsky for useful conversations.
They also thank the anonymous referees for their careful reading and helpful comments.

CC acknowledges support by the European Union in the form of a Mar\'ia Zambrano grant.\ FFF was supported by the Herchel Smith Postdoctoral Fellowship Fund.\ YL was partially supported by the NSF CAREER award 2240136.\ MM was partially supported by the INdAM--GNSAGA Project CUP E55F22000270001 and by the ERC ``Definable Algebraic Topology" DAT - Grant Agreement n. 101077154.

CC, FFF and YL thank also the Instituto de Ciencias Matem\'aticas (ICMAT) for its hospitality during the Thematic program on geometric group theory and low-dimensional geometry and topology (May-July 2023), where part of this research was conducted.\ MM also thanks the Forschungsinstitut f\"ur Mathematik (ETH Z\"urich) for its hospitality, where part of this research was conducted.

This work has been supported by the Madrid Government (Comunidad de Madrid - Spain) under the multiannual Agreement with UAM in the line for the Excellence of the University Research Staff in the context of the V PRICIT (Regional Programme of Research and Technological Innovation). This work has been funded by the European Union - NextGenerationEU under the National Recovery and Resilience Plan (PNRR) - Mission 4 Education and research - Component 2 From research to business - Investment 1.1 Notice Prin 2022 -  DD N. 104 del 2/2/2022, from title ``Geometry and topology of manifolds", proposal code 2022NMPLT8 - CUP J53D23003820001.

\section{Preliminaries}\label{sec:preliminaries}

\subsection{Bounded cohomology with dual coefficients}\label{subsec:bddcohom:def}

Every group in this article is assumed to be \emph{discrete}. In this section we recall the basic definitions of \emph{bounded cohomology} of groups~\cite{Frigerio:book, monod}.

Let $\Gamma$ be a group. We say that $E$ is a \emph{Banach} $\Gamma$-\emph{module} if $E$ is a complete normed $\R$-module admitting a $\Gamma$-action by linear isometries. We refer to $E$ as a \emph{trivial} module when the $\Gamma$-action is the trivial action.
In this paper the Banach $\Gamma$-module $\R$ will be always assumed to be trivial, unless specified otherwise.

Given a group $\Gamma$ and a Banach $\Gamma$-module $E$, we can then define the cochain complex of bounded functions:
\[
C_b^\bullet(\Gamma; E) \coloneqq \Big\{f \colon \Gamma^{\bullet + 1} \to E \, \mid \, \| f \|_\infty \coloneqq \sup_{\gamma_0, \ldots, \gamma_\bullet} \|f(\gamma_0, \ldots, \gamma_\bullet)\|_E< +\infty\Big\}
\]
endowed with the simplicial coboundary operator
\[
\delta^\bullet \colon C_b^\bullet(\Gamma; E) \to C_b^{\bullet+1}(\Gamma; E)
\]
$$\delta^\bullet f(g_0,\ldots,g_{\bullet+1})=\sum_{0\leq i\leq \bullet+1} (-1)^i f(g_0,\ldots,\hat{g_i},\ldots,g_{\bullet+1})$$
that sends bounded functions to bounded functions (we will typically simply denote it by $\delta$ when the degree is clear from the context). By construction $\Gamma$ acts on $C_b^\bullet(\Gamma; E)$ as follows: For every $f \in C_b^\bullet(\Gamma; E)$,
\begin{equation}\label{eq:action:on:functions}
(\gamma \cdot f)(\gamma_0, \ldots, \gamma_\bullet) = \gamma f (\gamma^{-1} \gamma_0, \ldots, \gamma^{-1}\gamma_\bullet)
\end{equation}
for all $\gamma, \gamma_0, \ldots, \gamma_\bullet \in \Gamma$.
We denote by $C_b^\bullet(\Gamma; E)^\Gamma$ the subspace of $\Gamma$-\emph{invariant} bounded functions. It is immediate to check that the coboundary operator is $\Gamma$-equivariant, so that $(C_b^\bullet(\Gamma; E)^\Gamma, \delta^\bullet)$ is also a cochain complex.
We denote by $Z^\bullet_b(\Gamma; E)$ its cocycles and by $B^\bullet_b(\Gamma; E)$ its coboundaries.

\begin{defi}[Bounded cohomology]
    We define the \emph{bounded cohomology of} $\Gamma$ \emph{with coefficients} $E$ as the cohomology of the cochain complex $(C_b^\bullet(\Gamma; E)^\Gamma, \delta^\bullet)$:
    \[
    H_b^\bullet(\Gamma; E) \coloneqq H^\bullet(C_b^\bullet(\Gamma; E)^\Gamma, \delta^\bullet) = Z^\bullet_b(\Gamma; E)/B^\bullet_b(\Gamma; E).
    \]
\end{defi}

\begin{rem}[Comparison map]
    Since the cochain complex of $\Gamma$-invariant bounded functions sits inside the cochain complex of $\Gamma$-invariant (possibly unbounded) functions, there exists a canonical map, called \emph{comparison map},
    \[
\comp^{\bullet} \colon H_b^\bullet(\Gamma; E) \to H^\bullet(\Gamma; E)
    \]
    induced by the inclusion.
\end{rem}

Following Monod's notation~\cite[Sec.~11.1]{monod}, given a group $\Gamma$ we denote by $\X$ a class of Banach $\Gamma$-modules. We are interested in the following classes: 

\begin{enumerate}
    \item $\Xdual$ denotes the class of all dual Banach $\Gamma$-modules.\ We recall that a Banach $\Gamma$-module lies in $\Xdual$ if it is isometrically $\Gamma$-isomorphic to the dual of some Banach $\Gamma$-module;
    \item $\Xsep$ denotes the class of all \emph{separable} dual Banach $\Gamma$-modules;    
   \item $\Xr$ denotes the class consisting only of the Banach $\Gamma$-module $\R$ endowed with the trivial $\Gamma$-action.
\end{enumerate}

By definition we have the following inclusions:
\begin{equation}\tag{*}\label{eq:coeff}
\Xr \subset \Xsep \subset \Xdual.
\end{equation}

\begin{defi}[$\X$-boundedly acyclic groups]
    Let $\Gamma$ be a group, let $\X$ be a class of Banach modules and let $n \geq 1$ be an integer or $n = \infty$. We say that $\Gamma$ is $\X$-\emph{boundedly} $n$-\emph{acyclic} if $H^k_b(\Gamma; E) = 0$
    for every $1 \leq k \leq n$ and every Banach $\Gamma$-module $E$ in $\X$. The group $\Gamma$ is said to be $\X$-\emph{boundedly acyclic} if it is $\X$-boundedly $\infty$-acyclic.
    $\Xr$-boundedly acyclic groups are usually simply called  \emph{boundedly acyclic}.
\end{defi}

The chain of inclusions \eqref{eq:coeff} shows that the classes of $\X$-boundedly acyclic (BAc) groups satisfy the following 
\[
\Xdual\mbox{-}\textup{BAc}  \subseteq \Xsep\mbox{-}\textup{BAc} \subseteq \Xr\mbox{-}\textup{BAc}.
\]
\begin{rem}[Strict inclusions]
As proved by Monod~\cite{monod:thompson}, there are $\Xsep$-boundedly acyclic groups that are not amenable whence not $\Xdual$-boundedly acyclic, e.g. $F_2 \wr \Z$, where $F_2$ denotes the non-abelian free group of rank two.
We do not know of an $\Xr$-boundedly acyclic group that is not $\Xsep$-boundedly acyclic.

However, there exists a group which is $\Xr$-boundedly $2$-acyclic, but is not $\Xsep$-boundedly $2$-acyclic, namely $\Gamma \coloneqq \overline{T} \rtimes \mathbb{Z}/2$, where $\overline{T}$ is the lift of Thompson's group to the real line, and the order $2$ automorphism is given by the change of orientation. Indeed, it follows by a combination of two results \cite[Example 6.12]{autinv} and \cite[Example 12.4.3]{monod} that $H^2_b(\Gamma; \R) = 0$. However, by Eckmann--Shapiro induction~\cite[Proposition 10.1.3]{monod} there exists an isomorphism $H^2_b(\Gamma; \ell^\infty(\Gamma / \overline{T}, \R)) \cong H^2_b(\overline{T}; \R) \neq 0$ (via the rotation quasimorphism~\cite{calegari_scl}); but $\ell^\infty(\Gamma / \overline{T}, \R)$ is $2$-dimensional, thus dual (reflexive) and separable.
\end{rem}

Here is a list of examples of $\X$-boundedly acyclic groups:

\begin{example}[$\X$-boundedly acyclic groups]
We have the following:
\begin{enumerate}
    \item All amenable groups are $\Xdual$-boundedly acyclic~\cite{Johnson};

     \item Thompson's group~$F$ is $\Xsep$-boundedly acyclic~\cite{monod:thompson};
    
    \item Let $\Lambda$ be a group. All wreath products 
    \[
    \Lambda \wr \Z = \big(\bigoplus_{\Z} \Lambda) \rtimes \Z
    \]
    are $\Xsep$-boundedly acyclic~\cite{monod:thompson};

    \item Every countable group embeds into a $\Xr$-boundedly acyclic group, and there exists a finitely presented $\Xr$-boundedly acyclic group that contains all finitely presented groups \cite{fflm1};

    \item All binate groups \cite{berrick} are $\Xr$-boundedly acyclic~\cite{Matsu-Mor, Loeh:dim, fflm2};

    \item The groups of homeomorphisms and diffeomorphisms of compact support of manifolds of the form $\R^n \times M$, with $M$ closed  and $n\geq 1$, are $\Xr$-boundedly acyclic \cite{monodnariman}.
    \end{enumerate}
\end{example}

We end this section with Gromov's foundational Mapping Theorem, that will be useful in some arguments along the text.
\begin{thm}[Gromov's Mapping Theorem~{\cite{vbc}, \cite[Theorem~4.24]{Frigerio:book}}]
\label{thm:mapping}
Let $n \geq 1$ be an integer.
Let $\Gamma$ be a group and let $N \leq \Gamma$ be an amenable normal subgroup. Let $\pi \colon \Gamma \to \Lambda \coloneqq \Gamma / N$ be the quotient, and let $E$ be a dual Banach $\Lambda$-module. Then the pullback $H^n_b(\Lambda; E) \to H^n_b(\Gamma; \pi^\ast E)$ is an isometric isomorphism in all degrees.

In particular, given a class $\X \subseteq \Xdual$ of Banach $\Lambda$-modules,
if $\Gamma$ is $\X$-boundedly $n$-acyclic, then $\Lambda$ is also $\X$-boundedly $n$-acyclic.
\end{thm}

\begin{rem}[Gromov's Mapping Theorem via inflation]
\label{rem:mapping:converse}
Gromov's Mapping Theorem can also be stated in the more general case of \emph{inflation}, namely also the induced map $H^n_b(\Lambda; E^N)\to H^n_b(\Gamma; E)$ is an isometric isomorphism in all degrees and for all dual Banach $\Gamma$-modules $E$~\cite[Theorem in the introduction]{moraschini_raptis_2023}. Here $E^N$ denotes the subspace of the $N$-fixed points of $E$.\ Thus, we also get a converse to the last conclusion: 
if $\Lambda$ is $\X$-boundedly $n$-acyclic for some family $\X \subseteq \Xdual$ of Banach $\Gamma$-modules, then $\Gamma$ is also $\X$-boundedly $n$-acyclic.
\end{rem}

Recall that a subgroup $H$ of a group $G$ is \emph{co-amenable} in $G$ if $\ell^{\infty}(G/H)$ admits a $G$-invariant mean.

\begin{thm}[co-amenability and bounded cohomology]{\cite[Proposition~3]{Monod:Popa}}\label{thm:co-amenable}
Let $n \geq 1$ be an integer.
Let $\Gamma$ be a group, let $\Lambda \leq \Gamma$ be a co-amenable subgroup, and let $E$ be a dual Banach $\Gamma$-module. Then the restriction $H^n_b(\Gamma; E) \to H^n_b(\Lambda; E)$ is injective in all degrees.

In particular, given a class $\X \subseteq \Xdual$ of Banach $\Gamma$-modules, if $\Lambda$ is $\X$-boundedly $n$-acyclic, then $\Gamma$ is also $\X$-boundedly $n$-acyclic.
\end{thm}

\subsection{Highly ergodic Zimmer-amenable actions}\label{subsec:HEZA}

We recall in this section how to compute bounded cohomology using highly ergodic Zimmer-amenable actions. Recall that a Borel space is \emph{standard} if it is Borel isomorphic to the Borel space associated to a Polish space. Given a non-singular (i.e.\ measure class preserving) action of a group $\Gamma$ on a standard Borel measure space $S$, and a Banach $\Gamma$-module $E$, there is an induced action of $\Gamma$ on $L^\infty(S, E)$, defined in the same way as the action of $\Gamma$ on $C^n_b(\Gamma; E)$ (see Equation~\eqref{eq:action:on:functions}).

\begin{defi}[Zimmer-amenable space]
\label{def:zimmeramenable}
Let $\Gamma$ be a group.
Let $S$ be a standard Borel probability space, equipped with a non-singular action of $\Gamma$. A \emph{conditional expectation} 
\[
m \colon L^\infty(\Gamma \times S, \R) \to L^\infty(S, \R)
\]
is a norm one linear map such that:
\begin{enumerate}
\item $m(\mathbbm{1}_{\Gamma \times S}) = \mathbbm{1}_S$;
\item For all $f \in L^\infty(\Gamma \times S, \R)$ and every measurable subset $A \subset S$, it holds $m(f \cdot \mathbbm{1}_{\Gamma \times A}) = m(f) \cdot \mathbbm{1}_A$.
\end{enumerate}
We say that $S$ is a \emph{Zimmer-amenable} $\Gamma$-space, if there exists a conditional expectation $m \colon L^\infty(\Gamma \times S) \to L^\infty(S)$ that is moreover $\Gamma$-equivariant (on $\Gamma\times S$ we consider the diagonal $\Gamma$-action).
\end{defi}

The following proposition provides useful examples of Zimmer-amenable spaces that will be of crucial use for us~\cite[Section~2.3, Corollary 8 and Proposition 9]{monod:thompson}.

\begin{prop}[Zimmer-amenable spaces]\label{prop:zimmer:ame:spaces}
We have the following:
    \begin{enumerate}
        \item Let $\Gamma$ be a countable group and let $S$ be the standard Borel probability $\Gamma$-space obtained by taking $S = \Gamma$ and by endowing it with a distribution of full support. Then, the $\Gamma$-action on $S$ via left multiplication turns $S$ into a Zimmer-amenable $\Gamma$-space. Moreover, $S$ is a Zimmer-amenable $\Lambda$-space, for every subgroup $\Lambda < \Gamma$.
        \item Let $\{\Gamma_i\}_{i \in \N}$ be a countable collection of countable groups and let $\{S_i\}_{i \in \N}$ be a countable collection of Zimmer-amenable $\Gamma_i$-spaces. Then, the direct product $\prod_{i \in \N}S_i$ is a Zimmer-amenable $\bigoplus_{i \in \N} \Gamma_i$-space.
        \item Let $\Gamma$ be a countable group and let $\Lambda < \Gamma$ be a co-amenable subgroup. Let $S$ be a standard Borel probability space equipped with a non-singular action of $\Gamma$. If $S$ is a Zimmer-amenable $\Lambda$-space, then it is also a Zimmer-amenable $\Gamma$-space.
    \end{enumerate}
\end{prop}

Zimmer-amenable spaces play a fundamental role in the theory of bounded cohomology:

\begin{thm}[bounded cohomology and Zimmer-amenable spaces~{\cite[Theorem 7.5.3]{monod}}]
\label{thm:bc:via:zimmeramenable}

Let $\Gamma$ be a countable group, let $E \in \Xsep$ be a Banach $\Gamma$-module and let $S$ be a Zimmer-amenable $\Gamma$-space. Then the complex
\[0 \to L^\infty(S, E)^{\Gamma} \to L^\infty(S^2, E)^{\Gamma} \to L^\infty(S^3, E)^{\Gamma} \to \cdots,\]
endowed with the simplicial coboundary operator, isometrically computes the bounded cohomology of $\Gamma$ with coefficients in $E$.
\end{thm}

This comes particularly handy when $L^\infty(S^{n+1}, E)^{\Gamma}$ is small. Recall the following definition:

\begin{defi}[ergodicity with coefficients]
\label{def:ergodic}

Let $S$ be a standard Borel probability space, equipped with a non-singular action of $\Gamma$. We say that the action is \emph{ergodic} if every $\Gamma$-invariant subset of $S$ has either zero or full measure. Equivalently, the $\Gamma$-action on $S$ is ergodic if every essentially bounded measurable $\Gamma$-invariant function $S \to \R$ is essentially constant; in symbols: $L^\infty(S, \R)^{\Gamma} \cong \R$ (recall that $\R$ is always endowed with the trivial $\Gamma$-action).
Now let $E$ be a Banach $\Gamma$-module.\ We say that the action of $\Gamma$ on $S$ is \emph{ergodic with coefficients in $E$} if every essentially bounded measurable $\Gamma$-equivariant function $S \to E$ is essentially constant; in symbols $L^\infty(S, E)^{\Gamma} \cong E^{\Gamma}$.
\end{defi}

Before moving on we introduce the following definitions:

\begin{defi}[high ergodicity]\label{defi:highly:ergodic}
    Let $\Gamma$ be a group and let $S$ be a standard Borel probability space with a non-singular $\Gamma$-action. We say that the action $\Gamma \actson S$ is \emph{highly ergodic} if the diagonal action $\Gamma \actson S^n$ is ergodic for all integers $n \geq 1$.
\end{defi}

\begin{defi}[the generalised Bernoulli shift]\label{ex:gen:bernoulli:is:ergodic}
Let $\Gamma$ be a countable group acting on a countable set $X$.
Let $Y$ be a standard Borel probability space and assume that the measure does not concentrate on a single point. We can then define
a new standard Borel probability space $Y^X = \prod_{x \in X} Y$. 
The \emph{generalised Bernoulli shift} is the dynamical system given by the $\Gamma$-action of $Y^X$ defined as follows:
\[
\gamma \cdot(y_x)_{x \in X} \coloneqq (y_{\gamma^{-1} \cdot x})_{x \in X},
\]
for all $\gamma \in \Gamma$ and $(y_x)_{x \in X} \in Y^X$. 
\end{defi}

We now characterise when the generalised Bernoulli shift is ergodic:
\begin{prop}[ergodicity of the generalised Bernoulli shift]\label{prop:KechrisTsankov:ergodic}
Let $Y$ be a standard Borel probability space and assume that the measure does not concentrate on a single point. Let $\Gamma$ be a countable group acting on a countable set $X$.
Consider the generalised Bernoulli shift action of $\Gamma$ on $Y^X$ (as in the example above).
The following are equivalent:
\begin{enumerate}

\item The action of $\Gamma$ on $Y^X$ is ergodic.

\item The action of $\Gamma$ on $Y^X$ is highly ergodic.

\item All the orbits of the action of $\Gamma$ on $X$ are infinite.

\end{enumerate}

\end{prop}

\begin{proof}
    \emph{(1) $\Longleftrightarrow$ (3)} This is proved by Kechris and Tsankov~\cite[Proposition~2.1]{kechris:tsankov}.

    \emph{(2) $\Longrightarrow$ (1)} It is straightforward.

    \emph{(1) $\Longrightarrow$ (2)} By assumption the $\Gamma$-action on $Y^X$ is ergodic and so all the orbits of the $\Gamma$-action on $X$ are infinite by \emph{(3)}. We have to show that the diagonal action of $\Gamma$ on $\prod_{i = 1}^n Y^X$ is ergodic for every integer $n \geq 1$. For each $n$, this action is a generalised Bernoulli shift where $\Gamma \actson Y^{\sqcup_{i = 1}^n X}$ simply by acting diagonally on the disjoint union of the sets $X$. Moreover, all the orbits of the $\Gamma$-action on $X \sqcup \cdots \sqcup X$ are infinite. Hence, applying now \emph{(3)} $\Rightarrow$ \emph{(1)} to the generalised Bernoulli shift $\Gamma \actson Y^{\sqcup_{i = 1}^n X}$ we conclude that the action is ergodic, whence the thesis.
\end{proof}

\begin{rem}[ergodicity and subgroups]\label{rem:ergodicity:subgroups}
    Let $\Gamma$ be a group and let $Y$ be a standard Borel probability 
    $\Gamma$-space. Suppose that there exists a subgroup $\Lambda \leq 
    \Gamma$ such that the restricted action $\Lambda \actson Y$ is 
    ergodic. Then, the action $\Gamma \actson Y$ is also ergodic (indeed, 
    if all the $\Lambda$-invariant functions are essentially constant 
    also the $\Gamma$-invariant functions must be so).
\end{rem}

\begin{defi}[HEZA]
\label{intro:defi:heza}
Let $\Gamma$ be a group and let $S$ be a standard Borel probability $\Gamma$-space. We say that $S$ is \emph{highly ergodic Zimmer-amenable}, \emph{$\heza$} for short, if $S$ is a Zimmer-amenable $\Gamma$-space and the $\Gamma$-action on $S$ is highly ergodic.
\end{defi}

The following is an important criterion for the vanishing of bounded cohomology which will be of key use for us.

\begin{cor}[HEZA spaces and bounded acyclicity~{\cite[p.~668]{monod:thompson}}]
\label{cor:zimmer:ergodic}
Let $\Gamma$ be a countable group for which there exists a $\heza$-space.
Then $\Gamma$ is $\Xsep$-boundedly acyclic.
\end{cor}

\subsection{Commuting conjugates}\label{subsec:cc}

In the sequel we are going to introduce sufficient conditions for $\X$-bounded acyclicity. We recall here the notion of \emph{commuting conjugates}~\cite{kotschick, fflodha} that we will extend later:

\begin{defi}[commuting conjugates]\label{defi:comm:conj}
    A group $\Gamma$ has \emph{commuting conjugates} if for every finitely generated subgroup $H \leq \Gamma$ there exists an element $g \in \Gamma$ such that 
    \[
    [H, gHg^{-1}] = 1.
    \]
    For simplicity from now on we will always denote the conjugation by ${}^{g}H \coloneqq gHg^{-1}$.
\end{defi}

This natural property detects bounded acyclicity in low degree:

\begin{thm}[{commuting conjugates and bounded $2$-acyclicity~\cite[Theorem~1.2]{fflodha}}]
    Let $\Gamma$ be a group with commuting conjugates. Then, $\Gamma$ is $\Xr$-boundedly $2$-acyclic.
\end{thm}

\begin{example}[groups with commuting conjugates]
    The following groups have commuting conjugates:
    \begin{enumerate}
        \item Let $R$ be a ring. The infinite general linear group $GL(R)$ has commuting conjugates~\cite[proof of Proposition~4.9]{fflm2}.
\item The group of symplectic (or Hamiltonian) diffeomorphisms of $\R^n$ with compact support and $n \geq 2$ has commuting conjugates. This result implicitly dates back to Kotschick's work on the stable commutator length of the identity components of symplectic (or Hamiltonian) diffeomorphism groups~\cite[Section~4]{kotschick}.

\item Every group $\Gamma$ of boundedly supported homeomorphisms of the real line with no global fixpoints has commuting conjugates~\cite[Proposition~3.8]{fflodha}.
    \end{enumerate}
\end{example}

\subsection{Direct unions}\label{subsec:direct:union}

We investigate here the most important combination result for the proof of Theorem \ref{intro:thm:ccc}.\ To our knowledge it is not known whether direct unions of $\X$-boundedly acyclic groups are $\X$-boundedly acyclic (besides the case $\Xdual$, of course, corresponding to amenable groups). However, direct unions with \emph{controlled} vanishing do preserve bounded acyclicity.\ This involves the notion of \emph{vanishing modulus with coefficients}, that extends the previous notion for trivial real coefficients~\cite{fflm2, monodnariman}. The notion of vanishing modulus for trivial real coefficients has already proved useful to show that bounded acyclicity is preserved under certain constructions:\ Namely if the vanishing modulus of a family of groups is uniformly bounded, then several constructions in bounded cohomology go through~\cite{fflm2, liloehmoraschini}. Our aim is to extend this philosophy and make it quantitative in the case of coefficients. Let us start with the definition:

\begin{defi}[vanishing modulus with coefficients]
\label{def:vanishingmodulus}

Let $\Gamma$ be a group, let $E$ be a Banach $\Gamma$-module and let $n$ be a positive integer.\ The $n$-th \emph{vanishing modulus of} $\Gamma$ \emph{with coefficients in} $E$, denoted by $\| H^n_b(\Gamma; E) \|$, is defined as the minimal $K \in [0, \infty]$ such that the following holds: For every cocycle $z \in Z^n_b(\Gamma; E)$ there exists a cochain $b \in C^{n-1}_b(\Gamma; E)^{\Gamma}$ such that $\delta^{n-1}(b) = z$ and $\| b \|_\infty \leq K \cdot \| z \|_\infty$.
\end{defi}

We apply the usual convention that $\inf \varnothing = \infty$, so if $H^n_b(\Gamma; E) \neq 0$, then $\| H^n_b(\Gamma; E) \| = \infty$. An easy application of the Open Mapping Theorem implies the converse:

\begin{lemma}[$\X$-boundedly acyclic groups have uniformly bounded vanishing modulus]
\label{lem:vm:finite}
Let $\Gamma$ be a group, let $E$ be a Banach $\Gamma$-module and let $n \geq 1$ be an integer. 
Then, we have the following:
\begin{enumerate}
    \item If $H^n_b(\Gamma; E) = 0$ then $\| H^n_b(\Gamma; E) \| < +\infty$;
    \item If $H^n_b(\Gamma; E) = 0$ for every $E \in \Xsep$, then $\| H^n_b(\Gamma; E) \|$ is uniformly bounded over all such $E$, i.e.\ 
    \[
    \sup_{E \in \Xsep} \| H_b^n(\Gamma; E) \| < +\infty .\]
\end{enumerate}
\end{lemma}

The first statement in the case of trivial real coefficients was already implicit in the work of Matsumoto and Morita~\cite{Matsu-Mor} but the details have been worked out only recently~\cite[Lemma 4.12]{fflm2} (compare also with the work by Monod and Nariman~\cite[after Definition 2.5]{monodnariman}).

\begin{proof}
\emph{Ad~(1)} Let $E$ be a Banach $\Gamma$-module and let us suppose that $H^n_b(\Gamma; E)=0$.\ Then, we can identify the space of $n$-cocycles with the space of $n$-coboundaries
\[
Z^n_b(\Gamma; E) = B^n_b(\Gamma; E) = \delta^{n-1}(C^{n-1}_b(\Gamma; E)^{\Gamma}).
\]
Since this is a closed subspace of the Banach space $C^{n}_b(\Gamma; E)^{\Gamma}$, it is itself a Banach space.\
Moreover, since $Z^{n-1}_b(\Gamma; E)$ is a closed subspace of the Banach space $C^{n-1}_b(\Gamma; E)^{\Gamma}$ there exists a natural quotient Banach space $C^{n-1}_b(\Gamma; E)^{\Gamma}/Z^{n-1}_b(\Gamma; E)$ endowed with the quotient norm.
It follows that the bounded linear operator $\delta^{n-1}$ is a surjective continuous linear map between Banach spaces, so the Open Mapping Theorem says that this is an open map. So the map
\[\delta^{n-1} \colon C^{n-1}_b(\Gamma; E)^{\Gamma} / Z^{n-1}_b(\Gamma; E) \to Z^n_b(\Gamma; E),\]
has a bounded inverse $\varphi$, which has finite operator norm $\| \varphi \|_{op}$.\ Then, for every cocycle $z \in Z^{n}_b(\Gamma; E)$, the class $\varphi(z)$ admits a representative $b \in C^{n-1}_b(\Gamma; E)^{\Gamma}$ of $\ell^\infty$-norm at most $2 \cdot \|\varphi\|_{op} \cdot \| z \|_\infty$. This shows that $\| H^n_b(\Gamma; E) \| < +\infty$, whence the thesis.

\medskip

\emph{Ad~(2)} Suppose now that $H^n_b(\Gamma; E) = 0$ for every Banach $\Gamma$-module $E \in \Xsep$. 
Assume by way of contradiction that: 
\[
\sup_{E \in \Xsep} \| H_b^n(\Gamma; E) \| = +\infty.
\]
Then there exists a countable sequence of Banach $\Gamma$-modules $E_i \in \Xsep$ such that $\| H_b^n(\Gamma; E_i) \| \to +\infty$ when $i \to +\infty$.
Then, we can consider their $\ell^2$-direct sum
$E = \bigoplus_{2} E_i$~\cite[p.~72]{conway}.\ 
It is easy to show that $E$ is a separable dual Banach $\Gamma$-module \cite[Chapter III, Proposition~4.4, Exercise~4.13 and Exercise~5.4]{conway}.
By our hypothesis $H^n_b(\Gamma; E) = 0$, and thus $K \coloneqq \| H^n_b(\Gamma; E) \| < +\infty$ by \emph{(1)}.

To achieve a contradiction, it suffices to show that $\| H^n_b(\Gamma; E_i) \| \leq K$ for all $i \in \N$.
Let us fix $i \in \N$. We consider the inclusion~$\iota \colon E_i \to E$ and the projection~$\pi \colon E \to E_i$.\ They are both $\Gamma$-invariant bounded maps~\cite[Chapter III, Proposition 4.4]{conway}.
Hence, the inclusion and the projection maps induce the following maps at the level of bounded cochain complexes:
\[
\pi_\ast^\bullet \colon C_b^\bullet(\Gamma; E)^\Gamma \to C_b^\bullet(\Gamma; E_i)^\Gamma
\]
and
\[
\iota_\ast^\bullet \colon C_b^\bullet(\Gamma; E_i)^\Gamma \to C_b^\bullet(\Gamma; E)^\Gamma.
\]
Let $z \in C^n_b(\Gamma; E_i)$ be a cocycle. We are going to show that we can find a primitive of norm at most $K \cdot \sv{z}_\infty$. First note that $\iota_*^n(z)\in C^n_b(\Gamma; E)^\Gamma$ is also a cocycle. By definition of vanishing modulus with coefficients in $E$ and the definition of $K$, we know that there exists a cochain $b \in C^{n-1}_b(\Gamma; E)^{\Gamma}$ such that $\delta^{n-1}(b) = \iota_*^n(z)$ and $\| b \|_\infty \leq K \cdot \| \iota_*^n(z) \|_\infty$. Then we define $\pi_*^{n-1}(b) \in C^{n-1}_b(\Gamma; E_i)^{\Gamma}$. By construction we have that
\[\| \pi_*^{n-1} (b) \|_\infty \leq \| b \|_\infty \leq K \cdot \| \iota_*^n (z) \|_\infty \leq K \cdot \| z \|_\infty.\]
Since $\delta^{n-1}(\pi^{n-1}_*(b)) = z$, this proves our claim.
\end{proof}

The reason why the vanishing modulus is relevant for us is the following result, that extends the corresponding statement for trivial real coefficients~\cite[Proposition 4.13]{fflm2} (in fact that statement is only about direct unions, but for our purposes we need to treat slightly more general unions):

\begin{prop}[vanishing modulus with coefficients and direct unions]
\label{prop:dirun}
Let $I$ be an arbitrary index set and let $n \geq 1$ be an integer. Let $\Gamma$ be a group, and let $\{\Gamma_i\}_{i \in I}$ be a family of subgroups of $\Gamma$ such that each finitely generated subgroup of $\Gamma$ is contained in some $\Gamma_i$. Let $E$ be a Banach $\Gamma$-module in $\Xdual$, thus by restriction also a Banach $\Gamma_i$-module for all $i \in I$. Then
\[\| H^n_b(\Gamma; E) \| \leq \sup\limits_{i \in I} \| H^n_b(\Gamma_i; E) \|.\]
In particular, if the right-hand side is finite, then $H^n_b(\Gamma; E) = 0$.
\end{prop}

\begin{proof}
Let $K := \textup{sup}_{i \in I} \| H^n_b(\Gamma_i; E) \|$ and suppose that $K < +\infty$, otherwise there is nothing to show. Since $E \in \Xdual$, there exists a predual $F$ with the corresponding $\Gamma$-action.\ This shows that also $C^n_b(\Gamma; E) = \ell^\infty(\Gamma^{n+1}, E) \in \Xdual$ because it is the dual of $\ell^1(\Gamma^{n+1}, F)$ with the contragradient $\Gamma$-action. In particular, it comes endowed with the weak-$\ast$ topology. Fix a cocycle $z \in Z^n_b(\Gamma; E)$, and for each finitely generated subgroup $H \leq \Gamma$ define
\[B_H^z := \{ b \in C^{n-1}_b(\Gamma; E)^{H} \mid \delta^{n-1}(b|_{H}) = z|_{H} \text{ and } \| b \|_\infty \leq K \cdot \| z \|_\infty \}.\]
It suffices to show that $\bigcap_{H \leq \Gamma \text{ fin. gen.}} B_H^z \neq \varnothing$: Then since $\Gamma$ is the direct union of its finitely generated subgroups, an element $b$ in the intersection will satisfy 
\begin{itemize}
    \item $b \in C^{n-1}_b(\Gamma; E)^{\Gamma}$;
    \item $\delta^{n-1}(b) = z$;
    \item $\| b \|_\infty \leq K \cdot \| z \|_\infty$.
\end{itemize}
Now each $B_H^z$ is a bounded weak-$\ast$ closed subset of $B_{K \cdot \sv{z}_\infty}\subset C^{n-1}_b(\Gamma; E)$.
Hence, by the Banach--Alaoglu Theorem, it suffices to show that the family $\{B_H^z\}_{H \leq \Gamma \text{ fin. gen.}}$ satisfies the finite intersection property.

Since $B_H^z \subset B_{H'}^z$ whenever $H' \leq H$, and the system is directed, the finite intersection property will follow as soon as each $B_H^z$ is non-empty. Now, fix $H \leq \Gamma$ and let $i$ be such that $H \leq \Gamma_i$. Since $\| H^n_b(\Gamma_i; E) \| \leq K$, there exists $b \in C^{n-1}_b(\Gamma_i, E)^{\Gamma_i}$ such that $\delta^{n-1}(b) = z|_{\Gamma_i}$ and $\| b \|_\infty \leq K \cdot \| z \|_\infty$; extending $b$ by $0$ we obtain an element in $B_H^z$, which concludes the proof.
\end{proof}

The next proposition describes a way to obtain a uniform upper bound on a family of groups without actually having to estimate the individual vanishing moduli. It exploits Lemma \ref{lem:vm:finite}.

\begin{prop}[vanishing modulus with coefficients and direct sums]
\label{prop:dirsum}
Let $I$ be an arbitrary index set and let $n \geq 1$ be an integer. Let $\{\Gamma_i\}_{i \in I}$ be a family of groups and let $\Gamma := \bigoplus_{i \in I} \Gamma_i$. Suppose that the classes of Banach $\Gamma_i$-modules $\X_i$ are in $\Xsep$ for all $i \in I$. Then, if $\Gamma$ is $\Xsep$-boundedly acyclic, we have that
\[
\sup_{i \in I, E_i \in \X_i} \| H^n_b(\Gamma_i; E_i) \| < +\infty.
\]
\end{prop}

A similar result for trivial real coefficients and unrestricted products was stated by Monod and Nariman~\cite[Proposition 2.6]{monodnariman}.\ Here we provide an elementary, self-contained proof, which allows for non-trivial coefficients as well. The core of the proposition lies in the following more general fact. Recall that a surjective homomorphism $r \colon \Gamma \to \Lambda$ is a \emph{retraction} if there exists 
an injective homomorphism $\iota \colon \Lambda \to \Gamma$ such that $r \circ \iota = \textup{id}_{\Lambda}$.

\begin{lemma}[vanishing modulus with coefficients and retractions]
\label{lem:vm:retractions}
Let $r \colon \Gamma \to \Lambda$ be a retraction, let $E$ be a Banach $\Lambda$-module and let $n \geq 1$ be an integer. Then $\| H^n_b(\Gamma; r^\ast E) \| \geq \| H^n_b(\Lambda; E) \|$.
\end{lemma}
\begin{proof}
Let $K \coloneqq \| H^n_b(\Gamma; r^\ast E) \|$ and suppose that $K < +\infty$, otherwise there is nothing to show. Let $z \in C^n_b(\Lambda; E)^\Lambda$ be a cocycle and let $C_b^n(r)(z) \in Z^n_b(\Gamma; r^\ast E)$ be its pullback along $r$. Then, by definition of $K$ there exists a cochain $b \in C_b^{n-1}(\Gamma; r^\ast E)^{\Gamma}$ such that $\delta^{n-1}(b) = C^n_b(r)(z)$ and $\| b \|_\infty \leq K \cdot \| C^n_b(r)(z) \|_\infty$. Then the cochain $C^{n-1}_b(\iota) (b) \in C_b^{n-1}(\Lambda; E)^{\Lambda}$ satisfies 
\[
\delta^{n-1} \circ C^{n-1}_b(\iota)(b) = C_b^n(\iota) \circ \delta^{n-1}(b) = C_b^n(\iota) \circ C^n_b(r) (z) = z
\] and
\[ \| C_b^{n-1}(\iota)(b) \|_\infty \leq \| b \|_\infty \leq K \cdot \| C_b^{n}(r)(z) \|_\infty \leq K \cdot \| z \|_\infty.
\]
This shows that $\| H^n_b(\Lambda; E) \| \leq K$, whence the thesis.
\end{proof}

\begin{proof}[Proof of Proposition \ref{prop:dirsum}]
By Lemma \ref{lem:vm:finite}.\emph{(2)} we know that there exists a uniform constant $K$ such that $\| H^n_b(\Gamma; E) \| \leq K$ for every Banach $\Gamma$-module $E$ in $\Xsep$.
In particular, for all $i \in I$ and for all Banach $\Gamma_i$-modules $E_i\in\mathfrak{X}_i$, letting $\pi_i \colon \Gamma \to \Gamma_i$ denote the corresponding retraction (that is the projection onto the $i$-th factor), Lemma \ref{lem:vm:retractions} implies that $\| H^n_b(\Gamma_i; E_i) \| \leq \| H^n_b(\Gamma; \pi_i^\ast E_i) \| \leq K$.\ Since $i \in I$ and $E_i$ were arbitrary chosen, the thesis follows.
\end{proof}

\begin{rem}[Uniform vanishing moduli]
    As we mentioned before, we do not know whether in general direct unions of $\Xsep$-boundedly acyclic groups are $\Xsep$-boundedly acyclic.\ It is hard to find a counterexample, since all known examples of $\Xsep$-boundedly acyclic groups have uniform vanishing modulus, as we will see later in this paper.\ For $\Xr$-bounded acyclicity, one may think that binate groups \cite{fflm2} provide potential counterexamples. However their vanishing moduli (with respect to $\Xr$) are also uniformly bounded: This follows by combining Proposition~\ref{prop:dirsum} with the fact that direct sums of binate groups are binate~\cite[Proposition 1.7]{binate:product}.
\end{rem}

\section{Wreath products}\label{sec:wp}

In this section, we extend the vanishing result of the bounded cohomology of wreath products~\cite{monod:thompson} by considering \emph{permutational} wreath products. That this is possible was already noticed by Monod \cite[Section 4.2]{monod:thompson}. Though at first glance a minor generalisation, this paves the way to our new criterion for $\Xsep$-bounded acyclicity.

\subsection{Permutational wreath products}\label{subsec:permutational:wp}

We extend Monod's construction for computing the bounded cohomology of wreath products to the case of \emph{permutational wreath products}.

Let $\Gamma, A$ be groups and let $X$ be an $A$-set. The \emph{permutational (restricted) wreath product} $\Gamma\wr_X A$ is defined as $\left(\bigoplus_{x\in X}\Gamma\right)\rtimes A$, where $A$ acts on the subgroup $\bigoplus_{x\in X}\Gamma$ by permutation of the factors. The case $X = A$ with the regular action recovers the usual \emph{regular wreath product}, or simply wreath product, denoted by $\Gamma\wr A$.

\begin{setup}[Permutational wreath products with amenable acting group]\label{setup:perm:wreath:products}
We consider the following setup:
\begin{itemize}
    \item Let $\Gamma$ and $A$ be two countable groups, with $A$ amenable;
    \item Let $X$ be a countable $A$-set such that all the $A$-orbits are infinite;
    \item Let $(Y_0, \mu_0)$ be the standard Borel probability $\Gamma$-space obtained by endowing $\Gamma$ with a distribution of full support (so $Y_0 = \Gamma$ as a set);
    \item Let $Y \coloneqq Y_0^X$ be the standard Borel probability space endowed with the product measure $\mu$;
    \item The permutational wreath product $\Gamma \wr_X A$ acts on $Y$ as follows: $\bigoplus_{x \in X} \Gamma$ acts coordinate-wise and $A$ acts by shifting the coordinates.
\end{itemize}    
\end{setup}

An analogous space was considered by Monod~\cite[Section 2.3]{monod:thompson} in the case of regular wreath products, and the generalization is straightforward. We have the following:

\begin{prop}[$A$ amenable $\Rightarrow$ $Y$ is $\heza$]\label{prop:Y:heza}
    In the Setup~\ref{setup:perm:wreath:products}, the $\Gamma \wr_X A$-space $Y$ is a $\heza$-space.
\end{prop}
\begin{proof}
    Under the hypothesis of Setup~\ref{setup:perm:wreath:products}, by Proposition \ref{prop:KechrisTsankov:ergodic}, we have that the action $A \actson (Y, \mu)$ is highly ergodic and so, by Remark~\ref{rem:ergodicity:subgroups},  the same holds for $\Gamma \wr_X A \actson (Y, \mu)$.

    We are left to show that the action is also Zimmer-amenable.
    By Proposition~\ref{prop:zimmer:ame:spaces}.\emph{(1)} the space $Y_0$ is a Zimmer-amenable $\Gamma$-space. Moreover, Proposition~\ref{prop:zimmer:ame:spaces}.\emph{(2)} shows that $Y$ is a Zimmer-amenable $\bigoplus_{x \in X} \Gamma$-space (because $X$ is countable).\ Since the base group $\bigoplus_{x \in X} \Gamma$ is co-amenable in $\Gamma \wr_X A$ (because $A$ is amenable), by Proposition~\ref{prop:zimmer:ame:spaces}.\emph{(3)} we conclude that $Y$ is a Zimmer-amenable $\Gamma \wr_X A$-space.
\end{proof}

\begin{rem}[Infinite orbits]
    The assumption on the cardinality of the $A$-orbits is essential, since as we observed in Proposition~\ref{prop:KechrisTsankov:ergodic} the $A$-action on a generalised Bernoulli shift is ergodic if and only if the cardinality of each orbit of the action $A \actson X$ is infinite.
\end{rem}

\begin{cor}[$A$ amenable $\Rightarrow$ $\Gamma \wr_X A$ is $\Xsep$-boundedly acyclic]
\label{cor:HEZA:permutational}
In the Setup~\ref{setup:perm:wreath:products}, the permutational wreath product $\Gamma \wr_X A$ is $\Xsep$-boundedly acyclic.
\end{cor}

\begin{proof}
The proof is now a straightforward combination of Corollary~\ref{cor:zimmer:ergodic} and Proposition~\ref{prop:Y:heza}.
\end{proof}

\subsection{Vanishing modulus with coefficients}\label{subsec:van:mod:coeff}

Quite surprisingly, the vanishing result in Corollary~\ref{cor:HEZA:permutational} is very controlled. Our next goal is to make this precise. This fact will be at the heart of the proof of Theorem~\ref{intro:thm:ccc}:

\begin{prop}[vanishing modulus and permutational wreath products]
\label{prop:vanishingmodulus:wreathproducts}
For every integer $n \geq 1$ there exists $K_n \in [0, \infty)$ with the following property:

Let $\Gamma, X$ and $A$ be as in Setup~\ref{setup:perm:wreath:products}.
Let $\Lambda$ be a quotient of $\Gamma \wr_X A$ by a normal amenable group $N$. Then we have $\| H^n_b(\Lambda; E) \| \leq K_n$ for every separable dual Banach $\Lambda$-module $E$.
\end{prop}

\begin{proof}
Let $n \geq 1$ be an integer. Suppose by contradiction that there does not exist such $K_n < +\infty$. Then we can find a countable sequence
\[
\{\Lambda_i = \Gamma_i \wr_{X_i} A_i \slash N_i\}_{i \in \N},
\]
where each $\Gamma_i \wr_{X_i} A_i$ is as in Setup~\ref{setup:perm:wreath:products}, and $N_i$ is a normal amenable subgroup; 
and for each $i \in \N$ a separable dual Banach $\Lambda_i$-module $E_i$ such that
\[
\sup_{i \in \N} \| H^n_b(\Lambda_i; E_i) \| = +\infty.
\]

{\bf Claim}: The direct sum $\bigoplus_{i \in \N} \Lambda_i$ is $\Xsep$-boundedly acyclic.

\medskip

Once the claim is proved, it will be in contradiction with Proposition~\ref{prop:dirsum}. This will imply that the original assumption was false and the required $K_n<\infty$ exists. 

\medskip
\emph{Proof of the claim:} Since the kernel of the quotient map $\bigoplus_{i \in \N} (\Gamma_i \wr_{X_i} A_i) \to \bigoplus_{i \in \N} \Lambda_i$ is the amenable group $\bigoplus_{i \in \N} N_i$ (being the direct sum of amenable groups), using Theorem~\ref{thm:mapping} it suffices to show the $\Xsep$-bounded acyclicity of the group $G \coloneqq \bigoplus_{i \in \N} (\Gamma_i \wr_{X_i} A_i)$.

We are going to show that there exists a $\heza$-space for $G$. Let $Y$ be $G$ itself endowed with a distribution of full support. This is a Zimmer-amenable space for $G$, and for each subgroup $\Gamma_i$, by Proposition~\ref{prop:zimmer:ame:spaces}.\emph{(1)}. Hence for each $i \in \N$ we can construct a generalised Bernoulli shift $A_i \actson Y^{X_i}$ and obtain an action of $\Gamma_i \wr_{X_i} A_i$ on $Y^{X_i}$ as explained in Setup~\ref{setup:perm:wreath:products}. Proposition~\ref{prop:Y:heza} still applies in this case (we only changed our space $Y$ in the Setup~\ref{setup:perm:wreath:products}) and so we have that for each $i \in \N$ the standard Borel probability space $Y^{X_i}$ is a $\heza$-space for $\Gamma_i \wr_{X_i} A_i$. 

We consider now the standard Borel probability space $\prod_{i \in \N} Y^{X_i}$ endowed with the product measure. We let $G$ act on it coordinate-wise. By Proposition~\ref{prop:zimmer:ame:spaces}.\emph{(2)} this action is still non-singular and Zimmer-amenable. We are left to prove that the action $G \actson \prod_{i \in \N} Y^{X_i}$ is also highly ergodic.

Let us consider the subgroup $\bigoplus_{i \in \N} A_i \leq G$. This group 
by construction acts on $\prod_{i \in \N} Y^{X_i} = Y^{\sqcup_{i \in \N} 
X_i}$ by the shifting coordinate-wise action $\bigoplus_{i \in \N} A_i 
\actson \sqcup_{i \in \N}X_i$. Hence by hypothesis of 
Setup~\ref{setup:perm:wreath:products} all the orbits are 
infinite (since the orbits of each action $A_i \actson X_i$ are so), the 
group $\bigoplus_{i \in \N} A_i$ is countable and the set $\sqcup_{i 
\in \N} X_i$ is as well. According to Proposition \ref{prop:KechrisTsankov:ergodic}, this new generalised 
Bernoulli shift is highly ergodic. 
Since $\bigoplus_{i \in \N} 
A_i$ is a subgroup of $G$, Remark~\ref{rem:ergodicity:subgroups} shows 
that the action of $G$ on $\prod_{i \in \N} Y^{X_i}$ is also highly ergodic. By Corollary \ref{cor:zimmer:ergodic}, $G$ is $\Xsep$-boundedly acyclic, whence the desired contradiction.
\end{proof}

\section{Commuting cyclic conjugates and bounded acyclicity}\label{sec:ccc:bac}

In this section we introduce our main algebraic criterion of admitting \emph{commuting cyclic conjugates} and we show that groups with commuting cyclic conjugates are $\Xsep$-boundedly acyclic. 

\subsection{Commuting cyclic conjugates and $\Xsep$-bounded acyclicity}
\label{subsec:main:thm}

We recall the definition of commuting cyclic conjugates from the introduction:

\begin{defi}[commuting cyclic conjugates]\label{def:comm:cycl:conj}
A group $\Gamma$ has \emph{commuting cyclic conjugates} if for every finitely generated subgroup $H \leq \Gamma$ there exist $t \in \Gamma$ and $n \in \Z_{\geq 2} \cup \{\infty\}$ such that $[H, {}^{t^p}H] = 1$ for $1 \leq p < n$ and $[H, t^n] = 1$.\ Here we read that $t^\infty = 1$ and ${}^{t^p}H$ denotes the conjugation $t^p H t^{-p}$.
\end{defi}

It will also be useful to treat the special situation in which $n = \infty$.

\begin{defi}[commuting $\Z$-conjugates]\label{def:comm:Z:conj}
A group has \emph{commuting $\Z$-conjugates} if for every finitely generated subgroup $H \leq \Gamma$ there exists $t \in \Gamma$ such that $[H, {}^{t^p}H] = 1$ for all integers $p \geq 1$.
\end{defi}

\begin{rem}[Order of $t$]
    In the definitions above, if the order of  $t$ is less than or equal to $n-1$ (or finite when $n = \infty$), then the group $H$ has to be abelian.
\end{rem}

Our main result is the following more precise version of Theorem \ref{intro:thm:ccc}:

\begin{thm}[commuting cyclic conjugates $\Rightarrow$ $\Xsep$-boundedly acyclic]
\label{thm:cac}
The following holds:
\begin{enumerate}
\item If $\Gamma$ is a group with commuting cyclic conjugates, then it is $\Xsep$-boundedly acyclic. 

\item For every integer $n \geq 1$, there exists a constant $K_n \in [0, \infty)$ such that for every group $\Gamma$ with commuting cyclic conjugates, and every separable dual Banach $\Gamma$-module $E$, we have that $\| H^n_b(\Gamma; E) \| \leq K_n$. Hence, the vanishing modulus is uniform.
\end{enumerate}
\end{thm}

Our result is built on Proposition \ref{prop:vanishingmodulus:wreathproducts} and so we have to rephrase the definition of commuting cyclic conjugates in terms of morphisms from wreath products. This is the content of the following technical proposition:

\begin{prop}[commuting cyclic conjugates and wreath products]
\label{lem:cac:wreathproduct}
Let $\Gamma$ be a group. Then the following are equivalent:
\begin{enumerate}
    \item $\Gamma$ has commuting cyclic conjugates;
    \item For every finitely generated subgroup $H \leq \Gamma$ there exist a countable amenable group $A$, an infinite-index subgroup $B < A$, and a homomorphism $f \colon A \to \Gamma$ such that
\begin{itemize}
    \item[(i)] $[H, {}^{f(a)}H] = 1$ for all $a \in A \setminus B$;
    \item[(ii)] $[H, f(b)] = 1$ for all $b \in B$;
\end{itemize}
    \item For every finitely generated subgroup $H \leq \Gamma$ there exist a countable amenable group $A$, an infinite transitive $A$-set $X$ and a homomorphism $$f\colon H \wr_X A \to \Gamma$$ such that $f$ restricts to the inclusion on some factor $H$, and $f$ has amenable kernel.
\end{enumerate}

\end{prop}

We postpone the proof of Proposition~\ref{lem:cac:wreathproduct} to the next Section~\ref{sec:proof:technical:prop} and we give here the proof of Theorem~\ref{thm:cac} (and thus of Theorem~\ref{intro:thm:ccc}):

\begin{proof}[Proof of Theorem \ref{thm:cac}]
Let $\Gamma$ be a group with commuting cyclic conjugates, and let $n \geq 1$. Let $H \leq \Gamma$ be a finitely generated subgroup. By Proposition~\ref{lem:cac:wreathproduct} there exist a countable amenable group $A$, an infinite transitive $A$-set $X$ (which is necessarily countable) and a morphism $f \colon H \wr_X A \to \Gamma$ that restricts to the inclusion on some factor $H$. In particular, we have that $H \leq \im(f)$. Again by Proposition~~\ref{lem:cac:wreathproduct} we can assume that the kernel of $f$ is amenable. Let $\Xsep$ be the class of separable dual Banach $\Gamma$-modules and let $i^*(\Xsep)$ be the class of separable dual Banach $\im(f)$-modules induced by the inclusion $i \colon \im(f) \to \Gamma$.

We are now ready to apply Proposition~\ref{prop:vanishingmodulus:wreathproducts}. Indeed, $H$ is finitely generated and thus countable, $X$ is countably infinite and the amenable group $A$ acts on $X$ with a (unique) infinite orbit. Hence, for every integer $n \geq 1$ there exists a constant $K_n \in [0, \infty)$ which does not depend on $H, A, X$ or even $\Gamma$, such that $\| H^n_b(\im(f); E) \| \leq K_n$ for every separable dual Banach $\im(f)$-module $E$.

Every finitely generated $H \leq \Gamma$ is contained in a group of the form $\im(f)$ as above. Hence, applying Proposition \ref{prop:dirun} with $\Gamma_i = \im(f)$, we conclude that $H^n_b(\Gamma; E) = 0$.
\end{proof}

\subsection{Proof of Proposition~\ref{lem:cac:wreathproduct}}\label{sec:proof:technical:prop}

We prove the chain of implications \emph{(1) $\Rightarrow$ (2) $\Rightarrow$ (3) $\Rightarrow$ (1)}. The first implication is the crucial one.

\begin{proof}[Proof of Proposition~\ref{lem:cac:wreathproduct} (1) $\Rightarrow$ (2)]
Suppose that $\Gamma$ has commuting cyclic conjugates, and let $H \leq \Gamma$ be finitely generated. Set $\Lambda_0 \coloneqq H$. Let $t_1 \in \Gamma$ and $n_1 \in \Z_{\geq 2} \cup \{ \infty \}$ be such that 
\begin{itemize}
    \item $[H, {}^{t_1^p}H]=1$ for all $1 \leq p < n_1$; and
    \item $[H, t_1^{n_1}] = 1$.
\end{itemize}
Define $\Lambda_1 := \langle H, t_1 \rangle$. We proceed by induction. At the $i$-th step, $i \in \N$, we have a finitely generated group $\Lambda_{i-1}$, and we find $t_{i} \in \Gamma$ and $n_{i} \in \Z_{\geq 2} \cup \{ \infty \}$ such that 
\begin{itemize}
    \item $[\Lambda_{i-1}, {}^{t_i^p}\Lambda_{i-1}]=1$ for all $1 \leq p <n_i$; and
    \item $[\Lambda_{i-1}, t_{i}^{n_{i}}] = 1$.
\end{itemize}
We set $\Lambda_{i} := \langle \Lambda_{i-1}, t_{i} \rangle$.

\medskip

Suppose first that $n_i = \infty$ for some $i \in \N$. In that case, for all integers $p \geq 1$ we have 
\[
[H, {}^{t_i^p}H] \leq [\Lambda_{i-1}, {}^{t_i^p}\Lambda_{i-1}] = 1,
\] 
and moreover for all integers $p \leq -1$ we have 
\[
{}^{t_i^{-p}}[H, {}^{t_i^p}H] = [{}^{t_i^{-p}} H, H] = [H, {}^{t_i^{-p}}H]^{-1} = 1 \Longrightarrow [H, {}^{t_i^p}H] = 1.
\] 
Thus we can choose $A = \Z, B = \{ 0 \}$ and a homomorphism
\begin{align*}
    f \colon A = \Z &\to \Gamma \\
    m &\mapsto t_i^m,
\end{align*}
which satisfies the conditions required in Item \emph{(2)}.\ Note how this already suffices to show that groups with commuting $\Z$-conjugates satisfy condition \emph{(2)}. From now on, we may assume that $n_i \in \Z_{\geq 2}$ for all integers $i \geq 1$.

\medskip

Set $A_1 := \Z$, and define by induction 
\[A_{i+1} \coloneqq A_i \wr_{\Z/n_{i+1}} \Z,\] 
where the action of $\Z$ on $\Z/n_{i+1}$ is by left translation. By construction, each $A_i$ is amenable being an extension of amenable groups. Consider the inclusion of $A_i$ into $A_{i+1}$, concentrated at the coordinate $0 \in \Z/n_{i+1}$. Let $A$ be the direct union of the $A_i$. Hence, the group $A$ is countable and amenable, since it is a direct union of countable amenable groups.

Let $B_1 \coloneqq n_1 \Z \leq A_1$, and define by induction $B_{i+1}$ to be the subgroup of $A_{i+1}$ consisting of elements of the form $((a_p)_{p \in \Z/n_{i+1}}, m) \in A_i \wr_{\Z/n_{i+1}} \Z$, where $a_0 \in B_i$ and $m \in n_{i+1}\Z$. This is indeed a group, since the $\Z$ coordinate only consists of multiples of $n_{i+1}$, which act trivially on $\Z/n_{i+1}$. Denote by $B$ the direct union of the $B_i$. By construction $B < A$.

To see that the subgroup $B < A$ has infinite index, consider the $A$-conjugates of $A_1$. First note that there are exactly $n_i$ conjugates of $A_{i-1}$ inside $A_i$, which have pairwise trivial intersection: these are the copies of $A_{i-1}$ concentrated at the different coordinates of $\Z/n_i$. It follows that the number of $A_i$-conjugates of $A_1$ is at least $n_2 \cdots n_i$. Since by definition of commuting cyclic conjugates each $n_i \geq 2$ (Definition~\ref{def:comm:cycl:conj}), there are infinitely many $A$-conjugates of $A_1$. On the other hand, let us show that $B$ centralizes $A_1$. This is clear for $B_1$. Recall that an element of $B_{i+1}$ is of the form $((a_p)_{p \in \Z/n_{i+1}}, m)$ where $a_0 \in B_i$ and $m \in n_{i+1}\Z$. Then $a_0$ centralizes $A_1$ by induction hypothesis, and $a_p$ with $p \neq 0$ as well as $m$ centralize $A_1$ by definition of the permutational wreath product. Thus all of $B_{i+1}$ centralizes $A_1$ and so also their direct union $B$ centralizes $A_1$. 
In conclusion we have proved that the conjugation action of $A$ on the set of its subgroups has the following properties:
\begin{enumerate}
    \item The $A$-conjugates of $A_1$ are infinite (i.e. the orbit ${}^{A}A_1$ is infinite);
    
    \item $B$ centralizes $A_1$ (i.e. $B$ is a subgroup of the stabilizer $\textup{Stab}_A(A_1)$).
\end{enumerate}
The orbit-stabilizer theorem then shows that
\[
\left[ A : B\right] \geq  \left[ A : \textup{Stab}_A(A_1)\right] = |{}^{A} A_1|  = +\infty,
\]
and thus $B$ has infinite index in $A$.

We are left to define the desired homomorphism $f \colon A \to \Gamma$ as required in condition~\emph{(2)}. We also construct the homomorphism $f$ by induction. Let us begin by setting 
\begin{align*}
    f \colon A_1 = \Z &\to \Gamma \\
    m &\mapsto t_1^m.
\end{align*}
By definition, we have $f(A_1) \leq \Lambda_1$. Suppose now that $f$ has already been defined on $A_i$ (and is such that $f(A_i) \leq \Lambda_{i}$), and set
\begin{align*}
   f \colon A_{i+1} = A_i \wr_{\Z/n_{i+1}} \Z &\to \Gamma \\
   ((a_p)_{p \in \Z/n_{i+1}}, m) &\mapsto \left( \prod\limits_{p = 0}^{n_{i+1}-1} {}^{t_{i+1}^p}f(a_p) \right) t_{i+1}^m. 
\end{align*}
Let us check that this is a homomorphism. By induction it is a homomorphism on each summand $A_i$. Recall that $\langle t_{i+1}^{n_{i+1}} \rangle$ centralizes $\Lambda_i$ and therefore $f(A_i)$ (since $f(A_i) \leq \Lambda_i$ by induction). Recall also that $[f(A_i), {}^{t_{i+1}^p}f(A_i)] \leq [\Lambda_i, {}^{t_{i+1}^p}\Lambda_i] = 1$ for all $1 \leq p < n_{i+1}$. Applying a suitable conjugation, we see that the subgroups ${}^{t_{i+1}^p}f(A_i)$ pairwise commute, hence $f$ is a homomorphism on the direct sum of the $A_i$. Finally, the action of $\Z$ by conjugacy on the summands is respected, by construction and because $t_{i+1}^{n_{i+1}}$ centralizes each summand.

We now need to check the two required properties in condition~\emph{(2)}:
\begin{enumerate}
    \item[$(i)$] $[H, {}^{f(a)} H] = 1$ for all $a \in A \setminus B$; and
    \item[$(ii)$] $[H, f(b)] = 1$ for all $b \in B$.
\end{enumerate}
The property \emph{(ii)} is verified by induction on the $B_i$, with the same argument as the proof above that $B$ centralizes $A_1$ together with the definition of $f$.

Let us prove that property \emph{(i)} is also satisfied. We proceed again by induction, by showing that if $a \in A_i$ is such that $[H, {}^{f(a)} H] \neq 1$, then $a \in B_i$. For $i = 1$, this is by definition of $t_1$. Now suppose that $a \in A_{i+1}$ is an element such that $[H, {}^{f(a)} H] \neq 1$. Using the definition of $A_{i+1} = A_i \wr_{\Z/n_{i+1}} \Z$, we write $a = ((a_p)_{p \in \Z/n_{i+1}}, m)$. Now for $1 \leq p < n_{i+1}$ the image of $a_p$ lies in ${}^{t_{i+1}^p}f(A_i) \leq {}^{t_{i+1}^p}\Lambda_i$, which by choice of $t_{i+1}$ commutes with $\Lambda_i$, and therefore with $H$. In other words, conjugating by these coordinates has no effect on $H$, and therefore ${}^{f(a)^{-1}} H = {}^{t_{i+1}^{-m}}({}^{f(a_0)^{-1}}H)$. Suppose first that $m \notin n_{i+1} \Z$. Then ${}^{t_{i+1}^{-m}}({}^{f(a_0)^{-1}}H) \leq {}^{t_{i+1}^{-m}}\Lambda_i$, which commutes with $\Lambda_i$ and thus with $H$. A standard manipulation shows then that ${}^{f(a)} H$ also commutes with $H$. Since we are assuming that $[H, {}^{f(a)} H] \neq 1$, we get a contradiction. Therefore we must have $m \in n_{i+1} \Z$, which implies ${}^{t_{i+1}^{-m}}({}^{f(a_0)^{-1}}H) = {}^{f(a_0)^{-1}}H$, since $t_{i+1}^{n_{i+1}}$ centralizes $\Lambda_i$. Since we are assuming $[H, {}^{f(a)}H] \neq 1$, or equivalently $[H, {}^{f(a)^{-1}}H] \neq 1$, we obtain $[H, {}^{f(a_0)^{-1}}H] \neq 1$ and so by induction hypothesis $a_0$ actually lies in $B_i$. We have shown that $a = ((a_p)_{p \in \Z/n_{i+1}}, m)$ is such that $a_0 \in B_i$ and $m \in n_{i+1}\Z$. This shows that $a \in B_{i+1}$ by definition, and concludes the proof.
\end{proof}

\begin{proof}[Proof of Proposition~\ref{lem:cac:wreathproduct} (2) $\Rightarrow$ (3)]

Let $H \leq \Gamma$ be a finitely generated subgroup, and suppose that there exists a countable amenable group $A$, a morphism $f \colon A \to \Gamma$ and an infinite index subgroup $B < A$ as in condition \emph{(2)}. Note that the property \emph{(2)}.(i) implies that $[^{f(a_1)} H, ^{f(a_2)} H] = 1$ whenever $a_1B \neq a_2B$. Define $X := A/B$, with the left translation action of $A$. By the assumptions of \emph{(2)}, $X$ is a transitive infinite $A$-set. We extend $f$ to a homomorphism $f \colon H \wr_X A \to \Gamma$ as follows:
\[f((h_{aB})_{aB \in X}, g) = \left( \prod\limits_{aB \in X} {}^{f(a)}h_{aB} \right) f(g). \]
Since $f(B)$ centralizes $H$, the value of $^{f(a)}h_{aB}$ is independent of the choice of $a$. Moreover, since all of the $^{f(a)}H$ pairwise commute, when $a$ runs over a system of representatives, the product is well-defined independently of the order of multiplication. Finally, the action by conjugacy of $A$ on the $H$ summands corresponds equivariantly to the action by conjugacy of $f(A)$ on the commuting groups ${}^{f(a)}H : aB \in X$. This implies that $f$ is indeed a homomorphism, and by construction it restricts to the inclusion of $H$ on the factor $H_{1B}$. This proves the first part of the claim.

We are left to show that $f$ has amenable kernel.\ A proof in the case $A = X = \Z$ is already available in the literature~\cite[Lemma 5.4]{thompson:ulam}. For completeness, we show how to extend it. Let $f \colon H \wr_X A \to \Gamma$ be as above, and let $x_0 \in X$ be such that $f|_{H_{x_0}}$ is the inclusion of $H$. Note that the structure of the wreath product implies that $\ker(f)$ can be described as the following extension:
 \[
 1 \to K := \ker(f) \cap \bigoplus_{x \in X} H \to \ker(f) \to Q \to 1,
 \]
where $Q$ is the subgroup of $A$ corresponding to the projection of $\ker(f)$ to $A$. It is sufficient to show that $K$ is amenable; in fact we will show that it is abelian.
 
To see this, let $g, h \in K$, which we write as $g = (g_x)_{x \in X}, h = (h_x)_{x \in X}$ (we omit the $A$-component since it is always trivial). Then
\[1 = f(g) = \prod\limits_{x \in X} f(g_x) \quad \Longrightarrow \quad f(g_{x_0})^{-1} = \prod\limits_{x \in X \setminus \{ x_0 \}} f(g_x).\]
Since $f(H_x)$ commutes with $f(H_{x_0})$ for all $x \in X \setminus \{ x_0 \}$, from the above expression it follows in particular that $f(g_{x_0})$ commutes with $f(h_{x_0})$. But $f|_{H_{x_0}}$ is injective by hypothesis, therefore $g_{x_0}$ and $h_{x_0}$ commute. Up to conjugating by a suitable element of $A$, we obtain that $g_x$ commutes with $h_x$ for all $x \in X$, and therefore $g$ and $h$ commute.
\end{proof}

\begin{proof}[Proof of Proposition~\ref{lem:cac:wreathproduct} (3) $\Rightarrow$ (1)]

Suppose that $\Gamma$ is a group satisfying condition \emph{(3)}. Let $H \leq \Gamma$ be a finitely generated group and let $f\colon H \wr_X A \to \Gamma$ be as required by condition \emph{(3)}. Let $x_0 \in X$ be such that $f$ restricts to the identity on the $x_0$-coordinate. We set $t \coloneqq f(a)$ for some $a \in A$ such that $a \cdot x_0 \neq x_0$. Let $n \geq 2$ be the minimal integer such that $a^n \cdot x_0 = x_0$ (possibly $n = \infty$). Then for all $1 \leq p < n$ we have $[H, {}^{t^p}H] = [f(H_{x_0}), {}^{f(a^p)}f(H_{x_0})] = f([H_{x_0}, H_{a^p \cdot x_0}]) = 1$, because $a^p \cdot x_0 \neq x_0$; and $[H, t^n] = f([H_{x_0}, a^n]) = 1$ since $a^n \cdot x_0 = x_0$. Hence, $\Gamma$ has commuting cyclic conjugates.
\end{proof}

\subsection{Constructions}

A classical question in bounded cohomology with real coefficients is whether surjective group homomorphisms induce injective maps in bounded cohomology~\cite{Bouarich:exact, amcat, fflm1, fflm2}.\ Since in this paper we are working with $\Xsep$-boundedly acyclic groups, a weaker formulation of the previous problem is whether quotients of $\Xsep$-boundedly acyclic groups are $\Xsep$-boundedly acyclic.\ For this reason, now that we provided a new criterion for bounded acyclicity, it is interesting to note that it does pass to quotients.\ This shows that quotients of groups with commuting cyclic conjugates are $\Xsep$-boundedly acyclic:

\begin{lemma}[Quotients]
\label{lemma:quotient:has:ccc}
A quotient of a group with commuting cyclic conjugates has commuting cyclic conjugates.\ A quotient of a group with commuting $\Z$-conjugates has commuting $\Z$-conjugates.
\end{lemma}

\begin{proof}
Let $\Gamma$ be a group, and let $\pi \colon \Gamma \to \overline{\Gamma}$ be a quotient. For a finitely generated subgroup $\overline{H} \leq \overline{\Gamma}$, choose a preimage for each of the generators to obtain a finitely generated subgroup $H \leq \Gamma$ such that $\pi(H) = \overline{H}$. Since $\Gamma$ has commuting cyclic conjugates, there exist $t \in \Gamma$ and $n \in \Z_{\geq 2} \cup \{ \infty \}$ such that $[H, {}^{t^p}H] = 1$ for $1 \leq p < n$ and $[H, t^n] = 1$. Then by definition of homomorphism we have 
\[
[\overline{H}, {}^{\pi(t)^p}\overline{H}] = \pi([H, {}^{t^p}H]) = 1
\]
for all $1 \leq p < n$; and 
\[
[\overline{H}, \pi(t)^n] = \pi([H, t^n]) = 1.\] 
Thus $\overline{H}$ satisfies the condition of commuting cyclic conjugates with $\pi(t)$ and $n$. Note that $n$ is unchanged, so this also proves that if $\Gamma$ has commuting $\Z$-conjugates, then $\overline{\Gamma}$ also has commuting $\Z$-conjugates.
\end{proof}

\begin{rem}[Groups with commuting conjugates but not commuting cyclic conjugates]
    Using this fact, one can easily construct groups with commuting conjugates but without commuting \emph{cyclic} conjugates~\cite[Proposition 3.6]{companion}.
\end{rem}

Extensions of $\Xsep$-boundedly acyclic groups can be shown to be $\Xsep$-boundedly acyclic, either using spectral sequences \cite[Section 12]{monod}, or generalised mapping theorems~\cite{moraschini_raptis_2023}. For groups with commuting cyclic conjugates, we can show more:

\begin{lemma}[Products]
\label{lem:products:ccc}

Let $\{ \Gamma_i \}_{i \in I}$ be a collection of groups with commuting $\Z$-conjugates. Then their direct product and direct sum have commuting $\Z$-conjugates.
\end{lemma}

\begin{proof}
Suppose that the $\Gamma_i$ have commuting $\Z$-conjugates. Let $\Gamma \coloneqq \prod_{i \in I} \Gamma_i$, and let $H \leq \Gamma$ be finitely generated. Let $H_i \leq \Gamma_i$ be the projection of $H$ on $\Gamma_i$, which is finitely generated. Since $\Gamma_i$ has commuting $\Z$-conjugates, there exists $t_i \in \Gamma_i$ such that $[H_i, {}^{t_i^p} H_i] = 1$ for all $p \geq 1$. Then $t = (t_i)_{i \in I} \in \Gamma$ satisfies $[H, {}^{t^p} H] = 1$ for all $p \geq 1$.

For direct sums, there exists a finite index set $J \subset I$ such that $H_i$ is trivial for all $i \in I \setminus J$. So we can choose $t_i = 1$ for all $i \in I \setminus J$, and the resulting element $t$ belongs to the direct sum.
\end{proof}

\begin{rem}[Products]
Note that the proof does not work for groups with commuting cyclic conjugates, as the numbers $n$ appearing in the definition may be different for different factors. In fact, the product of two groups with commuting cyclic conjugates need not have commuting cyclic conjugates \cite[Proposition 3.6]{companion}.
\end{rem}

We use this to prove an extension of the uniformity of the vanishing modulus from Theorem \ref{thm:cac}, which will be useful in the sequel:

\begin{cor}[Uniform vanishing modulus for amenable extensions]
\label{cor:coameanble:modulus}

For every integer $n \geq 1$, there exists a constant $K_n \in [0, \infty)$ such that for every group $\Gamma$ containing a normal subgroup $\Lambda < \Gamma$ such that $\Gamma / \Lambda$ is amenable and $\Lambda$ has commuting $\Z$-conjugates, and for every separable dual Banach $\Gamma$-module $E$, we have that $\| H^n_b(\Gamma; E) \| \leq K_n$.
\end{cor}

\begin{proof}
Suppose by contradiction that this is not the case. Then there exists a sequence $(\Lambda_i, \Gamma_i, E_i)_{i \in \N}$ such that $\Lambda_i < \Gamma_i$ is normal with amenable quotient, $\Lambda_i$ has commuting $\Z$-conjugates, and $E_i$ is a separable dual Banach $\Gamma_i$-module such that $\| H^n_b(\Gamma_i; E_i) \| \geq i$. Let $\Gamma \coloneqq \bigoplus_{i \in \N} \Gamma_i > \Lambda \coloneqq \bigoplus_{i \in \N} \Lambda_i$; note that $\Lambda$ is normal in $\Gamma$ with amenable quotient. Now $\Lambda$ has commuting $\Z$-conjugates by Lemma \ref{lem:products:ccc}, thus $\Lambda$ is $\Xsep$-boundedly acyclic by Theorem \ref{intro:thm:ccc}, and therefore so is $\Gamma$ by Theorem \ref{thm:co-amenable}. Thus we obtain a contradiction with Proposition \ref{prop:dirsum} and conclude.
\end{proof}

We end by noticing that the commuting cyclic conjugates condition passes to derived subgroups.\ Recall that for a group $\Gamma$, we denote $\Gamma^{(1)} \coloneqq [\Gamma, \Gamma]$ and $\Gamma^{(d)} \coloneqq [\Gamma^{(d-1)}, \Gamma^{(d-1)}]$ for all $d > 1$.

\begin{lemma}[Derived subgroups]
\label{lem:derived}
    Let $\Gamma$ be a group with commuting cyclic conjugates, and let $d \geq 1$.\ Then the $d$-th derived subgroup $\Gamma^{(d)}$ has commuting cyclic conjugates.
\end{lemma}

\begin{proof}
    It suffices to show that $[\Gamma, \Gamma]$ has commuting cyclic conjugates, the statement then follows by induction.\ Let $H \leq [\Gamma, \Gamma]$ be a finitely generated subgroup.\ Then there exists $t \in \Gamma$ and $n \in \Z_{\geq 2} \cup \{ \infty \}$ such that $[H, {}^{t^p} H] = 1$ for $1 \leq p < n$ and $[H, t^n] = 1$.\ Now the group $K \coloneqq \langle H, t \rangle$ is also finitely generated, so there exists $s \in \Gamma$ such that $[K, {}^s K] = 1$.\ Let $c \coloneqq [t, s] = t \cdot {}^s t^{-1}$.\ Then, for all $g \in K$, it holds ${}^c g = {}^{t}({}^{{}^s t^{-1}} g) = {}^t g$, and so by induction ${}^{c^p} g = {}^{t^p} g$ for all $p \geq 1$.\ It follows that $[H, {}^{c^p} H] = [H, {}^{t^p} H] = 1$ for all $1 \leq p < n$, and $[H, c^n] = [H, t^n] = 1$. Since $c \in [\Gamma, \Gamma]$, we conclude.
\end{proof}

\section{Applications}\label{sec:ex}

In this section, we show that many groups of interest in geometry, dynamics, and algebra have commuting cyclic conjugates. By virtue of Theorem~\ref{intro:thm:ccc}, they will be all $\Xsep$-boundedly acyclic. Combining these with some structural considerations we obtain further examples of $\Xsep$-boundedly acyclic groups. Lemma \ref{lem:derived} allows to extend some of our examples to derived subgroups; however in many occasions the groups in question are known to be perfect.

\subsection{Stable mapping class group}\label{subsec:smcg}

First we show that the \emph{stable mapping class group} $\Gamma_\infty=\cup_{g\geq 1}\Gamma_g^1$ has commuting cyclic conjugates and so it is $\Xsep$-boundedly acyclic, proving a conjecture by Bowden~\cite{bowden}.

\begin{cor}[Corollary \ref{corintro:stablemcg}]
\label{cor:stablemcg}

   The stable mapping class group~$\Gamma_\infty$ is $\Xsep$-boundedly acyclic. 
\end{cor}

Note that $\Gamma_\infty$ is perfect \cite[Theorem 5.2]{primer}.

\begin{proof}
    We adapt Kotschick's construction~\cite[Theorem 3.1]{kotschick} for the vanishing of the stable commutator length in order to show that $\Gamma_\infty$ has commuting cyclic conjugates. Let $H \leq \Gamma_\infty$ be a finitely generated subgroup. This is supported on the image of some $\Gamma_g^1$, that is the group of isotopy classes of diffeomorphisms of $\Sigma_g^1$ with compact support in the interior of $\Sigma_g^1$. As explained in Subsection~\ref{subsect:stable:mapping:class:group}, the stable mapping class group $\Gamma_\infty=\cup_{g\geq 1}\Gamma_g^1$ is the colimit along the inclusion $\Gamma_g^1 \hookrightarrow \Gamma_{g+1}^1$ given by adding a two-holed torus to $\Sigma_g^1$ and extending the elements of $\Gamma_g^1$ via the identity. Since each element $\gamma \in \Gamma_g^1$ has support in the interior of $\Sigma_g^1$, we can realize the previous inclusion $\Gamma_g^1 \hookrightarrow \Gamma_{2g}^1$ as follows. Let us consider another surface $S_g^1$ of genus $g$ and one boundary component and let $\Sigma_{2g}^1$ be the boundary connected sum of $\Sigma_g^1$ and $S_g^1$ (Figures \ref{fig:boundary connected sum} and \ref{fig:element t}). Then each element $\gamma \in \Gamma_g^1$, that is the identity on a small neighbourhood of the boundary of $\Sigma_g^1$, can be extended to the whole $\Sigma_{2g}^1$ as the identity providing the same inclusion of $\Gamma_g^1 \hookrightarrow \Gamma_{2g}^1$ as above. In particular, since $H \leq \Gamma_g^1$ we can view each element $h \in H$ as an element in $\Gamma_{2g}^1 \leq \Gamma_\infty$ extended as the identity over the surface $S_g^1$.

\begin{figure}[htbp]
\centering
\def\svgscale{0.4}
\begingroup%
  \makeatletter%
  \providecommand\color[2][]{%
    \errmessage{(Inkscape) Color is used for the text in Inkscape, but the package 'color.sty' is not loaded}%
    \renewcommand\color[2][]{}%
  }%
  \providecommand\transparent[1]{%
    \errmessage{(Inkscape) Transparency is used (non-zero) for the text in Inkscape, but the package 'transparent.sty' is not loaded}%
    \renewcommand\transparent[1]{}%
  }%
  \providecommand\rotatebox[2]{#2}%
  \newcommand*\fsize{\dimexpr\f@size pt\relax}%
  \newcommand*\lineheight[1]{\fontsize{\fsize}{#1\fsize}\selectfont}%
  \ifx\svgwidth\undefined%
    \setlength{\unitlength}{684bp}%
    \ifx\svgscale\undefined%
      \relax%
    \else%
      \setlength{\unitlength}{\unitlength * \real{\svgscale}}%
    \fi%
  \else%
    \setlength{\unitlength}{\svgwidth}%
  \fi%
  \global\let\svgwidth\undefined%
  \global\let\svgscale\undefined%
  \makeatother%
  \begin{picture}(1,0.42105263)%
    \lineheight{1}%
    \setlength\tabcolsep{0pt}%
    \put(0,0){\includegraphics[width=\unitlength,page=1]{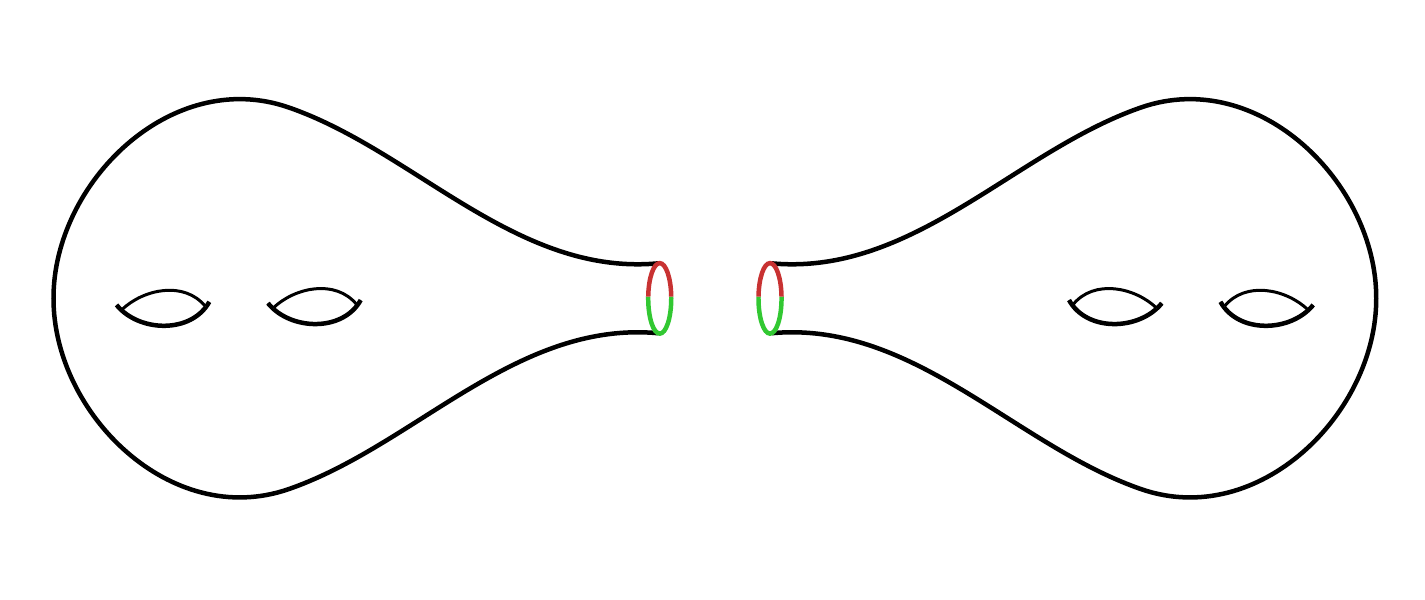}}%
    \put(0.40533293,0.1000388){\color[rgb]{0,0,0}\makebox(0,0)[lt]{\lineheight{1.25}\smash{\begin{tabular}[t]{l}$\Sigma^1_2$\end{tabular}}}}%
    \put(0.55196936,0.1000388){\color[rgb]{0,0,0}\makebox(0,0)[lt]{\lineheight{1.25}\smash{\begin{tabular}[t]{l}$S^1_2$\end{tabular}}}}%
  \end{picture}%
\endgroup%

\caption{Boundary connected sum of $\Sigma_2^1$ and $S_2^1$ along the lower half of their boundaries.}
\label{fig:boundary connected sum}
\end{figure}

We claim that there exists an element $t \in \Gamma_{2g}^1 \leq \Gamma_\infty$ such that 
\begin{enumerate}
    \item $[H, {}^t H] = 1$;
    \item $[H, t^2] = 1$.
\end{enumerate}

Since every element $h$ in $H \leq \Gamma_g^1$ fixes the boundary, the collar neighbourhood theorem ensures that one can find a small enough neighbourhood $S^1 \times [0, 1)$ of the boundary of $\Sigma_{2g}^1$ such that the support of every $h \in H$ is contained in $\Sigma_g^1 \cap (\Sigma_{2g}^1 \setminus (S^1 \times [0, 1)))$. Let now $t \in \Gamma_{2g}^1$ be the half turn around the curve $S^1\times \{3/4\}$ over $\Sigma_{2g}^1 \setminus (S^1 \times [0, 1/2])$ and extend it to the rest of $\Sigma_{2g}^1$ by fixing a neighbourhood $S^1\times [0,1/4]$ of its boundary (Figure \ref{fig:element t}): This $t$ thus acts as a central symmetry below the curve $S^1\times \{3/4\}$ and rigidly exchanges $\Sigma^1_g\cap (\Sigma_{2g}^1 \setminus (S^1 \times [0, 1)))$ and $S^1_g\cap (\Sigma_{2g}^1 \setminus (S^1 \times [0, 1)))$. By construction for every $h \in H$ the conjugate element ${}^t h$ has its compact support entirely contained in the interior of $S_g^1$ and so it commutes with $h$. This shows that $[H, {}^tH] = 1$.

\begin{figure}
\centering
\def\svgscale{0.4}
\begingroup%
  \makeatletter%
  \providecommand\color[2][]{%
    \errmessage{(Inkscape) Color is used for the text in Inkscape, but the package 'color.sty' is not loaded}%
    \renewcommand\color[2][]{}%
  }%
  \providecommand\transparent[1]{%
    \errmessage{(Inkscape) Transparency is used (non-zero) for the text in Inkscape, but the package 'transparent.sty' is not loaded}%
    \renewcommand\transparent[1]{}%
  }%
  \providecommand\rotatebox[2]{#2}%
  \newcommand*\fsize{\dimexpr\f@size pt\relax}%
  \newcommand*\lineheight[1]{\fontsize{\fsize}{#1\fsize}\selectfont}%
  \ifx\svgwidth\undefined%
    \setlength{\unitlength}{684bp}%
    \ifx\svgscale\undefined%
      \relax%
    \else%
      \setlength{\unitlength}{\unitlength * \real{\svgscale}}%
    \fi%
  \else%
    \setlength{\unitlength}{\svgwidth}%
  \fi%
  \global\let\svgwidth\undefined%
  \global\let\svgscale\undefined%
  \makeatother%
  \begin{picture}(1,0.52631579)%
    \lineheight{1}%
    \setlength\tabcolsep{0pt}%
    \put(0,0){\includegraphics[width=\unitlength,page=1]{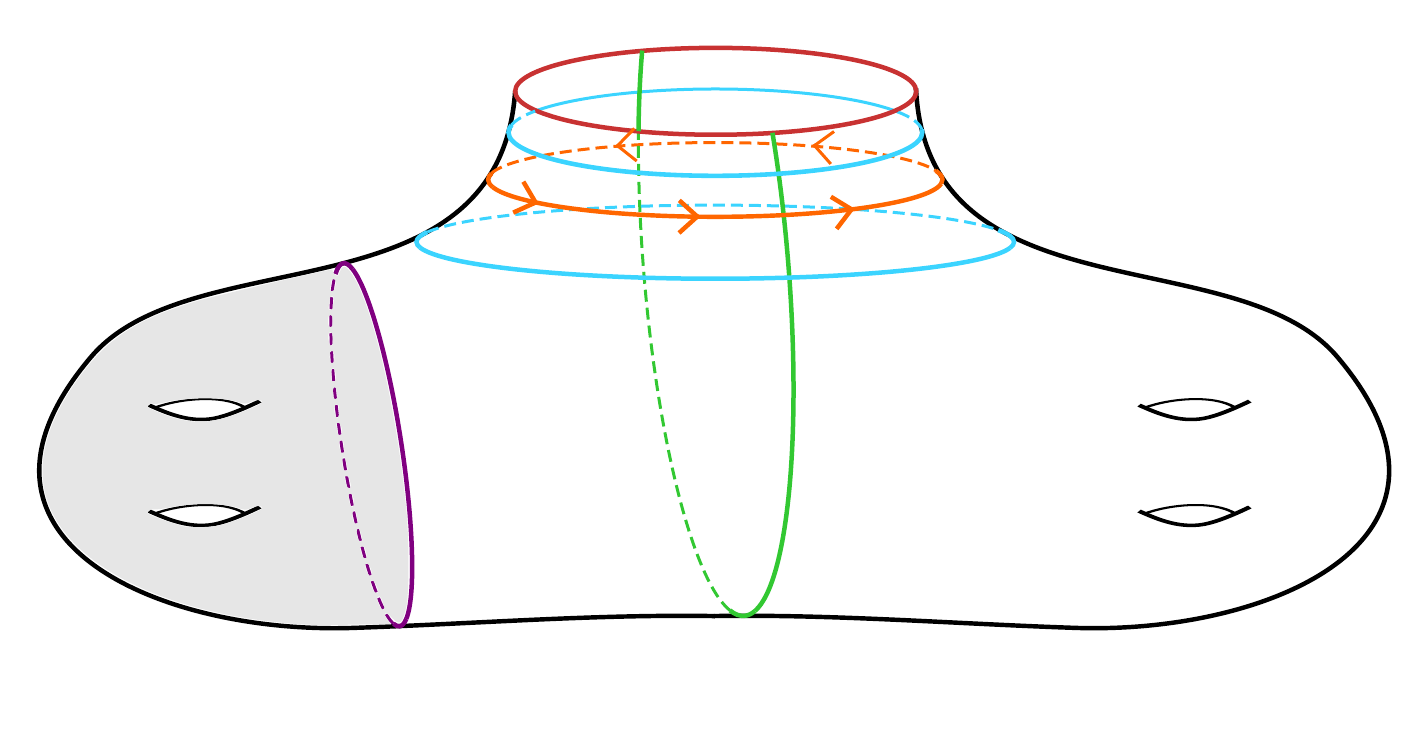}}%
    \put(0.37390351,0.03563596){\color[rgb]{0,0,0}\makebox(0,0)[lt]{\lineheight{1.25}\smash{\begin{tabular}[t]{l}$\Sigma^1_2$\end{tabular}}}}%
    \put(0.21992166,0.48436416){\color[rgb]{0,0,0}\makebox(0,0)[lt]{\lineheight{1.25}\smash{\begin{tabular}[t]{l}$\Sigma^1_4$\end{tabular}}}}%
    \put(0.66337719,0.03561776){\color[rgb]{0,0,0}\makebox(0,0)[lt]{\lineheight{1.25}\smash{\begin{tabular}[t]{l}$S^1_2$\end{tabular}}}}%
    \put(0.49991066,0.34409871){\color[rgb]{0,0,0}\makebox(0,0)[lt]{\lineheight{1.25}\smash{\begin{tabular}[t]{l}$t$\end{tabular}}}}%
    \put(0.11555464,0.03561776){\color[rgb]{0,0,0}\makebox(0,0)[lt]{\lineheight{1.25}\smash{\begin{tabular}[t]{l}$\supp(h)$\end{tabular}}}}%
    \put(0.65894244,0.42971365){\color[rgb]{0,0,0}\makebox(0,0)[lt]{\lineheight{1.25}\smash{\begin{tabular}[t]{l}$S^1 \times \left\{\frac{1}{4}\right\}$\end{tabular}}}}%
    \put(0.15874498,0.40146425){\color[rgb]{0,0,0}\makebox(0,0)[lt]{\lineheight{1.25}\smash{\begin{tabular}[t]{l}$S^1 \times \left\{\frac{3}{4}\right\}$\end{tabular}}}}%
  \end{picture}%
\endgroup%

\caption{The element $t\in \Gamma_{2g}^1$ is a half turn along the middle curve in the direction of the arrows.}
\label{fig:element t}
\end{figure}

On the other hand, if we consider the Dehn twist $t^2$, we obtain an element of $\Gamma_{2g}^1$ whose compact support is entirely contained in the open neighbourhood $S^1 \times (1/4, 3/4)$. Since we chose a small enough collar neighbourhood, we have that $t^2$ has support disjoint from that of every $h \in H$. Thus also condition $(2)$ is satisfied and we have shown that $\Gamma_\infty$ has commuting cyclic conjugates. By applying Theorem~\ref{intro:thm:ccc} we obtain the result.
\end{proof}

\subsection{Other compactly supported big mapping class groups}\label{subsec:big:mcg}

We studied the group $\Gamma_\infty$ as a natural direct limit of finite type mapping class groups. One could take another point of view, coming from the theory of \emph{big mapping class groups} \cite{aramayona2020big}. For a connected orientable surface $\Sigma$ of infinite type, i.e. with infinitely generated fundamental group, we denote by $\mcg(\Sigma)$ the group of orientation-preserving homeomorphisms of $\Sigma$, modulo those that are isotopic to the identity. The \emph{compactly supported} subgroup is the group $\mcgc(\Sigma)$ of mapping classes that admit a representative with compact support. With this language, the group $\Gamma_\infty$ is nothing by $\mcgc(\Sigma_{LN})$, where $\Sigma_{LN}$ is the \emph{Loch Ness monster surface} (see Figure \ref{fig:surfaces}).

\begin{figure}[htbp]
\centering
\def\svgscale{0.4}
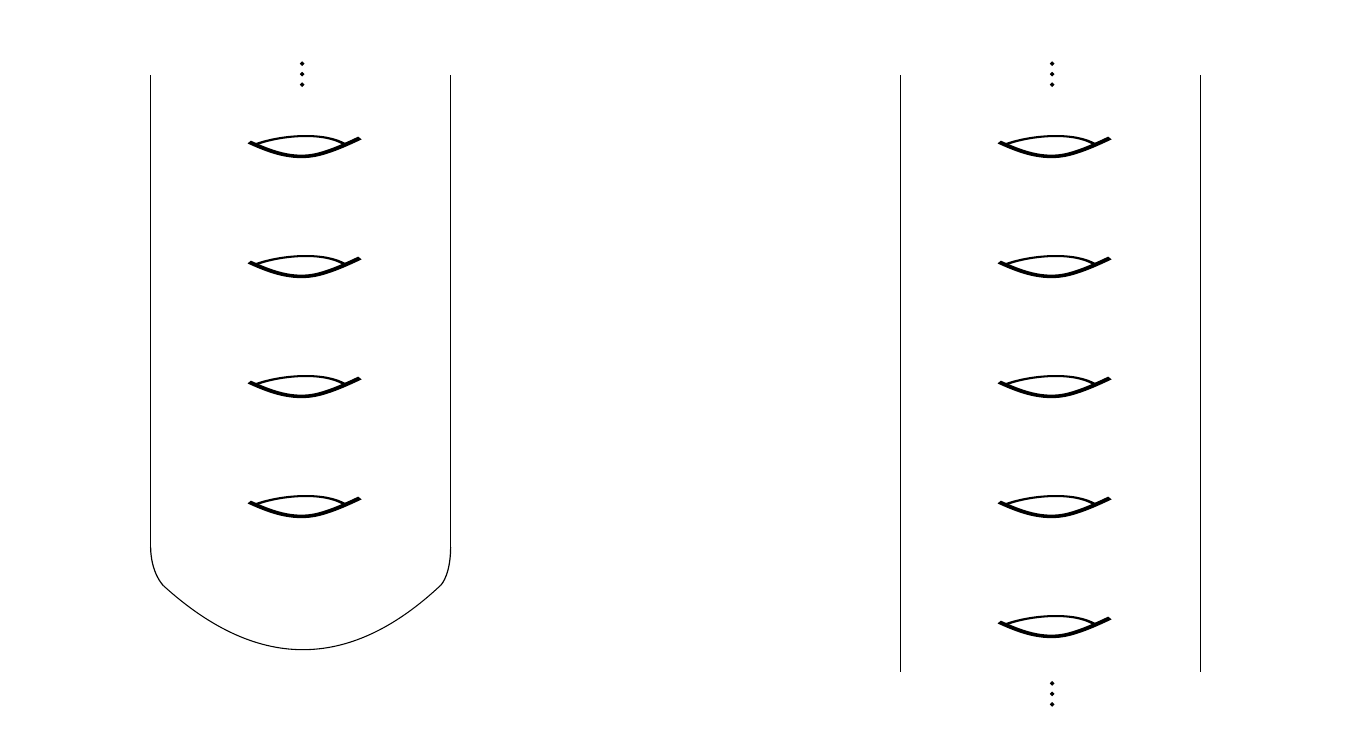
\caption{Two infinite-type surfaces: the Loch Ness monster (left) and Jacob's ladder (right).}
\label{fig:surfaces}
\end{figure}

The key ingredient in the proof of Corollary \ref{cor:stablemcg} is that every subsurface $K \subset \Sigma_{LN}$ of finite type can be displaced by a compactly supported mapping class (with additional control on powers of the displacing element). This is necessary for such a result to hold, as it is shown by a criterion due to Horbez--Qing--Rafi \cite{bigmcg:qm}. For its statement, we need some terminology. A subsurface of finite type $K \subset \Sigma$ is \emph{non-displaceable} if $\phi(K) \cap K \neq \emptyset$ for all $\phi \in \mathrm{Homeo}(\Sigma)$. Let $\widehat{K}$ denote the surface obtained from $K$ by capping off each boundary component with a once-punctured disc. A subgroup $G < \mcg(\Sigma)$ is said to be \emph{$K$-non-elementary} if the stabilizer of $K$ in $G$ contains two elements that induce independent pseudo-Anosov mapping classes of $\widehat{K}$.

\begin{thm}[{\cite[Corollary 2]{bigmcg:qm}}]
\label{thm:bigmcg:qm}
    Let $\Sigma$ be a connected orientable surface, and let $K \subset \Sigma$ be a non-displaceable subsurface. Then every $K$-non-elementary subgroup $G < \mcg(\Sigma)$ has infinite-dimensional second bounded cohomology.
\end{thm}

In particular, such a $G$ does not have commuting cyclic conjugates, or even commuting conjugates. This applies to $G = \mcgc(\Sigma)$, whenever $\Sigma$ contains a sufficiently complicated non-displaceable subsurface $K$.

\medskip

As a special instance, we consider alternative stabilizations of mapping class groups, in the same spirit as the previous construction of $\Gamma_\infty$. We attach $k$ copies of the two-holed torus to a sphere with $k$ boundary components, one on each of its boundary components. At each step this produces an embedding $\Sigma_{kg, k} \hookrightarrow \Sigma_{k(g+1), k}$, whence an inclusion of the mapping class groups $\Gamma_{kg}^k \hookrightarrow \Gamma_{k(g+1)}^k$.

We denote by $\Gamma_{\infty, k}$ the direct union of these groups. We can see $\Gamma_{\infty, k}$ as $\mcgc(\Sigma_{\infty, k})$, where $\Sigma_{\infty, k}$ is the limit of the surfaces $\Sigma_{kg, k}, g \to \infty$.

When $k = 1$, this is just $\Gamma_\infty$, and $\Sigma_{\infty, 1} = \Sigma_{LN}$ is the Loch Ness monster surface. When $k \geq 3$, Theorem \ref{thm:bigmcg:qm} applies, and shows that $\Gamma_{\infty, k}$ has infinite-dimensional second bounded cohomology. Indeed, the subsurfaces $\Sigma_{kg, k}$ are all non-displaceable, since $\Sigma_{\infty, k} \setminus \Sigma_{kg, k}$ has $k \geq 3$ connected components, while the complement of any subsurface disjoint from $\Sigma_{kg, k}$ has at most two connected components.

\medskip

An interesting intermediate case is given by $k = 2$, when $\Sigma_{\infty, 2} = \Sigma_{JL}$ is also known as \emph{Jacob's ladder surface} (see Figure \ref{fig:surfaces}). Then every subsurface of finite type is displaceable by a mapping class, which can be chosen to be a vertical translation; so Theorem \ref{thm:bigmcg:qm} does not apply. However, it is not possible to achieve this displacement using a compactly supported mapping class, i.e. an element of $\Gamma_{\infty, 2} = \mcgc(\Sigma_{JL})$. Indeed, let $K \cong \Sigma_{2g,2}$ be a subsurface of $\Sigma_{JL}$. Then an element $t \in \Gamma_{\infty, 2}$ must be represented by a homeomorphism supported on a subsurface $K' \cong \Sigma_{2h, 2}$ containing $K$, for some $h \geq g$. The complement $K' \setminus K$ is the disjoint union of two subsurfaces $\Sigma_{2g_1,2} \sqcup \Sigma_{2g_2, 2}$, and $t(K' \setminus K) = K' \setminus t(K)$ must be of the same type. It is easy to see that it is not possible to preserve the genera $g_1, g_2$, while ensuring that $t(K) \subset K'$ is disjoint from $K$. Contrast this with the case of the Loch Ness monster, which is illustrated in Figure \ref{fig:element t}.

Camille Horbez has informed us that it should be possible to generalize Theorem \ref{thm:bigmcg:qm} to include surfaces $K$ that are only \emph{$G$-non-displaceable}, i.e. such that $\varphi(K) \cap K \neq \emptyset$ for every homeomorphism $\varphi \in \homeo(\Sigma)$ representing an element of $G$. The above argument would then imply that $\Gamma_{\infty, 2}$ has infinite-dimensional second bounded cohomology.

\subsection{Stable braid group and the bounded cohomology of the braided Ptolemy--Thompson group}\label{subsec:sbraidgroup}

In this setion, we compute the bounded cohomology of the \emph{infinite braid group} $B_\infty$ (with compact support). Kotschick proved that $B_\infty$ has vanishing stable commutator length~\cite{kotschick}. We adapt his argument to show that it has in fact commuting cyclic conjugates.\ As a by-product we fully compute the bounded cohomology of the braided Ptolemy--Thompson group $T^\ast$~\cite{funar2008braided}.

\begin{cor}[Corollary \ref{corintro:braid}]
\label{cor:braid}
    The infinite braid group $B_\infty$ is $\Xsep$-boundedly acyclic, and so is its commutator subgroup.
\end{cor}

The commutator subgroup $[B_\infty, B_\infty]$ has been studied in some occasions \cite{burago2008conjugation, concordance, kimura}; note that it is perfect \cite{braid:perfect}.

\begin{proof}
    Let $H \leq B_\infty$ be a finitely generated group. Thus, $H$ is contained in some $B_n$. We can then embed $B_n$ into $B_{2n}$ simply by adding $n$ strands on the right. Let $t \in B_{2n} \leq B_\infty$ be the element that takes the first $n$ strands and puts them in positions $n+1, \dots, 2n$ by preserving the order and passing over the strands $n+1, \dots, 2n$ which are put in positions $1,\dots , n$, while all other strands are fixed. Then, for every $h \in H \leq B_n$ the support of the element ${}^t h$ lies on the strands $n+1, \dots, 2n$ and so it is disjoint from the one of $h$ (Figure~\ref{fig:conj:by:t}).\ This shows that $[H, {}^{t}H] = 1$.\ Now since $t^2$ fixes the strands supporting $H$ (Figure~\ref{fig:conj:by:t2}), we also have $[H, t^2]=1$, thus proving that $B_\infty$ has commuting cyclic conjugates.\ By applying Theorem~\ref{intro:thm:ccc}, the thesis follows.

\begin{figure}[!htb]
    \centering
    \begin{minipage}{.5\textwidth}
        \begin{center}
    \begin{tikzpicture}[scale = .35, very thick]

\draw[blue] (3,0) .. controls +(0,1) and +(0,-1.5) .. (0,3);

\draw[blue] (4,0) .. controls +(0,1.5) and +(0,-1) .. (1,3);

\draw[white, line width =5] (0,0) .. controls +(0,1.5) and +(0,-1) .. (3,3);

\draw[white, line width =5] (1,0) .. controls +(0,1) and +(0,-1.5) .. (4,3);

\draw[red] (0,0) .. controls +(0,1.5) and +(0,-1) .. (3,3);

\draw[red] (1,0) .. controls +(0,1) and +(0,-1.5) .. (4,3);


\draw[blue] (0,5) .. controls +(0,1.5) and +(0,-1) .. (3,8);

\draw[blue] (1,5) .. controls +(0,1) and +(0,-1.5) .. (4,8);

\draw[white, line width =5] (3,5) .. controls +(0,1) and +(0,-1.5) .. (0,8);

\draw[white, line width =5] (4,5) .. controls +(0,1.5) and +(0,-1) .. (1,8);

\draw[red] (3,5) .. controls +(0,1) and +(0,-1.5) .. (0,8);

\draw[red] (4,5) .. controls +(0,1.5) and +(0,-1) .. (1,8);


\draw (1.25,3) rectangle (-.25,5);
\node at (.5, 4) {$H$};

\draw[red] (3,5) -- (3,3);

\draw[red] (4,5) -- (4,3);


\draw[fill, very thin] (4.5,3.1) .. controls +(.5,0) and +(-.25,0) .. (4.75,1.5).. controls +(-.15,0) and +(.5,0) .. (4.5,3.1); 
\draw[fill, very thin] (4.5,-.1) .. controls +(.5,0) and +(-.25,0) .. (4.75,1.5).. controls +(-.15,0) and +(.5,0) .. (4.5,-.1);

\draw[fill, very thin] (4.5,8.1) .. controls +(.5,0) and +(-.25,0) .. (4.75,6.5).. controls +(-.15,0) and +(.5,0) .. (4.5,8.1); 
\draw[fill, very thin] (4.5,4.9) .. controls +(.5,0) and +(-.25,0) .. (4.75,6.5).. controls +(-.15,0) and +(.5,0) .. (4.5,4.9);

\node[right] at (4.5,6.5) {$t$};

\node[right] at (4.5,1.5) {$t^{-1}$};
     \end{tikzpicture}
  \end{center}

  \caption{The conjugation $t H t^{-1}$.}
  \label{fig:conj:by:t}
    \end{minipage}%
    \begin{minipage}{0.5\textwidth}
    \begin{center}
    \begin{tikzpicture}[scale = .35, very thick]


\draw[red] (0,0) .. controls +(0,-1.5) and +(0,1) .. (3,-3);

\draw[red] (1,0) .. controls +(0,-1) and +(0,1.5) .. (4,-3);

\draw[white, line width =5] (3,0) .. controls +(0,-1) and +(0,1.5) .. (0,-3);

\draw[white, line width =5] (4,0) .. controls +(0,-1.5) and +(0,1) .. (1,-3);

\draw[blue] (3,0) .. controls +(0,-1) and +(0,1.5) .. (0,-3);

\draw[blue] (4,0) .. controls +(0,-1.5) and +(0,1) .. (1,-3);

\draw[blue] (3,0) .. controls +(0,1) and +(0,-1.5) .. (0,3);

\draw[blue] (4,0) .. controls +(0,1.5) and +(0,-1) .. (1,3);

\draw[white, line width =5] (0,0) .. controls +(0,1.5) and +(0,-1) .. (3,3);

\draw[white, line width =5] (1,0) .. controls +(0,1) and +(0,-1.5) .. (4,3);

\draw[red] (0,0) .. controls +(0,1.5) and +(0,-1) .. (3,3);

\draw[red] (1,0) .. controls +(0,1) and +(0,-1.5) .. (4,3);


\draw[blue] (0,5) .. controls +(0,1.5) and +(0,-1) .. (3,8);

\draw[blue] (1,5) .. controls +(0,1) and +(0,-1.5) .. (4,8);

\draw[white, line width =5] (3,5) .. controls +(0,1) and +(0,-1.5) .. (0,8);

\draw[white, line width =5] (4,5) .. controls +(0,1.5) and +(0,-1) .. (1,8);

\draw[red] (3,5) .. controls +(0,1) and +(0,-1.5) .. (0,8);

\draw[red] (4,5) .. controls +(0,1.5) and +(0,-1) .. (1,8);


\draw[red] (0,8) .. controls +(0,1.5) and +(0,-1) .. (3,11);

\draw[red] (1,8) .. controls +(0,1) and +(0,-1.5) .. (4,11);

\draw[white, line width =5]  (3,8) .. controls +(0,1) and +(0,-1.5) .. (0,11);

\draw[white, line width =5] (4,8) .. controls +(0,1.5) and +(0,-1) .. (1,11);

\draw[blue] (3,8) .. controls +(0,1) and +(0,-1.5) .. (0,11);

\draw[blue] (4,8) .. controls +(0,1.5) and +(0,-1) .. (1,11);

\draw (1.25,3) rectangle (-.25,5);
\node at (.5, 4) {$H$};

\draw[red] (3,5) -- (3,3);

\draw[red] (4,5) -- (4,3);


\draw[fill, very thin] (4.5,3.1) .. controls +(.5,0) and +(-.25,0) .. (4.75,0).. controls +(-.15,0) and +(.5,0) .. (4.5,3.1); 
\draw[fill, very thin] (4.5,-2.9) .. controls +(.5,0) and +(-.25,0) .. (4.75,0).. controls +(-.15,0) and +(.5,0) .. (4.5,-2.9);

\draw[fill, very thin] (4.5,11.1) .. controls +(.5,0) and +(-.25,0) .. (4.75,8).. controls +(-.15,0) and +(.5,0) .. (4.5,11.1); 
\draw[fill, very thin] (4.5,4.9) .. controls +(.5,0) and +(-.25,0) .. (4.75,8).. controls +(-.15,0) and +(.5,0) .. (4.5,4.9);

\node[right] at (4.5,8) {$t^2$};

\node[right] at (4.5,0) {$t^{-2}$};
\end{tikzpicture}
  \end{center}

  \caption{The conjugation $t^2 H t^{-2}$.}
  \label{fig:conj:by:t2}
    \end{minipage}
\end{figure}

The statement for derived subgroups follows from Lemma \ref{lem:derived}.
\end{proof}

Note that the proof really relies on being able to apply non-trivial permutations to the strands. This motivates the following question:

\begin{question}
\label{q:braid}
    Is the infinite pure braid group $PB_\infty$ ($\Xsep$-)boundedly acyclic?
\end{question}

As a consequence of Corollary~\ref{cor:braid}, we get the following:

\begin{cor}[Corollary \ref{corintro:ptolemy}]
\label{cor:ptolemy}
    The bounded cohomology ring (with the cup product structure) of the braided Ptolemy--Thompson group $T^\ast$ with trivial real coefficients is the following:
    \[
    H_b^\bullet(T^\ast; \R) \cong \R[x],
    \]
    where $|x| = 2$. Under this isomorphism, $x$ corresponds to the Euler class of the defining action on the circle of the quotient $T$.
\end{cor}
\begin{proof}
    It is sufficient to notice that the braided Ptolemy--Thompson group $T^\ast$ is an extension of Thompson's group $T$ by $B_\infty$~\cite{funar2008braided}:
    \[
    1 \to B_\infty \to T^\ast \to T \to 1.
    \]
    Since $B_\infty$ is boundedly acyclic by Corollary~\ref{cor:braid}, we can apply \cite[Theorem 4.1.1]{moraschini_raptis_2023} and obtain
    \[
    H_b^\bullet(T^\ast; \R) \cong H_b^\bullet(T; \R)
    \]
    as cohomology ring, where the isomorphism is induced by the pullback.
    The thesis now follows from the computation of the bounded cohomology of Thompson's group $T$~\cite{fflm2, monod:thompson, monodnariman}.
\end{proof}

\subsection{Braided Thompson groups}

We briefly recall the definition of $bV$ \cite{bV:brin, bV:dehornoy, bV:qm}.\ An element of $bV$ is represented by a triple of the form $(T_-, \beta, T_+)$, where $T_{\pm}$ are binary trees with $n$ leaves, and $\beta \in B_n$, the braid group on $n$ strands. We picture $T_+$ upside down and below $T_-$, with $\beta$ connecting the leaves of $T_+$ up to those of $T_-$, as in Figure \ref{fig:bV}.\ We call such a picture a \emph{strand diagram}. Elements of $bV$ are equivalence classes of such representative triples, denoted by $[T_-, \beta, T_+]$, where the equivalence is generated by a suitable notion of expansion analogous to that of tree pair diagrams for Thompson's groups. Given two elements $[T_-, \beta, T_+]$ and $[U_-, \gamma, U_+]$, up to expansion we can assume that $T_+ = U_-$, and then their product is $[T_-, \beta \gamma, U_+]$.

Denote by $\widehat{B_n}$ the copy of $B_{n-1}$ inside $B_n$ that is supported on the first $n-1$ strands.\ Then the subset of elements of $bV$ of the form $[T_-, \beta, T_+]$ with $\beta \in \widehat{B_n}$ forms a subgroup, that is denoted by $\widehat{bV}$.

The \emph{right depth} of a tree is the distance (length of the geodesic path) of its rightmost leaf to the root in a strand diagram (recall that $T_+$ is pictured upside down). The map $\chi_1\colon \widehat{bV} \to \Z$ sending an element $[T_-, \beta, T_+]$ to the right depth of $T_-$ minus the right depth of $T_+$, is a well-defined surjective homomorphism, which in fact realizes the abelianization of $\widehat{bV}$ \cite[Corollary 2.14]{bV:qm}.\ We denote by $\widehat{D}$ the kernel of $\chi_1$, and by $\widehat{bV}(1)$ the subgroup of $\widehat{D}$ consisting of elements $[T_-, \beta, T_+]$, where both $T_-$ and $T_+$ have right depth $1$.

\begin{cor}[Corollary \ref{corintro:bV}]
\label{cor:bV}
The group $\widehat{bV}$ is $\Xsep$-boundedly acyclic.
\end{cor}

\begin{proof}
By Theorems \ref{thm:co-amenable} and \ref{intro:thm:ccc}, it suffices to show that the co-amenable subgroup $\widehat{D}$ has commuting cyclic conjugates. Since every finitely generated subgroup of $\widehat{D}$ is conjugate into $\widehat{bV}(1)$~\cite[Lemma 2.12]{bV:qm}, it is enough to show that there exists an element $t \in \widehat{D}$ such that $[\widehat{bV}(1), {}^t \widehat{bV}(1)] = 1$, and $[\widehat{bV}(1), t^2] = 1$. We choose $t = [T, \beta, T]$, where $T$ is a caret with a second caret hanging on the right and $\beta$ is the first standard generator of $B_3$ (Figure \ref{fig:bV}).\ It is known that $[\widehat{bV}(1), {}^t \widehat{bV}(1)] = 1$~\cite[Lemma 2.16]{bV:qm}, wheras the fact that $[\widehat{bV}(1), t^2] = 1$ can be proved just as in our proof of Corollary \ref{cor:braid} (compare with Figure~\ref{fig:conj:by:t2}).
\begin{figure}[htb]
\centering
\begin{tikzpicture}[line width=0.4pt]

\draw (0,0) -- (0.5,0.5) -- (1,0) -- (1,-0.5) -- (0.5,-1) -- (0,-0.5);
\filldraw[lightgray!50] (-0.375,0.125) -- (0.375,0.125) -- (0.375,-0.625) -- (-0.375,-.625) -- (-0.375,0.125);
\draw (-0.375,0.125) -- (0.375,0.125) -- (0.375,-.625) -- (-0.375,-.625) -- (-0.375,0.125)   (0.5,-1) -- (0.5,-1.25)   (0.5,0.5) -- (0.5,0.75);
\filldraw (0.5,0.5) circle (0.75pt)   (1,0) circle (0.75pt)   (1,-0.5) circle (0.75pt)   (0.5,-1) circle (0.75pt);
\node at (1.5,-0.25) {$=$};

\begin{scope}[yshift=-1.75cm]
\draw (0,0) -- (0.5,0.5) -- (1,0) -- (1,-0.5) -- (0.5,-1) -- (0,-0.5)   (0.75,0.25) -- (0.5,0)   (0.5,-0.5) -- (0.75,-0.75);
\draw (0.5,0) to[out=-90,in=90] (0,-0.5);
\draw[white,line width=4pt] (0,0) to[out=-90,in=90] (0.5,-0.5);
\draw (0,0) to[out=-90,in=90] (0.5,-0.5);
\filldraw (0,0) circle (0.75pt)   (0.5,0.5) circle (0.75pt)   (1,0) circle (0.75pt)   (1,-0.5) circle (0.75pt)   (0.5,-1) circle (0.75pt)   (0,-0.5) circle (0.75pt)   (0.75,0.25) circle (0.75pt)   (0.5,0) circle (0.75pt)   (0.5,-0.5) circle (0.75pt)   (0.75,-0.75) circle (0.75pt);
\end{scope}

\begin{scope}[yshift=1.75cm]
\draw (0,0) -- (0.5,0.5) -- (1,0) -- (1,-0.5) -- (0.5,-1) -- (0,-0.5)   (0.75,0.25) -- (0.5,0)   (0.5,-0.5) -- (0.75,-0.75);
\draw (0,0) to[out=-90,in=90] (0.5,-0.5);
\draw[white,line width=2pt] (0.5,0) to[out=-90,in=90] (0,-0.5);
\draw (0.5,0) to[out=-90,in=90] (0,-0.5);
\filldraw (0,0) circle (0.75pt)   (0.5,0.5) circle (0.75pt)   (1,0) circle (0.75pt)   (1,-0.5) circle (0.75pt)   (0.5,-1) circle (0.75pt)   (0,-0.5) circle (0.75pt)   (0.75,0.25) circle (0.75pt)   (0.5,0) circle (0.75pt)   (0.5,-0.5) circle (0.75pt)   (0.75,-0.75) circle (0.75pt);
\end{scope}

\begin{scope}[xshift=3.8cm,yshift=-0.5cm,xscale=-1.2,yscale=-1.2]
\draw (0,-0.5) -- (0,0) -- (0.5,0.5) -- (1,0) -- (1,-0.5) -- (0.5,-1) -- (0,-0.5)   (0.5,0.5) -- (0.75,0.75) -- (1.5,0) -- (1.5,-0.5) -- (0.75,-1.25) -- (0.5,-1);
\begin{scope}[xshift=0.8cm, yshift=0.5]\filldraw[lightgray!50] (-0.375,0.125) -- (0.375,0.125) -- (0.375,-0.725) -- (-0.375,-0.725) -- (-0.375,0.125);
\draw (-0.375,0.125) -- (0.375,0.125) -- (0.375,-0.725) -- (-0.375,-0.725) -- (-0.375,0.125);\end{scope}
\filldraw (0.5,0.5) circle (0.75pt)   (0,0) circle (0.75pt)   (0,-0.5) circle (0.75pt)   (0.5,-1) circle (0.75pt)   (0.75,0.75) circle (0.75pt)   (1.5,0) circle (0.75pt)   (1.5,-0.5) circle (0.75pt)   (0.75,-1.25) circle (0.75pt);
\end{scope}
\end{tikzpicture}
\caption{The conjugation $t \widehat{bV}(1) t^{-1}$.}
\label{fig:bV}
\end{figure}
\end{proof}

\subsection{Automorphisms of products}\label{subsec:aut:free:group} 

Let $\Gamma$ be an arbitrary group.\ We denote by $\Gamma^{* \infty} \coloneqq *_{i \in \N} \Gamma$, by $\Gamma^{\times \infty} \coloneqq \prod_{i \in \N} \Gamma$, and by $\Gamma^{\oplus \infty} \coloneqq \oplus_{i \in \N} \Gamma$.\ In each case, we denote by $\Gamma_i$ the $i$-th copy of $\Gamma$. We define $\Aut_\infty(\Gamma^{* \infty})$ the subgroup of $\Aut(\Gamma^{* \infty})$ consisting of automorphisms that fix pointwise the groups $\Gamma_i$ for $i$ large enough. Similarly, we define $\Aut_\infty(\Gamma^{\times \infty})$ and $\Aut_\infty(\Gamma^{\oplus \infty})$.

\begin{cor}[Corollary \ref{corintro:auto}]
\label{cor:auto}
For every group $\Gamma$, the groups $\Aut_\infty(\Gamma^{* \infty})$, $\Aut_\infty(\Gamma^{\times \infty})$ and $\Aut_\infty(\Gamma^{\oplus \infty})$ are $\Xsep$-boundedly acyclic.
\end{cor}

\begin{proof}
We prove this for $\Aut_\infty(\Gamma^{* \infty})$; the other cases are analogous.\ First note that, by the universal property of free products, every automorphism is uniquely determined by its restriction to each free factor $\Gamma_i$. Let $H \leq \Aut_\infty(\Gamma^{* \infty})$ be a finitely generated subgroup.\ Then there exists $n\in\N$ such that $H$ fixes  $\Gamma_i$ pointwise for all $i > n$.\ We define $t$ as follows: For all $1 \leq i \leq 2n$, it swaps $\Gamma_i$ and $\Gamma_{2n-i}$, and for all $i > 2n$, it fixes $\Gamma_i$ pointwise. Then $t$ has order $2$, thus $[H, t^2] = 1$.\ Moreover, ${}^t H$ fixes $\Gamma_i$ pointwise, for $1 \leq i \leq n$, and so $[H, {}^t H] = 1$.\ This shows that $\Aut_\infty(\Gamma^{* \infty})$ has commuting cyclic conjugates, and we conclude by Theorem \ref{intro:thm:ccc}.
\end{proof}
\subsection{Cremona groups}\label{subsec:cremona:group} Recall that, for a field $K$, the Cremona groups $\mathrm{Cr}_n(K)=\mathrm{Aut}(K(x_1, \ldots, x_n)|K)$ are the groups of automorphisms of the field of rational functions~$K(x_1, \ldots, x_n)$ over $K$. We stabilize these groups along the natural inclusions $\mathrm{Cr}_n(K)\hookrightarrow \mathrm{Cr}_{n+1}(K)$, fixing the coordinate $x_{n+1}$. We consider their direct limit $$\mathrm{Cr}_{\infty}(K)=\bigcup_{n\in\mathbb{N}}\mathrm{Cr}_n(K).$$
In this section, we apply Theorem \ref{intro:thm:ccc} to $\mathrm{Cr}_\infty(K)$ and show: 
\begin{cor}[Corollary \ref{corintro:cremona}]
\label{cor:cremona}
   The group $\mathrm{Cr}_{\infty}(K)$ is $\Xsep$-boundedly acyclic.
\end{cor}
\begin{proof}
Note that every element of $\mathrm{Aut}(K(x_1, \ldots, x_n)|K)$ is determined by its action on $x_1, \ldots,x_n$. Let $H\leq \mathrm{Cr}_\infty(K)$ be a finitely generated subgroup. Then there exists $n\in\mathbb{N}$ such that $H\leq \mathrm{Cr}_n(K)$. Define an element $t\in \mathrm{Cr}_{2n}(K)$ as follows: for all $1\leq i\leq 2n$, $t$ swaps $x_i$ and $x_{2n-i}$. Thus $t$ has order $2$, so that $[H, t^2]=1$. Moreover ${}^t H$ fixes $x_i$ for all $1\leq i\leq n$, so that ${}^t H$ commutes with $H$. This shows that $\mathrm{Cr}_{\infty}(K)$ has commuting cyclic conjugates, and we conclude by Theorem \ref{intro:thm:ccc}.
\end{proof}
\subsection{Transformation groups}\label{subsec:transf:group}

In this section, we apply Theorem \ref{intro:thm:ccc} to some compactly supported homeomorphism and diffeomorphism groups of manifolds, proving all the parts of Corollary \ref{corintro:transformation}.\ The full groups of compactly supported homeomorphisms and diffeomorphisms were already tackled via other approaches~\cite{Matsu-Mor, monod:thompson, fflm2, monodnariman}.\ But using the commuting cyclic conjugates condition, we are also able to treat structure-preserving groups, such as symplectomorphisms and volume-preserving diffeomorphisms.\ Our approach is closer to Kotschick's argument for the vanishing of the stable commutator length of such groups~\cite{kotschick}.

Recall that given a topological space $X$ and a map $f\colon X \to X$, we set $\supp(f) \coloneqq \{ x \in X : f(x) \neq x \}$ and say that $f$ is \emph{compactly supported} if $\supp(f)$ has compact closure. We denote by $\homeo(X)$ the group of compactly supported homeomorphisms of $X$. For a group of homeomorphisms $H$, we denote by $\supp(H)$ the union of the supports of its elements. We will frequently use the formula 
\[\supp({}^{t}H) = t (\supp(H)).\]

We will repeatedly apply the following dynamical version of the commuting cyclic conjugates condition, both in this and the next section. 

\begin{lemma}[Dynamical commuting cyclic conjugates]
    \label{lem:ccc:dynamical}

    Let $X$ be a topological space, and let $T \leq \homeo(X)$.
    \begin{enumerate}
        \item Suppose that for every compact subset $K \subset X$, there exists $t \in T$ such that the sets $\{t^p(K)\}_{p \geq 1}$ are pairwise disjoint. Then every group $G$ such that $T^{(d)} \leq G \leq \homeo(X)$ for some $d \geq 0$ has commuting $\Z$-conjugates.
        \item Suppose that for every compact subset $K \subset X$, there exists $t \in T$ such that $t(K) \cap K = \emptyset$ and $t^2|_K = \id |_K$. Then every group $G$ such that $T^{(d)} \leq G \leq \homeo(X)$ for some $d \geq 0$ has commuting cyclic conjugates.
    \end{enumerate}
    In particular, in both cases, every such group $G$ is $\Xsep$-boundedly acyclic.
\end{lemma}

Recall that $T^{(d)}$ denotes the $d$-th derived subgroup of $T$. It will be apparent from the proof below that Item \emph{(2)} can be generalized by replacing $2$ with larger positive integers $n$, even with $n$ varying in terms of $K$. However $n = 2$ is the only case we will use, so we state the lemma like this for clarity.

\begin{proof}
    We will start by showing that if $T \leq G \leq \homeo(X)$, then $G$ has commuting $\Z$-conjugates in the first case, or commuting cyclic conjugates in the second case. Let $H \leq G$ be a finitely generated subgroup. Since each generator is compactly supported, there exists a compact set $K$ such that $\supp(H) \subset K$.
    
    For the first item, let $t$ be as in the hypothesis for this $K$, and notice that $t \in T \leq G$. Then the groups $\{{}^{t^p} H\}_{p \geq 1}$ are supported on the pairwise disjoint sets $\supp({}^{t^p} H) = t^p(\supp(H)) \subset t^p(K)$.\ So $[H, {}^{t^p} H] = 1$ for all $p \geq 1$.\ This shows that $G$ has commuting $\Z$-conjugates.

    For the second item, let $t$ be such that $t(K) \cap K = \emptyset$, and $t^2|_K = \id|_K$. First, we notice that $H$ and ${}^t H$ have disjoint supports, and so $[H, {}^t H] = 1$.\ Secondly, for all $h \in H$, we have $[h, t^2]|_K = [h, \id]|_K = \id|_K$ while $[h, t^2]|_{X \setminus K} = [\id, t^2]|_{X \setminus K} = \id|_{X \setminus K}$.\ So $[H, t^2] = 1$, and that proves that $G$ has commuting cyclic conjugates (with $n = 2$ in the definition).

    \medskip

    The generalization to derived subgroups now follows by the same argument as in Lemma \ref{lem:derived}.
\end{proof}

\begin{rem}[Universal covers]
    We will consider transformation groups as discrete groups. However, they all carry a natural group topology, and so to each of these groups $G$ is associated a universal cover $\tilde{G}$, which is also of interest in many settings \cite{banyaga}.
    However, $\tilde{G}$ is an extension of $G$ by the group $\pi_1(G)$, which is always abelian \cite[Exercise 52.7]{munkres}. Therefore, by Gromov's Mapping Theorem (Theorem \ref{thm:mapping}) and its inflation version (Remark \ref{rem:mapping:converse}), $\tilde{G}$ is $\Xsep$-boundedly acyclic if and only if $G$ is.
\end{rem}

\subsubsection{Hamiltonian diffeomorphisms and symplectomorphisms}

Let $(M, \omega)$ be a symplectic manifold (i.e.\ an even dimensional smooth manifold $M$ with a symplectic form $\omega \in \Omega^2(M)$). The key example is that of $\R^{2n}$ endowed with the standard symplectic form $\omega = \sum_{i = 1}^n dx_i \wedge dy_i$.\ We recall that $\symp(M, \omega)$ denotes the group of compactly supported diffeomorphisms that preserve $\omega$ and $\sympo(M, \omega)$ its identity component. We will remove $\omega$ from the notation if it is clear from the context.\ Given a compactly supported smooth function $H \colon M \to \R$, called \emph{Hamiltonian}, we say that the \emph{Hamiltonian vector field} $X_H$ is the unique vector field such that 
$\iota_{X_H} \omega = dH$,
where $\iota_{X_H}$ denotes the interior product~\cite{da2009introduction}, note that $X_H$ is compactly supported, so it is complete.\ Then, the \emph{Hamiltonian flow} $\phi \colon M \times \R \rightarrow M$ is a one-parameter family of diffeomorphisms $\phi_t \colon M \to M$ associated to $X_H$.\ By construction, each $\varphi_t$ is compactly supported. We call such diffeomorphisms \emph{Hamiltonian}; note that up to rescaling the flow we may realize each Hamiltonian diffeomorphism as the time-one map of a Hamiltonian flow \cite[Exercise 1.4.A]{polterovich2012geometry}.\ We denote by $\ham(M, \omega)$ the group of (compactly supported) Hamiltonian diffeomorphisms.\ This is indeed a subgroup of $\sympo(M, \omega)$~\cite{polterovich2012geometry}.

\begin{cor}[Symplectic transformation groups of $\R^{2n}$]
\label{cor:tg:symp}
    Let $\R^{2n}$ be endowed with its standard symplectic form.\ Let $G$ be a group such that $[\ham(\R^{2n}), \ham(\R^{2n})] \leq G \leq \homeo(\R^{2n})$.\ Then $G$ is $\Xsep$-boundedly acyclic.
\end{cor}

\begin{example}[Symplectic transformation groups of $\R^{2n}$]
    Corollary \ref{cor:tg:symp} applies to the groups $\ham(\R^{2n})$, $\sympo(\R^{2n})$ and $\symp(\R^{2n})$. It also applies to the group of compactly supported Hamiltonian homeomorphisms of $\R^{2n}$ \cite{hamhomeo1, hamhomeo2}. Note that we stop at the first derived subgroup since it is perfect \cite[Chapter 4]{banyaga}.
\end{example}

\begin{proof}
We are going to show that $\ham(\R^{2n}) \leq \homeo(\R^{2n})$ satisfies Lemma \ref{lem:ccc:dynamical}.\emph{(2)}; the statement will follow.

We see $\R^{2n}$ as $\R \times \R \times \R^{2n - 2}$, with coordinates $(x, y, \underline{z})$. Let $K \subset \R^{2n}$ be a compact set. Then there exists a compact set $A \subset \R^{2n-2}$ such that $K \subset B_R(0, 0) \times A$, where $B_R(0, 0) \subset \R^2$ is the closed disk of radius $R$ around the origin.

We define $t$ starting from a smooth Hamiltonian $F \colon \R^{2n} \to \R$ satisfying:
\[
    (x, y, \underline{z}) \mapsto
        \begin{cases}
            \frac{1}{2} \big((x-2R)^2 + y^2\big) &\mbox{ if } (x, y) \in B_{4R}(2R, 0), \\
            0 &\mbox{ if } (x, y) \in \R^2 \setminus B_{6R}(2R, 0).
            \end{cases}
\]
When $n > 1$, $F$ is not compactly supported, so we slightly modify it by taking the product with a compactly supported function $\eta \colon \R^{2n} \to \R$ such that $\eta|_{B_{4R}(2R, 0) \times A} \equiv 1$. Hence, we get a compactly supported Hamiltonian $H = \eta \cdot F \colon \R^{2n} \to \R$.
The Hamiltonian vector field associated with $H$ (represented in Figure~\ref{fig:ham} in its projection to $\R^2$) satisfies the following~\cite[Section~15.2]{woit2017quantum}:    
\[
    X_H(x, y, \underline{z}) = y \frac{\partial}{\partial x} - (x - 2R) \frac{\partial}{\partial y} \qquad \text{ if } (x, y, \underline{z}) \in B_{4R}(2R, 0) \times A.
\]
This vector field over the compact set $B_{4R}(2R, 0) \times A$ produces a Hamiltonian flow that corresponds to the clockwise rotation around the point $(2R, 0)$ in the first copy of $\R^2$ and that is the identity over the other coordinates. More precisely, if  $(x, y, \underline{z}) \in B_{4R}(2R, 0) \times A$, then $\phi_s(x, y, \underline{z})$ equals
\[
        \Big(2R+ (x - 2R) \cos (s) + y \sin (s), y \cos (s) - (x - 2R) \sin (s), \underline{z} \Big),
\]
and the support of $\varphi_s$ is contained in the (compact) support of $\eta$.
\begin{figure}[htbp]
\includegraphics[width=\linewidth]{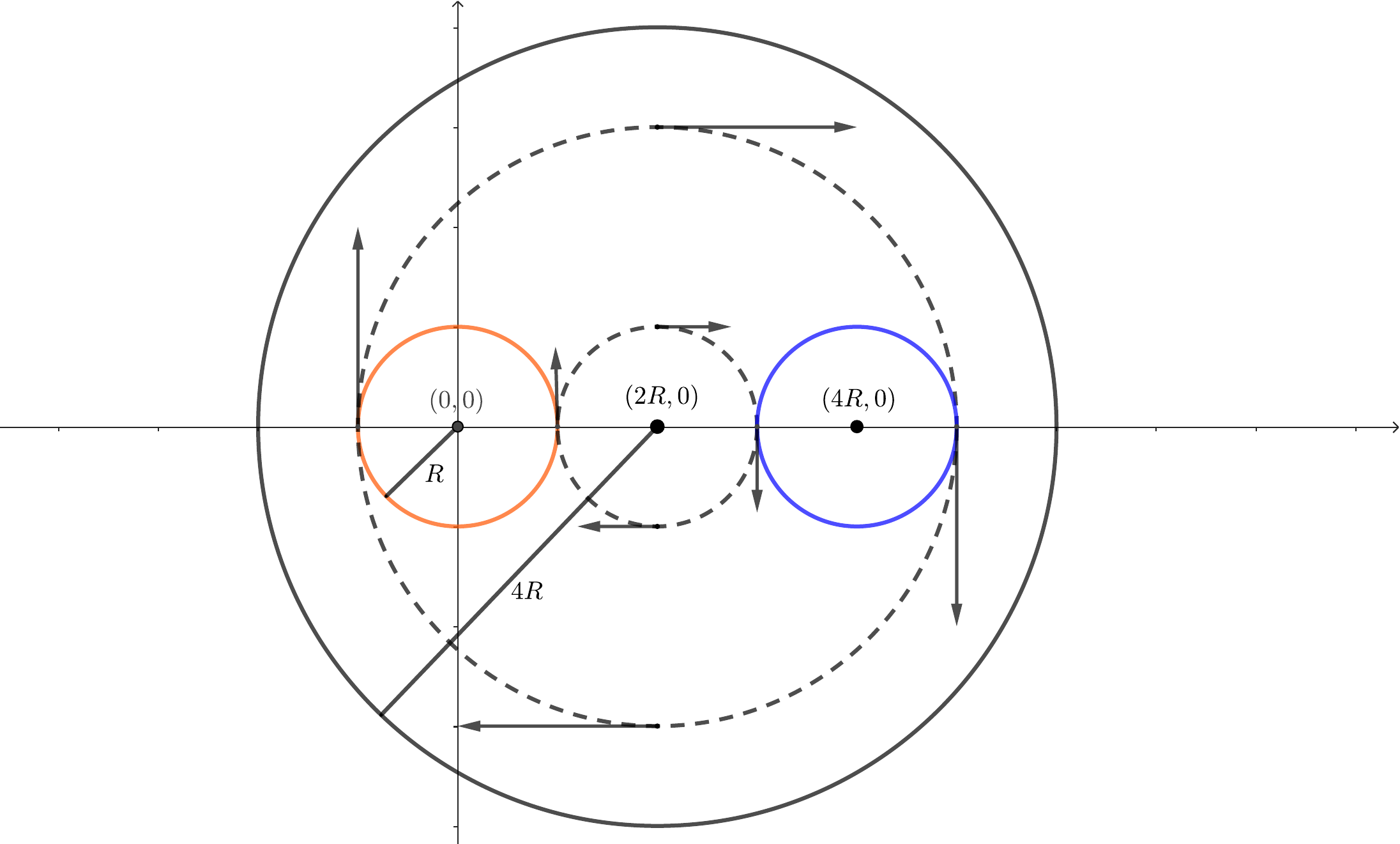}
\caption{The projection of the Hamiltonian vector field on the $2$-dimensional disc $B_{4R}(2R, 0)$ of the first copy of $\R^2$.}\label{fig:ham}
\end{figure}
    We define $t \coloneqq \phi_\pi \in \ham(\R^{2n})$. Then, it is immediate to check that 
    \[
    t(K) \subset t\Big( B_R(0, 0) \times A \Big) \subset B_R(4R, 0) \times A,
    \]
    that is disjoint from $B_R(0, 0) \times A$, whence from $K$. Moreover, by construction we also have that $t^2$ coincides with the identity on $K$, since it is the identity over the larger set $B_{R}(0, 0) \times A$.
    Thus, $\ham(\R^{2n}) \leq \homeo(\R^{2n})$ satisfies Lemma \ref{lem:ccc:dynamical}.\emph{(2)}, and we conclude.
\end{proof}

\subsubsection{Volume-preserving diffeomorphisms}

For an $n$-manifold $M$ with a volume form $\vol \in \Omega^n(M)$, we denote by $\diffvol(M)$ the group of compactly supported diffeomorphisms that preserve the volume, and by $\diffovol(M)$ its identity component. The key example is that of $\R^n$ endowed with the standard volume form $dx_1 \wedge \cdots \wedge dx_n$. Notice that symplectomorphisms of $\R^{2n}$ are volume-preserving, since $\vol = \omega^n/n!$, where $\omega$ is the standard symplectic form. Therefore, Corollary \ref{cor:tg:symp} already proves that $\diffovol(\R^{2n})$ and $\diffvol(\R^{2n})$ are $\Xsep$-boundedly acyclic. In order to include odd dimensions, we use the following trick:

\begin{lemma}[Volume-preserving cut-off]
\label{lem:cutoff}

    Let $\{f_t\}_{t \in [0, 1]}$ be a smooth isotopy of diffeomorphisms of $\R^n$ such that $f_t$ is volume-preserving for all $t \in [0, 1]$.\ Let $\eta\colon \R \to [0, 1]$ be smooth.\ Then the map $\varphi\colon \R^{n+1} \to \R^{n+1}$ defined as $\varphi(\underline{x}, z) = (f_{\eta(z)}(\underline{x}), z)$ is a volume-preserving diffeomorphism, where $\underline{x} \in \R^n$.
    
    Moreover, if the isotopy $\{f_t\}_{t \in [0, 1]}$ is compactly supported, $f_0 = \id$ and $\eta$ is compactly supported, then $\varphi$ is compactly supported.
\end{lemma}

Notice that in the last statement, we are using two different notions of support: For the diffeomorphisms $f_t$ we have $\supp(f_t) = \{ x \in X : f_t(x) \neq x \}$, while for the function $\eta$ we have $\supp(\eta) = \{x \in X : \eta(x) \neq 0 \}$.

\begin{proof}
    As a composition of smooth functions, $\varphi$ is smooth, and it is a diffeomorphism since its inverse $\varphi^{-1}(\underline{x}, z) = (f_{\eta(z)}^{-1}(\underline{x}), z)$ is also smooth. It remains to show that $\varphi$ is volume-preserving, equivalently that its Jacobian $D_{(\underline{x}, z)} \varphi$ has determinant $1$. We notice that $\left. \frac{\partial \varphi^i}{\partial x_j} \right|_{(\underline{x}, z)} = \frac{\partial f_{\eta(z)}^i}{\partial x_j}$ for $1 \leq i, j \leq n$; that $\left. \frac{\partial \varphi^{n+1}}{\partial x_j} \right|_{(\underline{x}, z)} = 0$ for $1 \leq j \leq n$ and $\left. \frac{\partial \varphi^{n+1}}{\partial z} \right|_{(\underline{x}, z)}= 1$.
    Therefore:
    \[D_{(\underline{x}, z)}\varphi = 
    \begin{pmatrix}
        D_{\underline{x}} f_{\eta(z)} & * \\
        0 & 1
    \end{pmatrix}\]
    which has determinant $1$ since $f_{\eta(z)}$ is volume-preserving. 
    
    The last sentence follows from the following fact: $(\underline{x}, z) \in \textup{supp}(\varphi)$ if and only if $\underline{x} \in \textup{supp} (f_{\eta(z)})$.\ Therefore $\textup{supp} (\varphi) \subset \left(\cup_{t \in [0, 1]} \textup{supp} (f_t)\right) \times \textup{supp}(\eta)$, which is compact by assumption.
\end{proof}

\begin{cor}[Volume-preserving transformation groups of $\R^n$]
    \label{cor:tg:volume}
    Let $n \geq 2$ and let $\R^n$ be endowed with its standard volume form.\ Let $G$ be a group such that $[\diffovol(\R^n), \diffovol(\R^n)] \leq G \leq \homeo(\R^n)$.\ Then $G$ is $\Xsep$-boundedly acyclic.
\end{cor}

Note that we stop at the first derived subgroup since it is perfect \cite[Chapter 5]{banyaga}.

\begin{proof}
    As we mentioned above, by Corollary \ref{cor:tg:symp} it suffices to show the statement in odd dimensions. Therefore we consider $\R^{2n+1}$, for $n \geq 1$, and we show that $\diffovol(\R^{2n+1})$ satisfies Lemma \ref{lem:ccc:dynamical}.\emph{(2)}.

    Let $K \subset \R^{2n+1}$ be compact. Then there exist a compact set $A \subset \R^{2n}$ and a compact interval $I \subset \R$ such that $K \subset A \times I$. Let $f \in \ham(\R^{2n})$ be such that $f(A) \cap A = \emptyset$ and $f^2|_A = \id|_A$, given by the proof of Corollary \ref{cor:tg:symp}.\ Since $f \in \ham(\R^{2n})$ is the time-one map of the flow of a compactly supported vector field, there exists a compactly supported smooth isotopy of volume-preserving diffeomorphisms $\{f_t\}_{t \in [0, 1]}$ such that $f_0 = \id, f_1 = f$. Let $\eta\colon \R \to [0, 1]$ be a compactly supported smooth function such that $\eta|_I \equiv 1$.\ Then, if we define $\varphi(x, t) = (f_{\eta(t)}(x), t)$, Lemma \ref{lem:cutoff} implies that $\varphi \in \diffovol(\R^{2n+1})$. Moreover, $\varphi|_{A \times I} = f|_{A} \times \textup{id}$; thus we have
    \[A \times I \cap \varphi(A \times I) = (f(A) \cap A) \times I = \emptyset, \] and \[\varphi^2|_{A \times I} = f^2|_{A} \times \textup{id} = \id_{A \times I}.\] This shows that $\diffvol$ satisfies Lemma \ref{lem:ccc:dynamical}.\emph{(2)} and so the thesis follows.
\end{proof}

\begin{rem}[Commuting cyclic vs commuting $\Z$-conjugates]
    Note that all the transformation groups encountered so far, though they have commuting cyclic conjugates, do not have commuting $\Z$-conjugates in general~\cite[Section 3.2]{companion}.
\end{rem}

\subsubsection{Portable manifolds}

If we consider the entire group of compactly supported diffeomorphisms, then it is $\Xsep$-boundedly acyclic for all portable manifolds. Let us start by recalling the definition.

\begin{defi}[Portable manifold \cite{burago2008conjugation}]
Let $M$ be a smooth connected open manifold. We say that $M$ is \emph{portable} if it admits a complete vector field $X$ and a compact subset $M_0$, called the \emph{core} of $M$, such that:
\begin{enumerate}
    \item $M_0$ is an attractor of the flow $X^t$ generated by $X$: For every compact subset $K \subset M$ there exists $\tau > 0$ such that $X^\tau(K) \subset M_0$;
    \item There exists $\theta \in \diffo(M)$ such that $\theta(M_0) \cap M_0 = \emptyset$.
\end{enumerate}
\end{defi}

\begin{rem}[Notation]
    Note that in the original paper by Burago, Ivanov and Polterovich~\cite{burago2008conjugation}, the authors use the notation $\textup{Diff}_0(M)$ for $\diffo(M)$.
\end{rem}

Any manifold that splits as the direct product of a closed manifold and $\R^n$ is portable, on the other hand compact manifolds cannot be portable. The key property of portable manifolds is the following:

\begin{lemma}[Displacement in portable manifolds {\cite[Lemma 3.1]{burago2008conjugation}}]
\label{lem:portable:displacement}
    Let $M$ be a portable manifold, and let $M_0$ be its core. Then there exists $t \in \diffo(M)$ such that $t^p(M_0) \cap M_0 = \emptyset$ for all $p \geq 1$.
\end{lemma}

We are now ready to prove:

\begin{cor}[Corollary \ref{corintro:portable}]
\label{cor:portable}
    Let $M$ be a portable manifold. Let $G$ be a group such that $\diffo(M) \leq G \leq \homeo(M)$. Then $G$ is $\Xsep$-boundedly acyclic.
\end{cor}

Note that we do not state this for derived subgroups since $\diffo(M)$ is perfect \cite[Chapter 2]{banyaga}.

\begin{proof}
We will prove that $\diffo(M)$ satisfies Lemma \ref{lem:ccc:dynamical}.\emph{(1)}.\ Let $K \subset M$ be compact.\ By definition, there exists $\tau > 0$ such that $X^\tau(K) \subset M_0$. Let $t$ be as in Lemma~\ref{lem:portable:displacement}.\ Then $s = X^{-\tau} t X^{\tau}$ belongs to $\diffo(M)$, since $\diffo(M)$ is normal in the full diffeomorphism group of $M$. Moreover
\[s^p(K) \cap K = X^{-\tau} (t^p X^{\tau}(K) \cap X^{\tau}(K) ) \subset X^{-\tau}(t^p(M_0) \cap M_0) = \emptyset\]
for all $p \geq 1$.
\end{proof}

\begin{example}[Portable manifolds]
\label{ex:portable}
    Corollary \ref{cor:portable} applies to homeomorphisms and diffeomorphism groups of all regularities, as well as their identity components, of manifolds of the form $\R^n \times M$, with $M$ closed.\ This situation was already studied by Monod--Nariman \cite{monodnariman}.\ It also applies to compactly supported diffeomorphisms and homeomorphisms of open handlebodies \cite{burago2008conjugation}.
\end{example}

\begin{rem}[Comparison with \cite{monodnariman}]
    In Corollary \ref{cor:portable}, we showed that diffeomorphism groups of portable manifolds have commuting $\Z$-conjugates.\ In fact, they even satisfy the stronger criterion for $\Xsep$-bounded acyclicity introduced by Monod~\cite[Corollary 5]{monod:thompson}.\ This requires cutting off the flow outside a compact set in order to obtain compactly supported diffeomorphisms mapping a given compact set into $M_0$, which is already a step in the proof of Lemma \ref{lem:portable:displacement} (see the proof of this fact by Burago, Ivanov and Polterovich~\cite[Lemma 3.1]{burago2008conjugation}).
\end{rem}

\subsubsection{Contactomorphisms}

A \emph{contact form} on a $(2n + 1)$-dimensional manifold $M$ is a $1$-form $\alpha \in \Omega^1(M)$ such that $\alpha \wedge (d\alpha)^n$ is a volume form. A \emph{contact structure} $\xi$ on $M$ is an equivalence class of contact forms, where $\alpha$ is equivalent to $\alpha'$ if there exists a positive function $\lambda$ such that $\alpha = \lambda \alpha'$. We denote by $\cont(M, \xi)$ the group of compactly supported contactomorphisms of $M$, i.e.\ the group of compactly supported diffeomorphisms of $M$ that preserve the contact structure $\xi$, and by $\conto(M, \xi)$ its identity component. The main example is that of $\R^{2n+1}$ with coordinates $(x_1, y_1, \ldots, x_n, y_n, z)$ endowed with its standard contact form $\alpha = dz + \sum_i x_i dy_i$.

Contact geometry is in many ways an odd-dimensional counterpart of symplectic geometry \cite[Chapter 6]{banyaga}, so it is natural to include contactomorphisms in our study. Since contact structures are only defined up to scalar, the situation here is much closer to that of the full diffeomorphism group. We will use again the framework of portable manifolds.

\begin{defi}[Contact portable manifold \cite{contact:portable}]
Let $(M, \alpha)$ be an open connected contact manifold. We say that $(M, \xi)$ is \emph{contact portable} if it admits a contact isotopy $\{P_t\}_{t > 0}$, $P_0 = \id_M$ and a connected compact set $M_0 \subset M$, called the \emph{core} of $M$, such that
\begin{enumerate}
    \item $M_0$ is an attractor of $P_t$: For every compact subset $K \subset M$ and every open neighbourhood $U_0 \supset M_0$ there exists $\tau > 0$ such that $P_\tau(K) \subset U_0$;
    \item There exists a contactomorphism $\theta$ such that $\theta(M_0) \cap M_0 = \emptyset$.
\end{enumerate}
\end{defi}

\begin{rem}[Notation]
    Note that in the original paper, the authors use the notation $\mathcal{G}(M)$ for $\conto(M, \xi)$~\cite{contact:portable}.
\end{rem}

The definition is very similar to that of portable manifolds, although there are some important differences: For instance $\theta$ is not required to be compactly supported. Still, $\R^{2n+1}$ is contact portable \cite[Example 3.12]{contact:portable}, and the key displacement property still holds:

\begin{lemma}[Displacement in contact portable manifolds {\cite[Proof of Proposition 3.13]{contact:portable}}]
    Let $(M, \xi)$ be a contact portable manifold, and let $M_0$ be its core. Then there exists an open neighbourhood $U_0 \supset M_0$ and $t \in \conto(M, \xi)$ such that $t^p(U_0) \cap U_0 = \emptyset$ for all $p \geq 1$.
\end{lemma}

With this lemma at hand, the exact same proof as Corollary \ref{cor:portable} shows:

\begin{cor}[Contactomorphism groups]
\label{cor:contact}
    Let $M$ be a contact portable manifold. Let $G$ be a group such that $\conto(M, \xi)^{(d)} \leq G \leq \homeo(M)$ for some $d \geq 1$. Then $G$ is $\Xsep$-boundedly acyclic.
\end{cor}

\begin{example}[Contactomorphism groups]
    This applies to the group of contactomorphisms of $\R^{2n+1}$, and to the group of contactomorphisms of contact handlebodies. However, it seems that these are the only known examples of contact portable manifolds~\cite[paragraph before Theorem~D]{contact:portable:examples}. We state this corollary for all derived subgroups since it seems to be open whether any of them is perfect \cite[Remark after Theorem 6.3.7]{banyaga}.
\end{example}

\subsubsection{Products}

We end this section by extending our results from $\R^n$ to products of the form $\R^n \times M$, where $M$ is a compact topological space. This kind of argument has already been considered by Monod and Nariman for the full diffeomorphism groups \cite{monodnariman}.

\begin{cor}[Transformation groups of products]
    \label{cor:tg:products}

    Let $n \geq 2$ and let $M$ be a compact topological space. Let $D \leq \homeo(\R^n)$ be one of the groups from Corollaries \ref{cor:tg:symp}, \ref{cor:tg:volume} or \ref{cor:contact}. Let $G$ be a group such that $D \times \{ \id \} \leq G \leq \homeo(\R^n \times M)$. Then $G$ is $\Xsep$-boundedly acyclic.
\end{cor}

Notice that we are requiring that $M$ be compact to ensure that homeomorphisms in $D \times \{\id\}$ are compactly supported.

\begin{proof}
    Such a group $D \leq \homeo(\R^n)$ satisfies Item~\emph{(1)} (in the case of $\conto$) or Item~\emph{(2)} (in the other cases) of Lemma \ref{lem:ccc:dynamical}.\ We claim that $D \times \{ \id \} \leq \homeo(\R^n \times M)$ also satisfies the same item, and the result follows.
    
    We prove this in the case that Item~\emph{(2)} is satisfied, the other case is identical.\ Let $K \subset \R^n \times M$ be compact.\ Then there exists a compact set $A \subset \R^n$ such that $K \subset A \times M$. Let $t \in D$ be such that $t(A) \cap A = \emptyset$ and $t^2|_A = \id|_A$.\ Then $(t, \id)(A \times M) \cap A \times M = (t(A) \cap A) \times M = \emptyset$, and $(t, \id)^2|_{A \times M} = t^2|_A \times \id|_M = \id|_{A \times M}$.
\end{proof}

\begin{example}[Transformation groups of products]
    Corollary \ref{cor:tg:products} applies to the following cases:
    \begin{enumerate}
        \item Suppose that $M$ is a closed symplectic manifold, and endow $\R^{2n} \times M$ with the product symplectic structure.\ Then $\ham(\R^{2n} \times M)$, $\sympo(\R^{2n} \times M)$ and $\symp(\R^{2n} \times M)$ are $\Xsep$-boundedly acyclic.
        \item Suppose that $M$ is a closed manifold endowed with a volume form, and endow $\R^n \times M$ with the product volume form.\ Then $\diffvol(\R^n \times M)$ and $\diffovol(\R^n \times M)$ are $\Xsep$-boundedly acyclic.
        \item Suppose that $M$ is a compact topological space endowed with a finite Borel measure, and endow $\R^n \times M$ with the product measure.\ Then the group of compactly supported measure-preserving homeomorphisms of $\R^n \times M$ is $\Xsep$-boundedly acyclic.
    \end{enumerate}
    It also applies to the analogous groups in all regularities.
\end{example}

It is important in our argument that $n \geq 2$. This motivates the following question:

\begin{question}
\label{q:cylinder}
    Let $M$ be a closed connected manifold endowed with a volume form, and let $\R \times M$ be endowed with the product volume form. Is $\diffvol(\R \times M)$ ($\Xsep$-)boundedly acyclic? In particular, is $\diffvol(\R \times S^1)$ ($\Xsep$-)boundedly acyclic?
\end{question}

\subsection{Piecewise linear and piecewise projective groups}\label{subsec:plpp}

We show that many piecewise linear and piecewise projective groups already known to be $2$-boundedly acyclic~\cite{fflodha} have commuting $\Z$-conjugates, possibly up to an amenable extension, thus proving all the parts of Corollary \ref{corintro:PL}.

The prime example of a piecewise linear group is Thompson's group $F$, which was already shown to be $\Xsep$-boundedly acyclic by Monod \cite{monod:thompson}.\ However Monod's method does not apply to all piecewise linear groups, for instance it does not apply to all subgroups of $F$.

\begin{cor}[Compactly supported groups without fixed points]
\label{cor:bsupp}

Let $\Gamma$ be a group of compactly supported homeomorphisms of the line acting without a global fixed point. Then $\Gamma$ is $\Xsep$-boundedly acyclic.
\end{cor}

\begin{proof}
We show that $\Gamma$ satisfies Lemma \ref{lem:ccc:dynamical}.\emph{(1)}, so it has commuting $\Z$-conjugates, whence it is $\Xsep$-boundedly acyclic. Let $K \subset \R$ be compact.\ Then it is contained in some compact interval $[a, b]$.\ Let $t \in \Gamma$ be an element such that $t(a) > b$, which exists because $\Gamma$ acts without a global fixed point. Since $\Gamma$ is compactly supported, it acts preserving the orientation, and thus $t^p(a) > t^{p-1}(b)$ for all $p \geq 1$.\ This shows that the sets $t^p(K)$ are pairwise disjoint for all $p \geq 1$, and we conclude.
\end{proof}

We will now show that groups of piecewise linear homeomorphisms of $[0,1]$ and piecewise projective homeomorphisms of $\R$ can always be expressed as an amenable extension of a group as in Corollary \ref{cor:bsupp}, and therefore are $\Xsep$-boundedly acyclic by Theorem \ref{thm:co-amenable}. This also applies directly to chain groups of homeomorphisms~\cite{chain} or groups that admit a coherent action on the line~\cite{coherent}, that were already known to have vanishing second bounded cohomology with trivial coefficients~\cite{fflodha}.

\begin{cor}[Chain groups and coherent actions]
\label{cor:coherent}
Let $\Gamma$ be either a chain group of homeomorphisms of the line, or a group that admits a coherent action on the line.\ Then $\Gamma$ is $\Xsep$-boundedly acyclic.
\end{cor}

Again using co-amenability and Theorem \ref{intro:thm:ccc} we can deduce the following:

\begin{cor}[Piecewise linear and piecewise projective groups]
\label{cor:PL}
Let $\Gamma$ be a group of piecewise linear homeomorphisms of the interval, or piecewise projective homeomorphisms of the line. Then $\Gamma$ is $\Xsep$-boundedly acyclic.
\end{cor}

\begin{proof}
Let $\Gamma$ be such a group, and suppose first that $\Gamma$ is orientation-preserving. Under this assumption, it is known that there exists a subgroup $\Gamma_0 \leq \Gamma$ such that the quotient $\Gamma / \Gamma_0$ is amenable and $\Gamma_0$ has commuting conjugates~\cite[Theorem 1.3]{fflodha}.\ More precisely, whenever $H \leq \Gamma_0$ is finitely generated, there exist finitely many pairwise disjoint open intervals $J_1, \ldots, J_k$ such that $H$ is supported on their union and preserves each one, and pairwise disjoint open intervals $I_i \supset J_i$ each preserved by $\Gamma_0$, and an element $t \in \Gamma_0$ such that $t(J_i)$ is disjoint from $J_i$. Since $t$ is orientation-preserving, it follows that $t^p(J_i)$ is disjoint from $J_i$ for all $p \geq 1$. Moreover, since $J_i \subset I_i$ and the latter is preserved by $\Gamma_0$, we have $t^p(J_i) \subset I_i$, which is disjoint from $I_j$, and thus from $J_j$, for all $i \neq j$. Thus the support of ${}^{t^p} H$ is disjoint from that of $H$ for all $p \geq 1$. This shows that $\Gamma_0$ has commuting $\Z$-conjugates, and so is $\Xsep$-boundedly acyclic by Theorem \ref{intro:thm:ccc}. By Theorem~\ref{thm:co-amenable} we conclude that $\Gamma$ itself is $\Xsep$-boundedly acyclic.

In the general case, $\Gamma$ contains an orientation-preserving subgroup $\Gamma_+$ of index at most $2$. We showed that $\Gamma_+$ is $\Xsep$-boundedly acyclic, so it follows again from Theorem \ref{thm:co-amenable} that $\Gamma$ is $\Xsep$-boundedly acyclic.
\end{proof}

Note that in the proofs of both Corollary \ref{cor:coherent} and Corollary \ref{cor:PL}, the $\Xsep$-bounded acyclicity is proved by exhibiting such groups as amenable extensions of groups with commuting $\Z$-conjugates. Therefore we can apply Corollary \ref{cor:coameanble:modulus} to obtain the following strengthening:

\begin{cor}[Vanishing modulus of piecewise linear and piecewise projective groups]
\label{cor:PL:modulus}

For every integer $n \geq 1$, there exists a constant $K_n \in [0, \infty)$ such that for every group $\Gamma$ as in Corollary \ref{cor:coherent} or Corollary \ref{cor:PL}, and every separable dual Banach $\Gamma$-module $E$, we have that $\|H^n_b(\Gamma; E) \| \leq K_n$.
\end{cor}

The group $G_0$ is a group of orientation-preserving piecewise projective homeomorphisms of the line, which is non-amenable and finitely presented~\cite{LodhaMoore}, in fact, it is even of type $F_\infty$~\cite{LodhaFinf}.\ The previous result readily implies the following:

\begin{cor}[Corollary \ref{corintro:alternative}]
\label{cor:alternative}
There exists a group of type $F_\infty$ that is non-amenable, but all of whose subgroups are $\Xsep$-boundedly acyclic.
\end{cor}

A homological version of the von Neumann--Day problem proposed by Calegari \cite{scl_pl} asks whether a finitely presented torsion-free group all of whose subgroups have vanishing stable commutator length is necessarily amenable.\ Recently, this was shown to fail, even when one asks for the vanishing of second bounded cohomology: The desired  counterexample is the piecewise projective group $G_0$~\cite{fflodha}. Corollary \ref{cor:alternative} can be thus seen as an even stronger negative solution to this problem.

\subsection{Highly transitive actions}
\label{ss:highlytrans}

We take a detour from producing examples of groups with commuting cyclic conjugates, and show how the previous results can be used as a basis for more complex computations of bounded cohomology rings.\ The main idea is to produce highly transitive actions with stabilizers falling in the classes of boundedly acyclic groups that were presented so far. This method was already used with success by several authors \cite{fflm2, monodnariman, fflodha, konstantin}, and as in those cases we restrict to trivial coefficients for simplicity.

\begin{cor}[Corollary \ref{corintro:circleaction}]
\label{cor:circleaction}

Let $\Gamma$ be a group of orientation-preserving piecewise linear homeomorphisms of the circle, or piecewise projective homeomorphisms of the projective line. Suppose that there exists an infinite orbit $S$ of $\Gamma$ such that, for all $n \geq 1$, $\Gamma$ acts transitively on positively oriented $n$-tuples in $S$. Then there exists an isomorphism
\[H^\bullet_b(\Gamma; \R) \cong \R[x],\]
where $|x| = 2$, and $x$ corresponds to the Euler class of the defining action on the circle of $\Gamma$.
\end{cor}

The result is a generalization of the computation of the bounded cohomology of Thompson's group $T$ \cite{fflm2}. In that case, the proof relied on the bounded acyclicity of Thompson's group $F$ \cite{monod:thompson}, but thanks to Corollary \ref{cor:PL} we can drop all hypotheses on point stabilizers.

\begin{proof}
Corollary \ref{cor:PL} implies that subgroups of $\Gamma$ fixing a point on the circle are boundedly acyclic, since they are isomorphic to a group of piecewise linear homeomorphisms of the interval, or a group of piecewise projective homeomorphisms of the line. Therefore we can apply \cite[Proposition 6.9]{fflm2} and conclude.
\end{proof}

\begin{example}[Groups acting on the circle]
\label{ex:circleaction}

Corollary~\ref{cor:circleaction} applies to Thompson's group $T$, but also to the Stein--Thompson groups~\cite{stein} of the form $T_{2, n_2, \ldots, n_k}$.\ The (circular) high transitivity is proved in \cite[Example 3.7]{monsters} (note that the reference assumes that $n_2 = 3$, but this is not used in the proof of high transitivity).\ It also applies in other examples~\cite[Examples~3.8 and 3.9]{monsters} including the golden ratio Thompson's group $T_\tau$ and its commutator subgroup~\cite{cleary, ttau} and the piecewise projective group $S$~\cite{lodhaS}.
\end{example}

Another context in which highly transitive actions yield computations of bounded cohomology is group actions on the line. This is a slight variation of the method for groups acting on the circle, however to our knowledge it is the first time that it is used in this context.

\begin{prop}[Groups acting on the line]
\label{prop:lineaction:general}

Let $\Gamma$ be a group of orientation-preserving homeomorphisms of the line, and let $S$ be an infinite orbit of $\Gamma$. Suppose that, for all $n \geq 1$, the action of $\Gamma$ on linearly ordered $n$-tuples in $S$ is transitive and has boundedly acyclic stabilizers. Then $\Gamma$ is boundedly acyclic.
\end{prop}

\begin{proof}
Since there are finitely many orbits of unordered $n$-tuples, and each has boundedly acyclic stabilizer, we can compute $H^\bullet_b(\Gamma; \R)$ using the complex of invariant alternating cochains $\ell^\infty_{alt}(S^{\bullet+1}, \R)^\Gamma$~\cite[Corollary 5.15, Remark 5.17, Example A.11]{liloehmoraschini}.\ Recall that a cochain $f \in \ell^\infty(S^{\bullet+1}, \R)$ is \emph{alternating} if $f(s_{\sigma(0)}, \ldots, s_{\sigma(n)}) = \textup{sign}(\sigma) f(s_0, \ldots, s_n)$ for all $\sigma \in \textup{Sym}\{0, \ldots, n\}$.

By the transitivity assumption, for all $n \geq 1$ the space $\ell^\infty(S^{n+1}; \R)^\Gamma$ is one-dimensional, spanned by the cochain $f_n$ defined by
$f_n(s_0, \ldots, s_n) = 1$ if $s_0 < s_1 < \cdots < s_n$, which can be extended uniquely to an alternating function, and is automatically $\Gamma$-invariant since $\Gamma$ acts by preserving the orientation. Moreover, if $s_0 < s_1 < \cdots < s_n < s_{n+1}$, we have
\[ \delta(f_n)(s_0, \ldots, s_{n+1}) = \sum\limits_{i = 0}^{n+1} (-1)^i =
\begin{cases}
0 & \text{if } n \text{ is even}; \\
1 & \text{if } n \text{ is odd}.
\end{cases}
\]
This shows that $\delta^n = 0$ if $n$ is even, and $\delta^n(f_n) = f_{n+1}$ if $n$ is odd. Therefore the complex $\ell^\infty(S^{n+1}; \R)^\Gamma$ is isomorphic to the complex
\[ \cdots \to \R \xrightarrow{\textup{id}} \R \xrightarrow{0} \R \to \cdots\]
whose cohomology is trivial in all positive degrees.
\end{proof}

\begin{cor}[Corollary \ref{corintro:lineaction}]
\label{cor:lineaction}

Let $\Gamma$ be a group of orientation-preserving homeomorphisms of the line, and let $S$ be an infinite orbit of $\Gamma$. Suppose that for all $n \geq 1$, the action of $\Gamma$ on linearly ordered $n$-tuples in $S$ is transitive. Suppose moreover that, for $s \in S$, the stabilizer $\Gamma_s$ admits a co-amenable subgroup $\Gamma_s^0$, such that every finitely generated subgroup of $\Gamma_s^0$ is isomorphic to a group of piecewise linear homeomorphisms of the interval. Then $\Gamma$ is boundedly acyclic.
\end{cor}

\begin{proof}
By Proposition \ref{prop:lineaction:general}, it suffices to show that stabilizers of $n$-tuples in $S$ are boundedly acyclic.\ In fact, we will show that every subgroup of $\Gamma_s$ is boundedly acyclic, and since the action of $\Gamma$ on $S$ is transitive, this will suffice.\ By Theorem \ref{thm:co-amenable}, it is enough to show that the co-amenable subgroup $\Gamma_s^0$ is boundedly acyclic.

Now $\Gamma_s^0$ is the direct union of its finitely generated subgroups. Each of these is isomorphic to a group of piecewise linear homeomorphisms of the interval.\ Corollary \ref{cor:PL} shows that these are $\Xsep$-boundedly acyclic, and moreover Corollary \ref{cor:PL:modulus} shows that their vanishing moduli are uniformly bounded.\ Therefore we can apply Proposition \ref{prop:dirun} and obtain that $\Gamma_s^0$ is $\Xsep$-boundedly acyclic, which concludes the proof.
\end{proof}

\begin{example}[Finitely generated left orderable simple boundedly acyclic groups]

In his 2018 ICM address, Navas asked a series of questions about the possible landscape of finitely generated left orderable groups (equivalently, finitely generated groups of orientation-preserving homeomorphisms of the line) \cite{navas}. One of these questions asks whether there exists such a group $\Gamma$ with $H^2_b(\Gamma; \R) = 0$, but which does not surject onto $\Z$ \cite[Question 8]{navas}. This is motivated in part by the Witte-Morris Theorem \cite{witte}, which states that an amenable left orderable group must surject onto $\Z$.

An answer to Navas's question is given by the first examples of infinite finitely generated simple left orderable groups $G_\rho$~\cite{grho} (here $\rho$ is a combinatorial object called a \emph{quasi-periodic labelling}), since they have vanishing second bounded cohomology with trivial real coefficients~\cite{fflodha}. Corollary \ref{cor:lineaction} shows that $G_\rho$ is in fact boundedly acyclic. Indeed, the action is (linearly) highly transitive on the orbit $\Z[1/2]$ (by virtue of containing the natural copy of Thompson's group $F$), and the condition on point stabilizers is satisfied~\cite[Proposition 4.16 and proof of Proposition~4.17]{fflodha}.

There are by now several more examples of finitely generated simple left orderable groups, constructed by using dynamical systems on Cantor spaces and flows. The first instance of this approach were the groups $T(\varphi)$ by Matte Bon and Triestino~\cite{tphi} (where $\varphi$ is a minimal homeomorphism of a Cantor space). Later a similar construction that encompasses also the groups $G_\rho$ has been carried out by Matte Bon and Le Boudec, producing the groups $T(\varphi, \sigma)$ \cite{tphisigma} (where $(\varphi, \sigma)$ is a minimal $D_\infty$-dynamical system on a Cantor space).\ All of these act on the line, (linearly) highly transitively on a copy of $\Z[1/2]$ by virtue of containing a copy of Thompson's group $F$.\ The condition on point stabilizers is also satisfied in the cases of the groups $T(\varphi)$~\cite[Lemma 7.2]{tphi} and $T(\varphi, \sigma)$~\cite[Lemma 4.8]{tphisigma}.
\end{example}

Similar in spirit to the proofs above is the strategy followed by de la Harpe and McDuff to show that certain large groups of automorphisms, such as the group of all permutations of an infinite set, or the group of all continuous linear automorphisms or of all invertible isometries of an infinite-dimensional Hilbert space, are acyclic \cite{delaHarpe_Mcduff_1983}. They first show that the stabilizers of suitably defined flags are acyclic; then they build on that to construct appropriate classifying spaces for the entire groups and show that they are acyclic too.
In fact, the proof of acyclicity of stabilizers \cite[Lemma 3]{delaHarpe_Mcduff_1983} shows that they have commuting $\Z$-conjugates, so that they are also $\Xsep$-boundedly acyclic. However it is less clear whether the second part of the argument can be adapted to the bounded cohomology setting. We can thus ask:
\begin{question}
\label{q:largegroups}
    Are the groups considered in \cite[Theorem 1]{delaHarpe_Mcduff_1983} boundedly acyclic?
\end{question}

\subsection{Interval exchange transformations}

Recall that an \emph{interval exchange transformation} of the half-open interval $[0, 1)$ is a right-continuous bijection given by cutting $[0, 1)$ into finitely many half-open subintervals and rearranging by translations.\ We denote by $\iet([0, 1))$ the group of interval exchange transformations of $[0, 1)$.\ We can define analogously the group $\iet([0, \infty))$, where again we cut into finitely many half-open subintervals.

\begin{lemma}
\label{lem:iet:dirun}
    $\iet([0, \infty))$ is the direct union of the groups $\{ \iet([0, n)) \}_{n \geq 1}$.\ In particular, $\iet([0, \infty))$ is amenable, respectively contains non-abelian free subgroups, if and only if $\iet([0, 1))$ does.
\end{lemma}

\begin{proof}
    The first statement follows from the fact that the unique half-open interval of infinite length has to be preserved by an $\iet$.\ The second statement follows from the isomorphism $\iet([0, 1)) \cong \iet([0, n))$ given by affine conjugacy, and the fact that direct unions preserve amenability and the property of not containing free subgroups.
\end{proof}

To our knowledge, nothing is known about the bounded cohomology of $\iet([0, 1))$.\ For the half line, we obtain:

\begin{cor}[Corollary \ref{corintro:iet}]
    \label{cor:iet}

    The group $\iet([0, \infty))$ is $\Xsep$-boundedly acyclic, and so are its derived subgroups.
\end{cor}

\begin{proof}
    By Lemma \ref{lem:iet:dirun}, a finitely generated subgroup $H \leq \iet([0, \infty))$ is contained in $\iet([0, n))$ for some $n \geq 1$. Let $t$ be the $\iet$ exchanging $[0, n)$ with $[n, 2n)$.\ Then ${}^t H \leq \iet([n, 2n))$ and thus $[H, {}^t H] = 1$. Moreover $[H, t^2] = [H, 1] = 1$, so $\iet([0, \infty))$ has commuting cyclic conjugates, and we conclude by Theorem \ref{intro:thm:ccc}. The statement on derived subgroups follows from Lemma \ref{lem:derived}.
\end{proof}

\begin{rem}
\label{rem:ret}
    Corollary \ref{cor:iet} can be easily generalized to groups of interval exchange transformations of the real line, to interval exchange transformations with flips \cite{iet:up}, and to higher dimensions \cite{RET}, as long as these are defined by allowing only finitely many intervals or rectangles to be exchanged.
\end{rem}

The infinite length of $[0, \infty)$ plays an important role in our proof.\ This motivates the following question:

\begin{question}
    \label{q:iet}

    Is the group $\iet([0, 1))$ ($\Xsep$-)boundedly acyclic?
\end{question}

Note that a negative answer would imply that $\iet([0, 1))$ is non-amenable.

\subsection{Direct limit groups in algebraic $K$-theory}\label{subsec:directlimitgroups}

We show that the direct limit of the classical linear groups and some groups appearing in algebraic $K$-theory are all $\Xsep$-boundedly acyclic.\ The cases of direct limits of $\textup{GL}_n(R)$, $\textup{Sp}_{2n}(R)$ and $\textup{O}_{n, n}(R)$ give an answer to a question recently asked by Kastenholz and Sroka~\cite[Question 1.4]{kastenholz-sroka} in the context of stability in bounded cohomology.

\begin{cor}[Corollary \ref{corintro:linear}]
\label{cor:linear}
   Let $R$ be a ring with multiplicative unit $1$. The following groups are $\Xsep$-boundedly acyclic:
   \begin{enumerate}
       \item $\textup{GL}(R)=\cup_{n\geq 1} \textup{GL}_n(R)$;\label{GL}
       \item $\textup{SL}(R)=\cup_{n\geq 1} \textup{SL}_n(R)$;\label{SL}
       \item $\textup{O}(R)=\cup_{n\geq 1} \textup{O}_n(R)$;\label{O}
       \item $\textup{SO}(R)=\cup_{n\geq 1} \textup{SO}_n(R)$;\label{SO}
        \item $\textup{Sp}(R)=\cup_{n\geq 1} \textup{Sp}_{2n}(R)$;\label{Sp}
        \item $\textup{E}(R)=\cup_{n\geq 1} \textup{E}_n(R)$;\label{E}
        \item $\textup{colim}_{n\rightarrow \infty}\textup{O}_{n, n}(R)=\cup_{n\geq 1}\textup{O}_{n, n}(R)$;\label{O+-}
        \item $\textup{St}(R)=\cup_{n\geq 1} \textup{St}_{n}(R)$.\label{St}
   \end{enumerate}
   For the groups \eqref{GL} -- \eqref{SO} and \eqref{E} we consider the left upper corner inclusions:
   \begin{align*}
   \textup{GL}_n(R)&\longrightarrow \textup{GL}_{n+1}(R)\\
       M&\longmapsto \left(\begin{array}{cc}
       M &0  \\
        0&1 
   \end{array}\right).
   \end{align*}
   We realize $\textup{Sp}_{2n}(R)$ as the subgroup of $\textup{GL}_{2n}(R)$ preserving the symplectic matrix $\left(\begin{array}{cc}
       0 &I_n  \\
        -I_n&0
   \end{array}\right)$, and for \eqref{Sp} we consider the inclusion
   \begin{align*}
   \textup{Sp}_{2n}(R)&\longrightarrow \textup{Sp}_{2n+2}(R)\\
       \left(\begin{array}{cc}
       M &N  \\
        R&S 
   \end{array}\right)&\longmapsto \left(\begin{array}{cc}
        \begin{array}{cc}
       M &0  \\
        0&0
   \end{array} &\begin{array}{cc}
       N &0  \\
        0&1
   \end{array}\\
   \begin{array}{cc}
       R &0  \\
        0&-1
   \end{array}&\begin{array}{cc}
       S &0  \\
        0&0
   \end{array}\end{array}\right).
   \end{align*}
   We realize $\textup{O}_{n,n}(R)$ as the subgroup of $\textup{GL}_{2n}(R)$ preserving the bilinear form defined by the matrix $I_n \otimes \begin{pmatrix}
       1 & 0 \\ 0 & -1
   \end{pmatrix}$
   and consider the left upper corner inclusions $\textup{O}_{n,n}(R)\hookrightarrow \textup{O}_{n+1,n+1}(R)$.
   The group \eqref{St} is the universal central extension of the group \eqref{E}.
\end{cor}
\begin{proof}
\emph{Ad \eqref{GL} -- \eqref{E}} Let us denote by $\Gamma_{(i)}=\cup_{n\geq 1}\Gamma_{(i)}^n$ the groups appearing in the statement for $1 \leq i \leq 5$. We are going to show that $\Gamma_{(i)}$ has commuting cyclic conjugates, so that we will conclude by Theorem~\ref{intro:thm:ccc}. For $i = 6$, we can just apply Lemma \ref{lem:derived} since $[\textup{GL}(R), \textup{GL}(R)]=\textup{E}(R)$: This is called Whitehead Lemma, see for example \cite[Lemma 3.1]{milnor1971introduction}. 
    
    Consider the infinite alternating group $\mathrm{Alt}_\infty=\cup_{n\geq 1}\mathrm{Alt}(n)$. Note that $\mathrm{Alt}_\infty\leq \Gamma_{(i)}$. Indeed, for $1 \leq i \leq 5$, we see that $\mathrm{Alt}(n)\leq \Gamma_{(i)}^n$ for every $n\geq 1$: For $1 \leq i \leq 4$, we have the usual inclusion $\sigma\mapsto M_\sigma$ by permuting the basis vectors according to the permutation $\sigma$; whereas for $i = 5$ we can choose the inclusion 
   \begin{align*}
       \mathrm{Alt}(n) &\longrightarrow \textup{Sp}_{2n}(R)\\
       \sigma &\longmapsto \left(\begin{array}{cc}
       M_\sigma &0  \\
        0&M_\sigma
   \end{array}\right).
   \end{align*}
   This indeed is a symplectic matrix. We are now ready to prove that $\Gamma_{(i)}$ has commuting cyclic conjugates.

    Take a finitely generated subgroup $H\leq \Gamma_{(i)}$. Then $H\leq \Gamma_{(i)}^n$ for some integer $n\geq 1$.\ Up to taking $n+1$ instead of $n$, we can assume that $n$ is even.\ Inside $\mathrm{Alt}_\infty$ we now identify a particular order $2$ element:
    $$t=(1, n+1)\cdots(n, 2n)\in \mathrm{Alt}_\infty.$$
    Under the above identifications, it is an element of $\Gamma_{(i)}$  for $1 \leq i \leq 6$.
    Then $[H, {}^{t}H]=1$, as ${}^{t}H$ leaves the first $n$ basis vectors unchanged, and $[H, t^2]=[H, \mathrm{id}]=1$ by definition of $t$.\
    This shows that all the groups $\Gamma_{(1)}, \ldots, \Gamma_{(6)}$ have commuting cyclic conjugates and so they are $\Xsep$-boundedly acyclic by Theorem~\ref{intro:thm:ccc}.

\emph{Ad \eqref{O+-}} The order $2$ permutation 
$$(1, 2n+1)(2, 2n+2)\cdots(2n-1, 4n-1)(2n, 4n)\in\mathrm{Alt}_\infty$$
can be realized inside the group $\textup{O}_{2n,2n}(R)$ by exchanging the first $2n$ basis vectors with the last $2n$ basis vectors.\ We call this transformation $t$.\
Then the same argument as above shows that if for every finitely generated subgroup $H$, taking $n$ such that $H\leq \textup{O}_{n, n}(R)$, the element $t$ above satisfies $[H, {}^{t}H]=1$ and $[H, t^2]=[H, \id]=1$, so we conclude again by Theorem \ref{intro:thm:ccc}.

\emph{Ad \eqref{St}} The Steinberg group $\textup{St}(R)$ is the universal central extension of $\textup{E}(R)$, that is $\Xsep$-boundedly acyclic as we proved above. The kernel of the extension is abelian, hence amenable.\ Then by applying the inflation version of Gromov's Mapping Theorem (Remark \ref{rem:mapping:converse}) to the surjective homomorphism $\textup{St}(R)\rightarrow\textup{E}(R)$ with amenable kernel, we obtain the desired result.
\end{proof}

\begin{rem}[Commuting cyclic vs commuting $\Z$-conjugates]
    Note that all the groups considered in Corollary \ref{cor:linear}, though they have commuting cyclic conjugates, do not have commuting $\Z$-conjugates in general~\cite[Section 3.2]{companion}.
\end{rem}

\subsection{Algebraically closed groups}

We start by recalling the definitions:

\begin{defi}
\label{defi:equations}
    Let $\Gamma$ be a group. A finite system of equations and inequations with constants in $\Gamma$
    \begin{align*}
        f_i(g_1, \ldots, g_n; x_1, \ldots, x_m) &= e : 1 \leq i \leq k; \\ f_j(g_1, \ldots, g_n; x_1, \ldots, x_m) &\neq e : k < i \leq l
    \end{align*}
    is said to be \emph{consistent} if there exists an overgroup $\Gamma \leq \Lambda$ over which the system has a solution $(x_1, \ldots, x_m) \in \Lambda^m$.

    A group is said to be \emph{algebraically closed} if every consistent finite system of equations with constants in $\Gamma$ has a solution in $\Gamma$. A group is said to be \emph{existentially closed} if every consistent finite system of equations and inequations with constants in $\Gamma$ has a solution in $\Gamma$.
\end{defi}

Algebraically closed groups were introduced by Scott \cite{algclosed:scott}. Shortly thereafter, B. H. Neumann proved that the classes of algebraically closed and existentially closed groups coincide, and that all algebraically closed groups are simple \cite{algclosed:BH}. Algebraically closed groups received much attention in the early days of combinatorial group theory and have been at the center of many important developments. For instance, MacIntyre's construction of countable algebraically closed groups is among the first instances of the use of model theoretic methods in group theory \cite{algclosed:mac}.

In the paper where they introduced mitotic groups, Baumslag--Dyer--Heller proved that algebraically closed groups are mitotic \cite{BDH}. It follows that algebraically closed groups are acyclic \cite{BDH} and boundedly acyclic \cite{Loeh:dim}. We extend this to encompass non-trivial coefficients:

\begin{cor}[Corollary \ref{corintro:algclosed}]
\label{cor:algclosed}
    Algebraically closed groups are $\Xsep$-boundedly acyclic.
\end{cor}

\begin{proof}
    We show that every algebraically closed group has commuting cyclic conjugates, and conclude by Theorem \ref{intro:thm:ccc}. Let $\Gamma$ be an algebraically closed group, and let $H = \langle g_1, \ldots, g_n \rangle \leq \Gamma$ be finitely generated. We consider the system of equations
    \[[g_i, {}^x g_j] = e; \qquad [g_i, x^2] = e \qquad : \qquad 1 \leq i, j \leq n.\]
    This system is consistent, because it has a solution in the overgroup $\Gamma \wr \Z/2$, by setting $x$ to be the generator of $\Z/2$. Since $\Gamma$ is algebraically closed, there exists a solution $t \in \Gamma$. This solution must satisfy $[H, {}^t H] = 1$ and $[H, t^2] = 1$. This shows that $\Gamma$ has commuting cyclic conjugates and concludes the proof.
\end{proof}

\subsection{The generic group}\label{subsec:generic:group}

By an \emph{enumerated group}, we mean a (countably infinite) group whose underlying set is $\N$. We let $\mathcal{G}$ denote the set of enumerated groups.\
To each enumerated group $G$, we assign an associated multiplication function $\mu_G\colon\N\times \N\to \N$, an inversion function $\iota_G\colon\N\to \N$
and an identity element $e_G\in \N$. It follows that each element of $\mathcal{G}$ is an element of the zero dimensional Polish space 
$\N^{\N\times \N}\times \N^{\N}\times \N$.\
It is easy to show that $\mathcal{G}$ is a closed subspace of this space, and hence a zero dimensional Polish space in its own right.\
This endows the set of enumerated groups with a natural topology, which was extensively studied~\cite{GKL} (for both the full space and various natural subspaces).

We consider group theoretic words $w=w(x_1,\ldots,x_n)$ (here group-theoretic means that the alphabet includes formal inverses of $x_1,\ldots,x_n$).\
By a \emph{system} we mean a finite system of equations and inequations of the form $w=e$ or $w\neq e$, where $w$ is a word.
We use symbols such as $\Sigma$ and $\Delta$ to denote systems.\
If the equations and inequations in a system $\Sigma$ use variables among $x_1,\ldots,x_n$, we write $\Sigma=\Sigma(x_1,\ldots,x_n)$.\
Given an ordered $n$-tuple $\overrightarrow{a}=(a_1,\ldots,a_{n})$ of elements in $\N$, we denote the \emph{system with coefficients}
$\Sigma(\overrightarrow{a})=\Sigma(a_1,\ldots,a_{n})$.

In this section, we will demonstrate the existence of a comeager set in $\mathcal{G}$ in which every element has commuting cyclic conjugates, and hence is $\Xsep$-boundedly acyclic.\
To do so, we shall need the following method of \emph{model-theoretic forcing}.\
Consider a two-player game where the players take turns playing a system of equations and inequations to define a group structure on $\N$.
For simplicity, throughout the game, the number $1$ is assumed to be the identity element.\
In round $n\in \N\setminus \{0\}$, player one plays a system with coefficients $\Sigma_n(\overrightarrow{a_n})$ and subsequently player two plays $\Delta_n(\overrightarrow{b_n})$
(with no restrictions on the lengths of $\overrightarrow{a_n}, \overrightarrow{b_n}$).\
There are three rules of the game:

\begin{enumerate}
\item At the end of each play, there is an enumerated group $G$ which satisfies all systems played until then, that is $\bigcup_{1\leq i \leq n-1}(\Sigma_i(\overrightarrow{a_i})\cup \Delta_i(\overrightarrow{b_i}))\bigcup \Sigma_n(\overrightarrow{a_n})$ right after player one has played its $n$-th play, and $\bigcup_{1\leq i \leq n}(\Sigma_i(\overrightarrow{a_i})\cup \Delta_i(\overrightarrow{b_i}))$ after player one and player two have both played their $n$-th play.
\item For each $m_1,m_2\in \N$, there is a round $n\in \N$ and an $m_3\in \N$ such that $m_1\cdot m_2=m_3$ appears in a system played at this round.
\item For each $m_1\in \N$, there is a round $n\in \N$ and an $m_2\in \N$ such that $m_1\cdot m_2=1$ appears in a system played at this round.
\end{enumerate}
The game is played for infinitely many steps, and the (unique) enumerated group satisfying $\bigcup_{n\in \N}(\Sigma_n(\overrightarrow{a_n})\cup \Delta_n(\overrightarrow{b_n}))$ is called the \emph{compiled group} 
of the game.

\begin{rem}[Necessity]
The items $(2),(3)$ in the rules ensure that the result of the infinite game is a group structure on $\N$,
not merely a set of relations that the group generated by $\N$ satisfy.
\end{rem}

We say that a property $P$ of groups is \emph{enforceable} if there is a strategy for player two that ensures that the compiled group satisfies $P$. \
We say that a property $P$ is \emph{Baire measurable} if the class of groups $\mathcal{C}_P\subseteq \mathcal{G}$ that satisfies $P$ is a Baire measurable subset of $\mathcal{G}$.\
We say that $P$ is \emph{invariant} if it is a property of the isomorphism type of the underlying group, and does not depend on the enumeration itself.\
It is immediate that the property of having commuting cyclic conjugates is invariant. We may also verify that it is Baire measurable, as follows.\
First, we note that for a fixed finitely generated subgroup generated by $a_1,\ldots,a_n\in \N$, the existence of $b\in \N$ that realizes the commuting cyclic conjugates condition 
can be expressed as a countable disjunction of quantifier-free formulas.\
It follows that having commuting cyclic conjugates for the whole group is a countable conjunction of these, and hence the property is Baire measurable.\
The following theorem, an easy version of a more general statement~\cite[Theorem 5.2.3]{GKL}, establishes the connection between Baire category and enforceability:

\begin{thm}[Enforceable properties]
\label{Theorem:modelforcing}
Suppose that $P$ is an invariant Baire measurable property.
Then $\mathcal{C}_P$ is a comeager subset of $\mathcal{G}$ if and only if $P$ is an enforceable property. 
\end{thm}

Using this method, we are now ready to prove our result.

\begin{cor}[Corollary \ref{corintro:generic}]
\label{cor:generic}
    The generic enumerated group is $\Xsep$-boundedly acyclic.
\end{cor}

Note that the generic enumerated group contains non-abelian free subgroups (in fact, it contains a copy of every finitely generated group with solvable word problem \cite[Theorem 1.1.2]{GKL}), and therefore is non-amenable.

\begin{proof}
We will show that the property of having commuting cyclic conjugates is enforceable, by demonstrating a strategy for player two in the aforementioned game.
At stage $n$, it is the goal of player two to ensure that in any compiled group $H$, there exists an element $f\in \N$ such that 
for $\{h_1=1,\ldots,h_n=n\}$, the following relations hold:
$$\{[{}^{f} h_i,h_j]=1,[h_i,f^2]=1\mid 1\leq i,j\leq n\}$$
Let $\Sigma_n$ be the union of systems of equations and inequations that have been played by both players until stage $n-1$, 
and the system played by player one at stage $n$.
Let $k_n\in \N$ be such that all numbers that occur as coefficients in $\Sigma_n$ are smaller than $k_n$ and, moreover, $k_n>n$.
Let $f=k_n$.
At this stage, player two plays the union $\Delta_n\cup \Delta_n'$ where:
$$\Delta_n=\{[{}^{f} h_i,h_j]=1\mid \forall 1\leq i,j\leq n \}\cup \{[h_i,f^2]=1\mid 1\leq i\leq n\}$$
and the system $\Delta_n'$ that consists of $h_i\cdot h_j=h_{i,j}$ and $h_i\cdot \bar{h}_i=1$ for $1\leq i,j\leq n$ and some suitable $h_{i,j},\bar{h}_i\in \N$ in some enumeration of a group satisfying
$\Sigma_n\cup \Delta_n$.

It suffices to show that $\Delta_n$ itself is a legal move, and the legality of $\Delta_n\cup \Delta_n'$ will follow from construction.
Our proof will follow from applying Theorem \ref{Theorem:modelforcing}.
To see that the move is legal, let $\Gamma$ be an enumerated group that satisfies $\Sigma_n$.
It suffices to show that there exists an enumerated group $G$ that satisfies $\Sigma_n\cup \Delta_n$.
Consider the group $G=(\bigoplus_{i\in \{0,1\}}\Gamma_i)\rtimes \Z$ where each $\Gamma_i$ is a copy of $\Gamma$,
and the automorphism given by the generator of $\Z$ permutes the two copies (and hence the square of the generator of $\Z$ commutes with each copy).
We can find an enumeration of $G$ which satisfies $\Sigma_n\cup \Delta_n$, finishing our proof.
\end{proof}

\begin{rem}[Subspaces of enumerated groups]
Let $P$ be one of the following properties of groups: Left orderable, sofic, locally indicable, torsion-free, contains no non-abelian free subgroups, and non-amenable.
Note that these properties are closed under taking semidirect products of the kind described in the last paragraph of the previous proof (for soficity, this was proved by Hayes and Sale~\cite{HayesSale} more generally for wreath products, but in our cases it suffices to use soficity of amenable extensions \cite{elekszabo}).\ Moreover, it is known that the subspace $\mathcal{G}_P\subseteq \mathcal{G}$ consisting of groups satisfying $P$
is also a Polish space~\cite{GKL}.\ Our proof above readily adapts to provide a proof (using a more general statement than Theorem \ref{Theorem:modelforcing} above~\cite[Theorem 5.2.3]{GKL}) that there exists a comeager subset $\mathcal{C}_P\subset \mathcal{G}_P$ such that 
each enumerated group in $\mathcal{C}_P$ is $\Xsep$-boundedly acyclic.

In particular, the generic group without free subgroup is $\Xsep$-boundedly acyclic, so for instance it admits no non-trivial quasimorphism.\ Whether this holds for \emph{all} groups without free subgroups is an intriguing question, which is open to the best of our knowledge.\ Partial results were obtained by Manning \cite{manning:geometry} and Calegari \cite{calegari:law}.\ A closely related question is whether groups without free subgroups are Ulam stable \cite[Question 1.10]{thompson:ulam}.\ Note that the generic group without free subgroups is non-amenable \cite[Corollary 1.3.1]{GKL}.
\end{rem}

\section{Bounded acyclicity and rigidity}
\label{s:rigidity}

We end with a short discussion on $\Xsep$-bounded acyclicity and its role in the study of \emph{representations} $\rho \colon \Gamma \to G$, where $G$ is the isometry group of a non-positively curved space~$X$. Recall that associated with each representation $\rho \colon \Gamma \to G$ there is an induced action of $\Gamma$ on the space~$X$. Such actions produce \emph{bounded characteristic classes} in the bounded cohomology of $\Gamma$ and actions with a trivial bounded characteristic class are `elementary' in a suitable sense. More precisely, the vanishing of the bounded characteristic classes of $\Gamma$ yields a rigidity theorem: All the actions induced by a representation $\rho \colon \Gamma \to G$ must be elementary.\
In most classical settings, these bounded cohomology classes belong to the second bounded cohomology with trivial real coefficients $H^2_b(-; \R)$. This happens, e.g., when $X$ is one of the following non-positively curved spaces: Hyperbolic spaces \cite{qm:wpd, qm:manning, qm:wwpd}, CAT(0) spaces \cite{qm:cat0}, CAT(0) cube complexes \cite{qm:cat0cc}, Hermitian symmetric spaces \cite{kahler1, kahler2} and the circle \cite{ghys, matsumoto}. In such contexts, also weaker vanishing results suffice \cite{kotschick, fflodha}.

However, group actions on the circle are the only setting in which this really leads to a complete understanding.\ In the other aforementioned cases, there are either strong restrictions on the types of actions to consider, or the notion of an `elementary' action is quite permissive.\ In such cases, certain characteristic classes with $\ell^p$ or $L^p$ coefficients, $1 < p < \infty$, can be used to prove more powerful rigidity theorems.\ Spaces for which stronger rigidity results can be obtained include simply connected manifolds of pinched negative curvature \cite{sela, gromov:asymptotic}, certain generalized notions of negatively curved spaces \cite{mineyevmonodshalom, monodshalom1, monodshalom2} and finite-dimensional CAT(0) cube complexes \cite{CFI}.\ Up to possibly adding countability conditions on either the space or the group, the coefficients for the corresponding classes are dual and separable, and therefore we obtain rigidity results for $\Xsep$-boundedly acyclic groups.

\medskip

So far we have only talked about rigidity of representations. One can more generally consider \emph{measurable cocycles} $\sigma \colon \Gamma \times \Omega \to G$, for some probability $\Gamma$-space $\Omega$. Measurable cocycles have been central to rigidity properties for higher-rank lattices, going back to the works of Margulis and Zimmer \cite{zimmer}.
Motivated by this, several results using bounded cohomology to prove rigidity of representations, were generalized to encompass measurable cocycles \cite{cocycles1, cocycles2, cocycles3, cocycles4, cocycles5, cocycles6}. The condition that ensures the rigidity of measurable cocycles is the vanishing of a characteristic class in $H^2_b(\Gamma; L^\infty(\Omega))$.

The space $L^\infty(\Omega)$ is not separable, so we cannot directly apply $\Xsep$-bounded acyclicity. However, as was noticed by Monod and Shalom \cite[Corollary 3.5]{monodshalom2}, for all countable $\Gamma$ the map $H^2_b(\Gamma; L^\infty(\Omega)) \to H^2_b(\Gamma; L^2(\Omega))$ is injective. Since $L^2(\Omega)$ is dual and separable (because $\Omega$ is a \emph{probability} measure space), we deduce that $H^2_b(\Gamma; L^\infty(\Omega)) = 0$ for every countable $\Xsep$-boundedly acyclic group $\Gamma$, and therefore such groups enjoy the rigidity properties for measurable cocycles mentioned above.\footnote{Since the first version of this paper, our vanishing results have been extended to encompass dual \emph{semi-}separable coefficients (in the sense of Monod \cite{monod:semiseparable}), see \cite{elena}). This includes in particular the spaces $L^\infty(\Omega)$ discussed here.}

The countability assumption on $\Gamma$ is important, in that it ensures the existence of a Zimmer-amenable space that is doubly ergodic with separable dual coefficients (the Poisson boundary \cite{kaimanovich}). However, note that our proof of $\Xsep$-bounded acyclicity for groups with commuting cyclic conjugates works by expressing them as a directed union of countable $\Xsep$-boundedly acyclic groups (see the proof of Theorem \ref{thm:cac}). Therefore, using Propositions \ref{prop:dirun} and \ref{prop:dirsum} as in that case, we deduce that if $\Gamma$ has commuting cyclic conjugates, then $H^2_b(\Gamma; L^\infty(\Omega)) = 0$ for every probability $\Gamma$-space $\Omega$, even when $\Gamma$ is not countable.

\medskip

Regarding higher degrees, there are several contexts in which characteristic classes of bundles can be proved to be bounded \cite[Chapters 11-13]{Frigerio:book}; in such contexts vanishing results for higher-degree bounded cohomology imply triviality results for spaces of bundles.\ Higher dimensional bounded cohomology groups also appear naturally in the study of the simplicial volume of a manifold \cite[Chapters 7-9]{Frigerio:book}.\ Even if one is only interested in computing bounded cohomology in degree $2$, knowledge in degree $3$ is sometimes required to carry out computations for group extensions \cite[Section 12]{monod}, or to compute other invariants closely related to bounded cohomology \cite{invariantqm, asymcoho}.

Finally, as we have seen in Subsection \ref{ss:highlytrans}, $\Xsep$-bounded acyclicity can also be used as a starting point for more computations in bounded cohomology, beyond the vanishing case, which can provide rigidity results of a different nature \cite[Corollary 5.2]{fflodha}.

\bibliographystyle{amsalpha}
\bibliography{svbib}

\newcommand{\etalchar}[1]{$^{#1}$}
\providecommand{\bysame}{\leavevmode\hbox to3em{\hrulefill}\thinspace}
\providecommand{\MR}{\relax\ifhmode\unskip\space\fi MR }
\providecommand{\MRhref}[2]{%
  \href{http://www.ams.org/mathscinet-getitem?mr=#1}{#2}
}
\providecommand{\href}[2]{#2}
\begin{thebibliography}{JMBMdlS18}

\bibitem[And22]{konstantin}
K.~Andritsch, \emph{Bounded {C}ohomology of {G}roups acting on {C}antor sets},
  arXiv preprint arXiv:2210.00459, 2022.

\bibitem[AV20]{aramayona2020big}
J.~Aramayona and N.~G. Vlamis, \emph{Big mapping class groups: an overview}, In
  the tradition of {T}hurston---geometry and topology, Springer, Cham, [2020]
  \copyright 2020, pp.~459--496. \MR{4264585}

\bibitem[Ban74]{banyaga:moser}
A.~Banyaga, \emph{Formes-volume sur les vari\'{e}t\'{e}s \`a bord}, Enseign.
  Math. (2) \textbf{20} (1974), 127--131. \MR{358649}

\bibitem[Ban97]{banyaga}
\bysame, \emph{The structure of classical diffeomorphism groups}, Mathematics
  and its Applications, vol. 400, Kluwer Academic Publishers Group, Dordrecht,
  1997. \MR{1445290}

\bibitem[Bav91]{bavard}
C.~Bavard, \emph{Longueur stable des commutateurs}, Enseign. Math. (2)
  \textbf{37} (1991), no.~1-2, 109--150. \MR{1115747}

\bibitem[BBM21]{subgroups3}
C.~Bleak, M.~G. Brin, and J.~T. Moore, \emph{Complexity among the finitely
  generated subgroups of {T}hompson's group}, J. Comb. Algebra \textbf{5}
  (2021), no.~1, 1--58. \MR{4241460}

\bibitem[BC16]{cremona:ah2}
J.~Blanc and S.~Cantat, \emph{Dynamical degrees of birational transformations
  of projective surfaces}, J. Amer. Math. Soc. \textbf{29} (2016), no.~2,
  415--471. \MR{3454379}

\bibitem[BDH80]{BDH}
G.~Baumslag, E.~Dyer, and A.~Heller, \emph{The topology of discrete groups}, J.
  Pure Appl. Algebra \textbf{16} (1980), no.~1, 1--47. \MR{549702}

\bibitem[Ber89]{berrick}
A.~J. Berrick, \emph{Universal groups, binate groups and acyclicity}, Group
  theory ({S}ingapore, 1987), de Gruyter, Berlin, 1989, pp.~253--266.
  \MR{981847}

\bibitem[BF02]{qm:wpd}
M.~Bestvina and K.~Fujiwara, \emph{Bounded cohomology of subgroups of mapping
  class groups}, Geom. Topol. \textbf{6} (2002), 69--89. \MR{1914565}

\bibitem[BF09]{qm:cat0}
\bysame, \emph{A characterization of higher rank symmetric spaces via bounded
  cohomology}, Geom. Funct. Anal. \textbf{19} (2009), no.~1, 11--40.
  \MR{2507218}

\bibitem[BG84]{cohoF}
K.~S. Brown and R.~Geoghegan, \emph{An infinite-dimensional torsion-free {${\rm
  FP}\sb{\infty }$} group}, Invent. Math. \textbf{77} (1984), no.~2, 367--381.
  \MR{752825}

\bibitem[BG92]{bargeghys}
J.~Barge and E.~Ghys, \emph{Cocycles d'{E}uler et de {M}aslov}, Math. Ann.
  \textbf{294} (1992), no.~2, 235--265. \MR{1183404}

\bibitem[BHS18]{hamhomeo1}
L.~Buhovsky, V.~Humili\`ere, and S.~Seyfaddini, \emph{A {$C^0$} counterexample
  to the {A}rnold conjecture}, Invent. Math. \textbf{213} (2018), no.~2,
  759--809. \MR{3827210}

\bibitem[BI04]{kahler1}
M.~Burger and A.~Iozzi, \emph{Bounded {K}\"{a}hler class rigidity of actions on
  {H}ermitian symmetric spaces}, Ann. Sci. \'{E}cole Norm. Sup. (4) \textbf{37}
  (2004), no.~1, 77--103. \MR{2050206}

\bibitem[BIP08]{burago2008conjugation}
D.~Burago, S.~Ivanov, and L.~Polterovich, \emph{Conjugation-invariant norms on
  groups of geometric origin}, Groups of diffeomorphisms, Adv. Stud. Pure
  Math., vol.~52, Math. Soc. Japan, Tokyo, 2008, pp.~221--250. \MR{2509711}

\bibitem[BIW07]{kahler2}
M.~Burger, A.~Iozzi, and A.~Wienhard, \emph{Hermitian symmetric spaces and
  {K}\"{a}hler rigidity}, Transform. Groups \textbf{12} (2007), no.~1, 5--32.
  \MR{2308026}

\bibitem[BK15]{concordance}
M.~Brandenbursky and J.~K{\c e}dra, \emph{Concordance group and stable
  commutator length in braid groups}, Algebr. Geom. Topol. \textbf{15} (2015),
  no.~5, 2861--2886. \MR{3426696}

\bibitem[BKM13]{baykur2013sections}
R.~I. Baykur, M.~Korkmaz, and N.~Monden, \emph{Sections of surface bundles and
  {L}efschetz fibrations}, Trans. Amer. Math. Soc. \textbf{365} (2013), no.~11,
  5999--6016. \MR{3091273}

\bibitem[BM19]{brandemarci2}
M.~Brandenbursky and M.~Marcinkowski, \emph{Entropy and quasimorphisms}, J.
  Mod. Dyn. \textbf{15} (2019), 143--163. \MR{3983639}

\bibitem[BM22]{brandemarci1}
\bysame, \emph{Bounded cohomology of transformation groups}, Math. Ann.
  \textbf{382} (2022), no.~3-4, 1181--1197. \MR{4403221}

\bibitem[BM25]{brandemarci3}
\bysame, \emph{Volume and {E}uler classes in bounded cohomology of
  transformation groups}, Glasg. Math. J. \textbf{67} (2025), no.~1, 34--49.
  \MR{4857505}

\bibitem[BNR22]{ttau}
J.~Burillo, B.~Nucinkis, and L.~Reeves, \emph{Irrational-slope versions of
  {T}hompson's groups {$T$} and {$V$}}, Proc. Edinb. Math. Soc. (2) \textbf{65}
  (2022), no.~1, 244--262. \MR{4393371}

\bibitem[Bog25]{elena}
E.~Bogliolo, \emph{Bounded cohomology and scl of verbal wreath products}, arXiv
  preprint arXiv:2501.07438, 2025.

\bibitem[Bou01]{Bouarich:exact}
A.~Bouarich, \emph{Exactitude \`a gauche du foncteur {$H^n_b(-,\mathbb R)$} de
  cohomologie born\'{e}e r\'{e}elle}, Ann. Fac. Sci. Toulouse Math. (6)
  \textbf{10} (2001), no.~2, 255--270. \MR{1896181}

\bibitem[Bow11]{bowden}
J.~Bowden, \emph{On closed leaves of foliations, multisections and stable
  commutator lengths}, J. Topol. Anal. \textbf{3} (2011), no.~4, 491--509.
  \MR{2887673}

\bibitem[Bri07]{bV:brin}
M.~G. Brin, \emph{The algebra of strand splitting. {I}. {A} braided version of
  {T}hompson's group {$V$}}, J. Group Theory \textbf{10} (2007), no.~6,
  757--788. \MR{2364825}

\bibitem[Bro06]{brown:cohoF}
K.~S. Brown, \emph{The homology of {R}ichard {T}hompson's group {$F$}},
  Topological and asymptotic aspects of group theory, Contemp. Math., vol. 394,
  Amer. Math. Soc., Providence, RI, 2006, pp.~47--59. \MR{2216705}

\bibitem[BS85]{BrinSquier}
M.~G. Brin and C.~C. Squier, \emph{Groups of piecewise linear homeomorphisms of
  the real line}, Invent. Math. \textbf{79} (1985), no.~3, 485--498.
  \MR{782231}

\bibitem[Cal07]{scl_pl}
D.~Calegari, \emph{Stable commutator length in subgroups of {${\rm PL}^+(I)$}},
  Pacific J. Math. \textbf{232} (2007), no.~2, 257--262. \MR{2366352}

\bibitem[Cal09]{calegari_scl}
\bysame, \emph{scl}, MSJ Memoirs, vol.~20, Mathematical Society of Japan,
  Tokyo, 2009. \MR{2527432}

\bibitem[Cal10]{calegari:law}
\bysame, \emph{Quasimorphisms and laws}, Algebr. Geom. Topol. \textbf{10}
  (2010), no.~1, 215--217. \MR{2580432}

\bibitem[Can11]{cremona:hyperbolic}
S.~Cantat, \emph{Sur les groupes de transformations birationnelles des
  surfaces}, Ann. of Math. (2) \textbf{174} (2011), no.~1, 299--340.
  \MR{2811600}

\bibitem[Can18]{cremona:survey}
\bysame, \emph{The {C}remona group}, Algebraic geometry: {S}alt {L}ake {C}ity
  2015, Proc. Sympos. Pure Math., vol.~97, Amer. Math. Soc., Providence, RI,
  2018, pp.~101--142. \MR{3821147}

\bibitem[CdS03]{da2009introduction}
A.~Cannas~da Silva, \emph{Introduction to symplectic and {H}amiltonian
  geometry}, Publica\c{c}\~{o}es Matem\'{a}ticas do IMPA. [IMPA Mathematical
  Publications], Instituto de Matem\'{a}tica Pura e Aplicada (IMPA), Rio de
  Janeiro, 2003. \MR{2115646}

\bibitem[CFFLM25]{companion}
C.~Campagnolo, F.~Fournier-Facio, Y.~Lodha, and M.~Moraschini,
  \emph{Displacement techniques in bounded cohomology}, Manuscripta Math.
  \textbf{176} (2025), no.~2, Paper No. 19, 26. \MR{4866607}

\bibitem[CFI16]{CFI}
I.~Chatterji, T.~Fern\'{o}s, and A.~Iozzi, \emph{The median class and
  superrigidity of actions on {$\rm CAT(0)$} cube complexes}, J. Topol.
  \textbf{9} (2016), no.~2, 349--400, With an appendix by Pierre-Emmanuel
  Caprace. \MR{3509968}

\bibitem[CFP96]{cfp_96}
J.~W. Cannon, W.~J. Floyd, and W.~R. Parry, \emph{Introductory notes on
  {R}ichard {T}hompson's groups}, Enseign. Math. (2) \textbf{42} (1996),
  no.~3-4, 215--256. \MR{1426438}

\bibitem[CGHM{\etalchar{+}}22]{qm:sphere}
D.~Cristofaro-Gardiner, V.~Humili\`ere, C.~Y. Mak, S.~Seyfaddini, and I.~Smith,
  \emph{Quantitative {H}eegaard {F}loer cohomology and the {C}alabi invariant},
  Forum Math. Pi \textbf{10} (2022), Paper No. e27, 59. \MR{4524364}

\bibitem[CGHS24]{hamhomeo2}
D.~Cristofaro-Gardiner, V.~Humili\`ere, and S.~Seyfaddini, \emph{Proof of the
  simplicity conjecture}, Ann. of Math. (2) \textbf{199} (2024), no.~1,
  181--257. \MR{4681145}

\bibitem[Cha95]{random:hyperbolic2}
C.~Champetier, \emph{Propri\'{e}t\'{e}s statistiques des groupes de
  pr\'{e}sentation finie}, Adv. Math. \textbf{116} (1995), no.~2, 197--262.
  \MR{1363765}

\bibitem[CL13]{cremona:ah1}
S.~Cantat and S.~Lamy, \emph{Normal subgroups in the {C}remona group}, Acta
  Math. \textbf{210} (2013), no.~1, 31--94, With an appendix by Y. de
  Cornulier. \MR{3037611}

\bibitem[CL22]{RET}
Y.~Cornulier and O.~Lacourte, \emph{On groups of rectangle exchange
  transformations}, arXiv preprint arXiv:2209.02300, 2022.

\bibitem[Cle00]{cleary}
S.~Cleary, \emph{Regular subdivision in {$\mathbf Z[\frac{1+\sqrt 5}{2}]$}},
  Illinois J. Math. \textbf{44} (2000), no.~3, 453--464. \MR{1772420}

\bibitem[CLM22]{amcat}
P.~Capovilla, C.~L{\"o}h, and M.~Moraschini, \emph{Amenable category and
  complexity}, Algebr. Geom. Topol. \textbf{22} (2022), no.~3, 1417--1459.
  \MR{4474790}

\bibitem[CM18]{contact:portable:examples}
S.~Courte and P.~Massot, \emph{Contactomorphism groups and {L}egendrian
  flexibility}, arXiv preprint arXiv:1803.07997, 2018.

\bibitem[Con90]{conway}
J.~B. Conway, \emph{A course in functional analysis}, second ed., Graduate
  Texts in Mathematics, vol.~96, Springer-Verlag, New York, 1990. \MR{1070713}

\bibitem[dC14]{iet:cornulier}
Y.~de~Cornulier, \emph{Groupes pleins-topologiques (d'apr\`es {M}atui,
  {J}uschenko, {M}onod, {$\ldots$})}, Ast\'{e}risque (2014), no.~361, Exp. No.
  1064, viii, 183--223. \MR{3289281}

\bibitem[Deh06]{bV:dehornoy}
P.~Dehornoy, \emph{The group of parenthesized braids}, Adv. Math. \textbf{205}
  (2006), no.~2, 354--409. \MR{2258261}

\bibitem[DFG13]{iet:free}
F.~Dahmani, K.~Fujiwara, and V.~Guirardel, \emph{Free groups of interval
  exchange transformations are rare}, Groups Geom. Dyn. \textbf{7} (2013),
  no.~4, 883--910. \MR{3134029}

\bibitem[DFG20]{iet:solvable}
\bysame, \emph{Solvable groups of interval exchange transformations}, Ann. Fac.
  Sci. Toulouse Math. (6) \textbf{29} (2020), no.~3, 595--618. \MR{4196698}

\bibitem[DGO17]{DGO}
F.~Dahmani, V.~Guirardel, and D.~Osin, \emph{Hyperbolically embedded subgroups
  and rotating families in groups acting on hyperbolic spaces}, Mem. Amer.
  Math. Soc. \textbf{245} (2017), no.~1156, v+152. \MR{3589159}

\bibitem[DlCM23]{cdcm}
C.~De~la Cruz~Mengual, \emph{The {D}egree-{T}hree {B}ounded {C}ohomology of
  {C}omplex {L}ie {G}roups of {C}lassical {T}ype}, arXiv preprint
  arXiv:2304.00607, 2023.

\bibitem[DlCMH23]{delacruzH2}
C.~De~la Cruz~Mengual and T.~Hartnick, \emph{{S}tabilization of {B}ounded
  {C}ohomology for {C}lassical {G}roups}, arXiv preprint arXiv:2201.03879,
  2023.

\bibitem[DlCMH25]{delacruzH1}
\bysame, \emph{A {Q}uillen stability criterion for bounded cohomology}, Algebr.
  Geom. Topol. \textbf{25} (2025), no.~4, 2317--2341. \MR{4950907}

\bibitem[dlHM83]{delaHarpe_Mcduff_1983}
P.~de~la Harpe and D.~McDuff, \emph{Acyclic groups of automorphisms}, Comment.
  Math. Helv. \textbf{58} (1983), no.~1, 48--71. \MR{699006}

\bibitem[EF97]{epstein:fujiwara}
D.~B.~A. Epstein and K.~Fujiwara, \emph{The second bounded cohomology of
  word-hyperbolic groups}, Topology \textbf{36} (1997), no.~6, 1275--1289.
  \MR{1452851}

\bibitem[EPP12]{polterovich:qm2}
M.~Entov, L.~Polterovich, and P.~Py, \emph{On continuity of quasimorphisms for
  symplectic maps}, Perspectives in analysis, geometry, and topology, Progr.
  Math., vol. 296, Birkh\"{a}user/Springer, New York, 2012, With an appendix by
  Michael Khanevsky, pp.~169--197. \MR{2884036}

\bibitem[ES06]{elekszabo}
G.~Elek and E.~Szab\'{o}, \emph{On sofic groups}, J. Group Theory \textbf{9}
  (2006), no.~2, 161--171. \MR{2220572}

\bibitem[FFL23]{fflodha}
F.~Fournier-Facio and Y.~Lodha, \emph{Second bounded cohomology of groups
  acting on 1-manifolds and applications to spectrum problems}, Adv. Math.
  \textbf{428} (2023), Paper No. 109162, 42. \MR{4604795}

\bibitem[FFLM23]{fflm2}
F.~Fournier-Facio, C.~L{\"o}h, and M.~Moraschini, \emph{Bounded cohomology and
  binate groups}, J. Aust. Math. Soc. \textbf{115} (2023), no.~2, 204--239.
  \MR{4640119}

\bibitem[FFLM24]{fflm1}
\bysame, \emph{Bounded cohomology of finitely presented groups: vanishing,
  non-vanishing, and computability}, Ann. Sc. Norm. Super. Pisa Cl. Sci. (5)
  \textbf{25} (2024), no.~2, 1169--1202. \MR{4778473}

\bibitem[FFLZ24a]{bV:qm}
F.~Fournier-Facio, Y.~Lodha, and M.~C.~B. Zaremsky, \emph{Braided {T}hompson
  groups with and without quasimorphisms}, Algebr. Geom. Topol. \textbf{24}
  (2024), no.~3, 1601--1622. \MR{4767881}

\bibitem[FFLZ24b]{monsters}
\bysame, \emph{Finitely presented left orderable monsters}, Ergodic Theory
  Dynam. Systems \textbf{44} (2024), no.~5, 1367--1378. \MR{4727563}

\bibitem[FFMN24]{ffmn}
F.~Fournier-Facio, N.~Monod, and S.~Nariman, \emph{The bounded cohomology of
  transformation groups of {E}uclidean spaces and discs}, Appendix by A.
  Kupers. arXiv preprint arXiv:2405.20395, 2024.

\bibitem[FFR24]{thompson:ulam}
F.~Fournier-Facio and B.~Rangarajan, \emph{Ulam stability of lamplighters and
  {T}hompson groups}, Math. Ann. \textbf{389} (2024), no.~3, 2469--2497.
  \MR{4753068}

\bibitem[FFT19]{qm:cat0cc}
T.~Fern\'{o}s, M.~Forester, and J.~Tao, \emph{Effective quasimorphisms on
  right-angled {A}rtin groups}, Ann. Inst. Fourier (Grenoble) \textbf{69}
  (2019), no.~4, 1575--1626. \MR{4010865}

\bibitem[FFW23]{autinv}
F.~Fournier-Facio and R.~D. Wade, \emph{{\rm {A}ut}-invariant quasimorphisms on
  groups}, Trans. Amer. Math. Soc. \textbf{376} (2023), no.~10, 7307--7327.
  \MR{4636691}

\bibitem[Fis60]{fisher1960group}
G.~M. Fisher, \emph{On the group of all homeomorphisms of a manifold}, Trans.
  Amer. Math. Soc. \textbf{97} (1960), 193--212. \MR{117712}

\bibitem[FK08]{funar2008braided}
L.~Funar and C.~Kapoudjian, \emph{The braided {P}tolemy-{T}hompson group is
  finitely presented}, Geom. Topol. \textbf{12} (2008), no.~1, 475--530.
  \MR{2390352}

\bibitem[FKS12]{funarmore}
L.~Funar, C.~Kapoudjian, and V.~Sergiescu, \emph{Asymptotically rigid mapping
  class groups and {T}hompson's groups}, Handbook of {T}eichm\"{u}ller theory.
  {V}olume {III}, IRMA Lect. Math. Theor. Phys., vol.~17, Eur. Math. Soc.,
  Z\"{u}rich, 2012, pp.~595--664. \MR{2952772}

\bibitem[FM12]{primer}
B.~Farb and D.~Margalit, \emph{A primer on mapping class groups}, Princeton
  Mathematical Series, vol.~49, Princeton University Press, Princeton, NJ,
  2012. \MR{2850125}

\bibitem[FPR18]{contact:portable}
M.~Fraser, L.~Polterovich, and D.~Rosen, \emph{On {S}andon-type metrics for
  contactomorphism groups}, Ann. Math. Qu\'{e}. \textbf{42} (2018), no.~2,
  191--214. \MR{3858469}

\bibitem[FPS15]{FPS}
R.~Frigerio, M.~B. Pozzetti, and A.~Sisto, \emph{Extending higher-dimensional
  quasi-cocycles}, J. Topol. \textbf{8} (2015), no.~4, 1123--1155. \MR{3431671}

\bibitem[Fri17]{Frigerio:book}
R.~Frigerio, \emph{Bounded cohomology of discrete groups}, Mathematical Surveys
  and Monographs, vol. 227, American Mathematical Society, Providence, RI,
  2017. \MR{3726870}

\bibitem[Ger73]{gersten}
S.~M. Gersten, \emph{{$K\sb{3}$} of a ring is {$H\sb{3}$} of the {S}teinberg
  group}, Proc. Amer. Math. Soc. \textbf{37} (1973), 366--368. \MR{320114}

\bibitem[GG04]{gambaudoghys}
J.-M. Gambaudo and E.~Ghys, \emph{Commutators and diffeomorphisms of surfaces},
  Ergodic Theory Dynam. Systems \textbf{24} (2004), no.~5, 1591--1617.
  \MR{2104597}

\bibitem[Ghy01]{ghys}
E.~Ghys, \emph{Groups acting on the circle}, Enseign. Math. (2) \textbf{47}
  (2001), no.~3-4, 329--407. \MR{1876932}

\bibitem[GKEL23]{GKL}
I.~Goldbring, S.~Kunnawalkam~Elayavalli, and Y.~Lodha, \emph{Generic algebraic
  properties in spaces of enumerated groups}, Trans. Amer. Math. Soc.
  \textbf{376} (2023), no.~9, 6245--6282. \MR{4630775}

\bibitem[GL69]{braid:perfect}
E.~A. Gorin and V.~Ja. Lin, \emph{Algebraic equations with continuous
  coefficients, and certain questions of the algebraic theory of braids}, Mat.
  Sb. (N.S.) \textbf{78(120)} (1969), 579--610. \MR{251712}

\bibitem[GL22]{iet:up}
N.~Guelman and I.~Liousse, \emph{Uniform perfectness for interval exchange
  transformations with or without flips}, Ann. Inst. Fourier (Grenoble)
  \textbf{72} (2022), no.~4, 1477--1501. \MR{4485831}

\bibitem[GLMR23]{asymcoho}
L.~Glebsky, A.~Lubotzky, N.~Monod, and B.~Rangarajan, \emph{Asymptotic
  cohomology and uniform stability for lattices in semisimple groups},
  arXiv:2301.00476. To appear in Mem. Eur. Math. Soc., 2023.

\bibitem[GLU22]{funarlike}
A.~Genevois, A.~Lonjou, and C.~Urech, \emph{Asymptotically rigid mapping class
  groups, {I}: {F}initeness properties of braided {T}hompson's and {H}oughton's
  groups}, Geom. Topol. \textbf{26} (2022), no.~3, 1385--1434. \MR{4466651}

\bibitem[Gre89]{greenberg}
P.~Greenberg, \emph{Homology of some groups of pl-homeomorphisms of the line},
  Topology Appl. \textbf{33} (1989), no.~3, 281--295. \MR{1026929}

\bibitem[Gri95]{grigorchuk}
R.~I. Grigorchuk, \emph{Some results on bounded cohomology}, Combinatorial and
  geometric group theory ({E}dinburgh, 1993), London Math. Soc. Lecture Note
  Ser., vol. 204, Cambridge Univ. Press, Cambridge, 1995, pp.~111--163.
  \MR{1320279}

\bibitem[Gro82]{vbc}
M.~Gromov, \emph{Volume and bounded cohomology}, Inst. Hautes \'{E}tudes Sci.
  Publ. Math. (1982), no.~56, 5--99 (1983). \MR{686042}

\bibitem[Gro93]{gromov:asymptotic}
\bysame, \emph{Asymptotic invariants of infinite groups}, Geometric group
  theory, {V}ol. 2 ({S}ussex, 1991), London Math. Soc. Lecture Note Ser., vol.
  182, Cambridge Univ. Press, Cambridge, 1993, pp.~1--295. \MR{1253544}

\bibitem[GS87]{cohoT}
E.~Ghys and V.~Sergiescu, \emph{Sur un groupe remarquable de
  diff\'{e}omorphismes du cercle}, Comment. Math. Helv. \textbf{62} (1987),
  no.~2, 185--239. \MR{896095}

\bibitem[GS99]{subgroups1}
V.~S. Guba and M.~V. Sapir, \emph{On subgroups of the {R}. {T}hompson group
  {$F$} and other diagram groups}, Mat. Sb. \textbf{190} (1999), no.~8, 3--60.
  \MR{1725439}

\bibitem[GS17]{subgroups2}
G.~Golan and M.~Sapir, \emph{On subgroups of {R}. {T}hompson's group {$F$}},
  Trans. Amer. Math. Soc. \textbf{369} (2017), no.~12, 8857--8878. \MR{3710646}

\bibitem[HL19]{grho}
J.~Hyde and Y.~Lodha, \emph{Finitely generated infinite simple groups of
  homeomorphisms of the real line}, Invent. Math. \textbf{218} (2019), no.~1,
  83--112. \MR{3994586}

\bibitem[HM21]{qm:wwpd}
M.~Handel and L.~Mosher, \emph{Second bounded cohomology and {WWPD}}, Kyoto J.
  Math. \textbf{61} (2021), no.~4, 873--904. \MR{4415399}

\bibitem[HQR22]{bigmcg:qm}
C.~Horbez, Y.~Qing, and K.~Rafi, \emph{Big mapping class groups with hyperbolic
  actions: classification and applications}, J. Inst. Math. Jussieu \textbf{21}
  (2022), no.~6, 2173--2204. \MR{4515292}

\bibitem[HS18]{HayesSale}
B.~Hayes and A.~W. Sale, \emph{Metric approximations of wreath products}, Ann.
  Inst. Fourier (Grenoble) \textbf{68} (2018), no.~1, 423--455. \MR{3795485}

\bibitem[Hur83]{hurder}
S.~Hurder, \emph{Global invariants for measured foliations}, Trans. Amer. Math.
  Soc. \textbf{280} (1983), no.~1, 367--391. \MR{712266}

\bibitem[JMBMdlS18]{extensive}
K.~Juschenko, N.~Matte~Bon, N.~Monod, and M.~de~la Salle, \emph{Extensive
  amenability and an application to interval exchanges}, Ergodic Theory Dynam.
  Systems \textbf{38} (2018), no.~1, 195--219. \MR{3742543}

\bibitem[Joh72]{Johnson}
B.~E. Johnson, \emph{Cohomology in {B}anach algebras}, Memoirs of the American
  Mathematical Society, vol. No. 127, American Mathematical Society,
  Providence, RI, 1972. \MR{374934}

\bibitem[Kai03]{kaimanovich}
V.~A. Kaimanovich, \emph{Double ergodicity of the {P}oisson boundary and
  applications to bounded cohomology}, Geom. Funct. Anal. \textbf{13} (2003),
  no.~4, 852--861. \MR{2006560}

\bibitem[Kas25]{kastenholz}
T.~Kastenholz, \emph{Symplectic groups, mapping class groups and the stability
  of bounded cohomology}, Topology Appl. \textbf{373} (2025), Paper No. 109531.
  \MR{4945077}

\bibitem[Kim16]{kimT}
H.~K. Kim, \emph{The dilogarithmic central extension of the
  {P}tolemy-{T}hompson group via the {K}ashaev quantization}, Adv. Math.
  \textbf{293} (2016), 529--588. \MR{3474328}

\bibitem[Kim18]{kimura}
M.~Kimura, \emph{Conjugation-invariant norms on the commutator subgroup of the
  infinite braid group}, J. Topol. Anal. \textbf{10} (2018), no.~2, 471--476.
  \MR{3809596}

\bibitem[Kim25]{kimura_2023}
\bysame, \emph{Gambaudo-{G}hys construction on bounded cohomology}, J. Math.
  Soc. Japan \textbf{77} (2025), no.~1, 135--152. \MR{4854771}

\bibitem[KKL19]{chain}
S.-H. Kim, T.~Koberda, and Y.~Lodha, \emph{Chain groups of homeomorphisms of
  the interval}, Ann. Sci. \'{E}c. Norm. Sup\'{e}r. (4) \textbf{52} (2019),
  no.~4, 797--820. \MR{4038452}

\bibitem[KKM{\etalchar{+}}24]{invariantqm}
M.~Kawasaki, M.~Kimura, S.~Maruyama, T.~Matsushita, and M.~Mimura, \emph{Survey
  on invariant quasimorphisms and stable mixed commutator length}, Topology
  Proc. \textbf{64} (2024), 129--174. \MR{4773366}

\bibitem[Kot08]{kotschick}
D.~Kotschick, \emph{Stable length in stable groups}, Groups of diffeomorphisms,
  Adv. Stud. Pure Math., vol.~52, Math. Soc. Japan, Tokyo, 2008, pp.~401--413.
  \MR{2509718}

\bibitem[KS23]{kastenholz-sroka}
T.~Kastenholz and R.~J. Sroka, \emph{Simplicial bounded cohomology and
  stability}, arxiv preprint arXiv:2309.05024, 2023.

\bibitem[KT08]{kechris:tsankov}
A.~S. Kechris and T.~Tsankov, \emph{Amenable actions and almost invariant
  sets}, Proc. Amer. Math. Soc. \textbf{136} (2008), no.~2, 687--697.
  \MR{2358510}

\bibitem[LBMB22]{tphisigma}
A.~Le~Boudec and N.~Matte~Bon, \emph{Confined subgroups and high transitivity},
  Ann. H. Lebesgue \textbf{5} (2022), 491--522. \MR{4443296}

\bibitem[LLM22]{liloehmoraschini}
K.~Li, C.~L{\"o}h, and M.~Moraschini, \emph{Bounded acyclicity and relative
  simplicial volume}, arXiv:2202.05606. To appear in J. Topol. Anal., 2022.

\bibitem[LM16]{LodhaMoore}
Y.~Lodha and J.~T. Moore, \emph{A nonamenable finitely presented group of
  piecewise projective homeomorphisms}, Groups Geom. Dyn. \textbf{10} (2016),
  no.~1, 177--200. \MR{3460335}

\bibitem[Lod19]{lodhaS}
Y.~Lodha, \emph{A finitely presented infinite simple group of homeomorphisms of
  the circle}, J. Lond. Math. Soc. (2) \textbf{100} (2019), no.~3, 1034--1064.
  \MR{4048731}

\bibitem[Lod20a]{coherent}
\bysame, \emph{Coherent actions by homeomorphisms on the real line or an
  interval}, Israel J. Math. \textbf{235} (2020), no.~1, 183--212. \MR{4068782}

\bibitem[Lod20b]{LodhaFinf}
\bysame, \emph{A nonamenable type {$\rm F_\infty$} group of piecewise
  projective homeomorphisms}, J. Topol. \textbf{13} (2020), no.~4, 1767--1838.
  \MR{4186144}

\bibitem[L{\"o}h10]{clara:book}
C.~L{\"o}h, \emph{Group cohomology and bounded cohomology}, {A}vailable online
  at
  \url{https://loeh.app.uni-regensburg.de/teaching/topologie3_ws0910/prelim.pdf},
  2010.

\bibitem[L{\"o}h17]{Loeh:dim}
C.~L{\"o}h, \emph{A note on bounded-cohomological dimension of discrete
  groups}, J. Math. Soc. Japan \textbf{69} (2017), no.~2, 715--734.
  \MR{3638282}

\bibitem[LU21]{cremona:cat}
A.~Lonjou and C.~Urech, \emph{Actions of {C}remona groups on {${\rm CAT}(0)$}
  cube complexes}, Duke Math. J. \textbf{170} (2021), no.~17, 3703--3743.
  \MR{4340723}

\bibitem[Mac72]{algclosed:mac}
A.~Macintyre, \emph{On algebraically closed groups}, Ann. of Math. (2)
  \textbf{96} (1972), 53--97. \MR{317928}

\bibitem[Man05]{manning:geometry}
J.~F. Manning, \emph{Geometry of pseudocharacters}, Geom. Topol. \textbf{9}
  (2005), 1147--1185. \MR{2174263}

\bibitem[Man08]{qm:manning}
\bysame, \emph{Actions of certain arithmetic groups on {G}romov hyperbolic
  spaces}, Algebr. Geom. Topol. \textbf{8} (2008), no.~3, 1371--1402.
  \MR{2443247}

\bibitem[Mat71]{mather1971vanishing}
J.~N. Mather, \emph{The vanishing of the homology of certain groups of
  homeomorphisms}, Topology \textbf{10} (1971), 297--298. \MR{288777}

\bibitem[Mat86]{matsumoto}
S.~Matsumoto, \emph{Numerical invariants for semiconjugacy of homeomorphisms of
  the circle}, Proc. Amer. Math. Soc. \textbf{98} (1986), no.~1, 163--168.
  \MR{848896}

\bibitem[MBT20]{tphi}
N.~Matte~Bon and M.~Triestino, \emph{Groups of piecewise linear homeomorphisms
  of flows}, Compos. Math. \textbf{156} (2020), no.~8, 1595--1622. \MR{4157429}

\bibitem[McD82]{mcduff:local1}
D.~McDuff, \emph{Local homology of groups of volume preserving diffeomorphisms.
  {I}}, Ann. Sci. \'{E}cole Norm. Sup. (4) \textbf{15} (1982), no.~4, 609--648
  (1983). \MR{707329}

\bibitem[McD83a]{mcduff:local2}
\bysame, \emph{Local homology of groups of volume-preserving diffeomorphisms.
  {II}}, Comment. Math. Helv. \textbf{58} (1983), no.~1, 135--165. \MR{699012}

\bibitem[McD83b]{mcduff:local3}
\bysame, \emph{Local homology of groups of volume-preserving diffeomorphisms.
  {III}}, Ann. Sci. \'{E}cole Norm. Sup. (4) \textbf{16} (1983), no.~4,
  529--540 (1984). \MR{740589}

\bibitem[McD83c]{mcduff:canonical}
\bysame, \emph{Some canonical cohomology classes on groups of volume preserving
  diffeomorphisms}, Trans. Amer. Math. Soc. \textbf{275} (1983), no.~1,
  345--356. \MR{678355}

\bibitem[Mil71]{milnor1971introduction}
J.~Milnor, \emph{Introduction to algebraic {$K$}-theory}, Annals of Mathematics
  Studies, No. 72, Princeton University Press, Princeton, NJ; University of
  Tokyo Press, Tokyo, 1971. \MR{349811}

\bibitem[MM85]{Matsu-Mor}
S.~Matsumoto and S.~Morita, \emph{Bounded cohomology of certain groups of
  homeomorphisms}, Proc. Amer. Math. Soc. \textbf{94} (1985), no.~3, 539--544.
  \MR{787909}

\bibitem[MMS04]{mineyevmonodshalom}
I.~Mineyev, N.~Monod, and Y.~Shalom, \emph{Ideal bicombings for hyperbolic
  groups and applications}, Topology \textbf{43} (2004), no.~6, 1319--1344.
  \MR{2081428}

\bibitem[MN23]{monodnariman}
N.~Monod and S.~Nariman, \emph{Bounded and unbounded cohomology of
  homeomorphism and diffeomorphism groups}, Invent. Math. \textbf{232} (2023),
  no.~3, 1439--1475. \MR{4588567}

\bibitem[Mon01]{monod}
N.~Monod, \emph{Continuous bounded cohomology of locally compact groups},
  Lecture Notes in Mathematics, vol. 1758, Springer-Verlag, Berlin, 2001.
  \MR{1840942}

\bibitem[Mon04]{monod04}
\bysame, \emph{Stabilization for {${\rm SL}_n$} in bounded cohomology},
  Discrete geometric analysis, Contemp. Math., vol. 347, Amer. Math. Soc.,
  Providence, RI, 2004, pp.~191--202. \MR{2077038}

\bibitem[Mon10]{monod:semiseparable}
\bysame, \emph{On the bounded cohomology of semi-simple groups,
  {$S$}-arithmetic groups and products}, J. Reine Angew. Math. \textbf{640}
  (2010), 167--202. \MR{2629693}

\bibitem[Mon13]{monod:pp1}
\bysame, \emph{Groups of piecewise projective homeomorphisms}, Proc. Natl.
  Acad. Sci. USA \textbf{110} (2013), no.~12, 4524--4527. \MR{3047655}

\bibitem[Mon22]{monod:thompson}
\bysame, \emph{Lamplighters and the bounded cohomology of {T}hompson's group},
  Geom. Funct. Anal. \textbf{32} (2022), no.~3, 662--675. \MR{4431125}

\bibitem[Mon25]{monod:pp2}
\bysame, \emph{Some comments on piecewise-projective groups of the line},
  Groups Geom. Dyn. \textbf{19} (2025), no.~2, 459--476. \MR{4940648}

\bibitem[Mos65]{moser}
J.~Moser, \emph{On the volume elements on a manifold}, Trans. Amer. Math. Soc.
  \textbf{120} (1965), 286--294. \MR{182927}

\bibitem[MP03]{Monod:Popa}
N.~Monod and S.~Popa, \emph{On co-amenability for groups and von {N}eumann
  algebras}, C. R. Math. Acad. Sci. Soc. R. Can. \textbf{25} (2003), no.~3,
  82--87. \MR{1999183}

\bibitem[MR23]{moraschini_raptis_2023}
M.~Moraschini and G.~Raptis, \emph{Amenability and acyclicity in bounded
  cohomology}, Rev. Mat. Iberoam. \textbf{39} (2023), no.~6, 2371--2404.
  \MR{4671438}

\bibitem[MS04]{monodshalom1}
N.~Monod and Y.~Shalom, \emph{Cocycle superrigidity and bounded cohomology for
  negatively curved spaces}, J. Differential Geom. \textbf{67} (2004), no.~3,
  395--455. \MR{2153026}

\bibitem[MS06]{monodshalom2}
\bysame, \emph{Orbit equivalence rigidity and bounded cohomology}, Ann. of
  Math. (2) \textbf{164} (2006), no.~3, 825--878. \MR{2259246}

\bibitem[MS22a]{cocycles1}
M.~Moraschini and A.~Savini, \emph{A {M}atsumoto-{M}ostow result for {Z}immer's
  cocycles of hyperbolic lattices}, Transform. Groups \textbf{27} (2022),
  no.~4, 1337--1392. \MR{4507989}

\bibitem[MS22b]{cocycles2}
\bysame, \emph{Multiplicative constants and maximal measurable cocycles in
  bounded cohomology}, Ergodic Theory Dynam. Systems \textbf{42} (2022),
  no.~11, 3490--3525. \MR{4492631}

\bibitem[Mun00]{munkres}
J.~R. Munkres, \emph{Topology}, Prentice Hall, Inc., Upper Saddle River, NJ,
  2000, Second edition of [ MR0464128]. \MR{3728284}

\bibitem[MW07]{madsen2007stable}
I.~Madsen and M.~Weiss, \emph{The stable moduli space of {R}iemann surfaces:
  {M}umford's conjecture}, Ann. of Math. (2) \textbf{165} (2007), no.~3,
  843--941. \MR{2335797}

\bibitem[Nav18]{navas}
A.~Navas, \emph{Group actions on 1-manifolds: a list of very concrete open
  questions}, Proceedings of the {I}nternational {C}ongress of
  {M}athematicians---{R}io de {J}aneiro 2018. {V}ol. {III}. {I}nvited lectures,
  World Sci. Publ., Hackensack, NJ, 2018, pp.~2035--2062. \MR{3966841}

\bibitem[Neu52]{algclosed:BH}
B.~H. Neumann, \emph{A note on algebraically closed groups}, J. London Math.
  Soc. \textbf{27} (1952), 247--249. \MR{46363}

\bibitem[Nit21]{nitsche}
M.~Nitsche, \emph{Higher-degree bounded cohomology of transformation groups},
  arXiv preprint arXiv:2105.08698, 2021.

\bibitem[Ols92]{random:hyperbolic1}
A.~Yu. Olshanski, \emph{Almost every group is hyperbolic}, Internat. J. Algebra
  Comput. \textbf{2} (1992), no.~1, 1--17. \MR{1167524}

\bibitem[Pol01]{polterovich2012geometry}
L.~Polterovich, \emph{The geometry of the group of symplectic diffeomorphisms},
  Lectures in Mathematics ETH Z\"{u}rich, Birkh\"{a}user Verlag, Basel, 2001.
  \MR{1826128}

\bibitem[Pol06]{polterovich:qm}
\bysame, \emph{Floer homology, dynamics and groups}, Morse theoretic methods in
  nonlinear analysis and in symplectic topology, NATO Sci. Ser. II Math. Phys.
  Chem., vol. 217, Springer, Dordrecht, 2006, pp.~417--438. \MR{2276956}

\bibitem[Py06]{py}
P.~Py, \emph{Quasi-morphismes et invariant de {C}alabi}, Ann. Sci. \'{E}cole
  Norm. Sup. (4) \textbf{39} (2006), no.~1, 177--195. \MR{2224660}

\bibitem[Rap19]{Raptis-acyclic}
G.~Raptis, \emph{Some characterizations of acyclic maps}, J. Homotopy Relat.
  Struct. \textbf{14} (2019), no.~3, 773--785. \MR{3987558}

\bibitem[RW23]{rw:survey}
O.~Randal-Williams, \emph{Diffeomorphisms of discs}, I{CM}---{I}nternational
  {C}ongress of {M}athematicians. {V}ol. 4. {S}ections 5--8, EMS Press, Berlin,
  [2023] \copyright 2023, pp.~2856--2878. \MR{4680344}

\bibitem[RWW17]{stability:auto}
O.~Randal-Williams and N.~Wahl, \emph{Homological stability for automorphism
  groups}, Adv. Math. \textbf{318} (2017), 534--626. \MR{3689750}

\bibitem[Sav21a]{cocycles4}
A.~Savini, \emph{Algebraic hull of maximal measurable cocycles of surface
  groups into {H}ermitian {L}ie groups}, Geom. Dedicata \textbf{213} (2021),
  375--400. \MR{4278334}

\bibitem[Sav21b]{cocycles5}
\bysame, \emph{Parametrized {E}uler class and semicohomology theory}, arXiv
  preprint arXiv:2101.11971, 2021.

\bibitem[Sco51]{algclosed:scott}
W.~R. Scott, \emph{Algebraically closed groups}, Proc. Amer. Math. Soc.
  \textbf{2} (1951), 118--121. \MR{40299}

\bibitem[Sel92]{sela}
Z.~Sela, \emph{Uniform embeddings of hyperbolic groups in {H}ilbert spaces},
  Israel J. Math. \textbf{80} (1992), no.~1-2, 171--181. \MR{1248933}

\bibitem[SS22]{cocycles3}
F.~Sarti and A.~Savini, \emph{Superrigidity of maximal measurable cocycles of
  complex hyperbolic lattices}, Math. Z. \textbf{300} (2022), no.~1, 421--443.
  \MR{4359531}

\bibitem[SS23]{cocycles6}
\bysame, \emph{Parametrized {K}\"{a}hler class and {Z}ariski dense orbital
  1-cohomology}, Math. Res. Lett. \textbf{30} (2023), no.~6, 1895--1929.
  \MR{4779157}

\bibitem[Ste92]{stein}
M.~Stein, \emph{Groups of piecewise linear homeomorphisms}, Trans. Amer. Math.
  Soc. \textbf{332} (1992), no.~2, 477--514. \MR{1094555}

\bibitem[SV87]{binate:product}
P.~Sankaran and K.~Varadarajan, \emph{Some remarks on pseudo-mitotic groups},
  Indian J. Math. \textbf{29} (1987), no.~3, 283--293 (1988). \MR{971641}

\bibitem[SW21]{bV:dense}
R.~Skipper and X.~Wu, \emph{Homological stability for the ribbon
  {H}igman--{T}hompson groups}, arXiv:2106.08751. To appear in Algebr. Geom.
  Topol., 2021.

\bibitem[Szy24]{cremona}
M.~Szymik, \emph{Homological stability for the {C}remona groups}, arXiv
  preprint arXiv:2403.07546, 2024.

\bibitem[Wei13]{kbook}
C.~A. Weibel, \emph{The {$K$}-book}, Graduate Studies in Mathematics, vol. 145,
  American Mathematical Society, Providence, RI, 2013, An introduction to
  algebraic $K$-theory. \MR{3076731}

\bibitem[WM06]{witte}
D.~Witte-Morris, \emph{Amenable groups that act on the line}, Algebr. Geom.
  Topol. \textbf{6} (2006), 2509--2518. \MR{2286034}

\bibitem[Woi17]{woit2017quantum}
P.~Woit, \emph{Quantum theory, groups and representations}, Springer, Cham,
  2017, An introduction. \MR{3726869}

\bibitem[WWZZ25]{WWZZ}
F.~Wu, X.~Wu, M.~Zhao, and Z.~Zhou, \emph{Embedding groups into boundedly
  acyclic groups}, J. Lond. Math. Soc. (2) \textbf{111} (2025), no.~5, Paper
  No. e70164, 39. \MR{4899448}

\bibitem[Zho24]{zhou}
Z.~Zhou, \emph{Bounded cohomology of diffeomorphism groups of higher
  dimensional spheres}, arXiv preprint arXiv:2411.14059, 2024.

\bibitem[Zim84]{zimmer}
R.~J. Zimmer, \emph{Ergodic theory and semisimple groups}, Monographs in
  Mathematics, vol.~81, Birkh\"{a}user Verlag, Basel, 1984. \MR{776417}

\end{thebibliography}

\end{document}